\documentclass{article}
\usepackage[utf8]{inputenc}

\usepackage{amsmath}
\usepackage{quiver}

\usepackage{comment}
\usepackage{leftidx}
\usepackage{multicol}
\usepackage{amssymb}
\usepackage{stmaryrd}
\usepackage{accents}

\usepackage{xfrac}
\usepackage{amsthm}
\usepackage{xcolor}
\usepackage{sectsty}
\usepackage{relsize}

\usepackage{indentfirst}
\usepackage{tikz}
\tikzset{shorten <>/.style={shorten >=#1,shorten <=#1}}

\usetikzlibrary{matrix,arrows,decorations.pathmorphing}
\usepackage{tikz-cd}
\newcounter{nodemaker}
\tikzset{%
    symbol/.style={%
        draw=none,
        every to/.append style={%
            edge node={node [sloped, allow upside down, auto=false]{$#1$}}}
    }
}

\makeatletter
\newcommand{\bigdoublevee}{\big@doubleop{\bigvee}}
\newcommand{\bigdoublewedge}{\big@doubleop{\bigwedge}}
\newcommand{\big@doubleop}[1]{%
  \DOTSB\mathop{\mathpalette\big@doubleop@aux{#1}}\slimits@
}

\newcommand\big@doubleop@aux[2]{%
  \sbox\z@{$\m@th#1#2$}%
  \makebox[1.35\wd\z@][s]{$\m@th#1#2\hss#2$}%
}
\makeatother

\makeatletter
\newcommand*{\doublerightarrow}[2]{\mathrel{
  \settowidth{\@tempdima}{$\scriptstyle#1$}
  \settowidth{\@tempdimb}{$\scriptstyle#2$}
  \ifdim\@tempdimb>\@tempdima \@tempdima=\@tempdimb\fi
  \mathop{\vcenter{
    \offinterlineskip\ialign{\hbox to\dimexpr\@tempdima+1em{##}\cr
    \rightarrowfill\cr\noalign{\kern.5ex}
    \rightarrowfill\cr}}}\limits^{\!#1}_{\!#2}}}
\newcommand*{\triplerightarrow}[1]{\mathrel{
  \settowidth{\@tempdima}{$\scriptstyle#1$}
  \mathop{\vcenter{
    \offinterlineskip\ialign{\hbox to\dimexpr\@tempdima+1em{##}\cr
    \rightarrowfill\cr\noalign{\kern.5ex}
    \rightarrowfill\cr\noalign{\kern.5ex}
    \rightarrowfill\cr}}}\limits^{\!#1}}}
\makeatother

\newcommand{\hirayo}{\text{\usefont{U}{min}{m}{n}\symbol{'207}}}
\DeclareFontFamily{U}{min}{}
\DeclareFontShape{U}{min}{m}{n}{<-> udmj30}{}

\DeclareUnicodeCharacter{3088}{\yo}

\usepackage{geometry}
\usepackage{hyperref}
\hypersetup{
    colorlinks=true,
    linkcolor=blue,
    filecolor=magenta,      
    urlcolor=cyan,
}

\usepackage{cleveref}

\newtheorem{theorem}{Theorem}[section]
\theoremstyle{proposition}
\newtheorem{proposition}[theorem]{Proposition}
\newtheorem{corollary}[theorem]{Corollary}
\newtheorem{corollary'}[theorem]{Corollary}
\newtheorem{lemma}[theorem]{Lemma}
\theoremstyle{definition}
\newtheorem{definition}[theorem]{Definition}

\theoremstyle{remark}
\newtheorem*{remark}{Remark}

\newcommand{\cod}
 {{\rm cod}}
 
 \newcommand{\Ind}
 {{\rm Ind}}

\newcommand{\comp}
 {\circ}

\newcommand{\Cont}
 {{\bf Cont}}

\newcommand{\dom}
 {{\rm dom}}

\newcommand{\li}
{{\textup{lim } }}

\newcommand{\lan}
{{\textup{lan } }}

\newcommand{\ran}
{{\textup{ran } }}

\newcommand{\colim}
{{\textup{colim}}}

\newcommand{\comma}[2]
{\mbox{$(#1\!\downarrow\!#2)$}}

\newcommand{\empstg}
 {[\,]}

\newcommand{\epi}
 {\twoheadrightarrow}

\newcommand{\coeq}
{{\textup{coeq}}}

\newcommand{\eq}
{{\textup{eq}}}

\newcommand{\hy}
 {\mbox{-}}

\newcommand{\im}
 {{\rm im}}

\newcommand{\imp}
 {\!\Rightarrow\!}

\newcommand{\mono}
 {\rightarrowtail}

\newcommand{\ob}
 {{\rm ob}}
 
 \newcommand{\Hom}
 {{\rm Hom}}

\newcommand{\op}
 {^{\rm op}}

\newcommand{\Set}
 {{\bf Set }}

\newcommand{\Sh}
 {{\bf Sh}}

\newcommand{\sh}
 {{\bf sh}}

\newcommand{\Sub}
 {{\rm Sub}}

\sectionfont{\large}
\subsectionfont{\normalsize}

\usepackage[backend=biber,
sorting=nty
]{biblatex}
\nocite{*}

\title{On Diers's theory of Spectrum I \\
Stable functors and right multi-adjoints}
\author{Axel Osmond}
\date{October 2020}

\begin{document}

\maketitle

\begin{abstract}
    Diers developed a general theory of right multi-adjoint functors leading to a purely categorical, point-set construction of spectra. Situations of “multiversal” properties return sets of canonical solutions rather than a unique one. In the case of a right multi-adjoint, each object deploys a canonical cone of local units jointly assuming the role of the unit of an adjunction. This first part revolves around the theory of multi-adjoint and recalls or precises results that will be used later on for geometric purpose. We also study the weaker notion of local adjoint, proving Beck-Chevalley conditions relating local adjunctions and the equivalence with the notion of stable functor. We also recall the link with the free-product completion, and describe factorization aspects involved in a situation of multi-adjunction. The relation between accessible right multi-adjoints and locally finitely multipresentable categories is also revisited.  
\end{abstract}

\section*{Introduction}

This paper, together with \cite{partII}, is the first part of a twofold work on Diers construction of spectra through the notion of right multi-adjoint, and more generally, to a series of papers devoted to synthesise current approaches about the notion of spectrum and how they are related. Spectra have played a prominent role in several regions of mathematics: for instance, algebraic geometry could resume in some sense as the study of the different flavors of spectra of commutative rings, while Stone duality is about spectra of distributive lattices and ordered structure; more loosely, categorical model theory is in some sense the study of ``spectra of theories", for a still-to-define notion of 2-dimensional spectra. We could sum up the central philosophy behind this notion though the following claim: \emph{spectra arise when free construction fails}. A very broad overview of the situation is the following: one starts with a category of algebraic ``ambient" objects, and a class of objects and maps between them one wan to see as ``local data", but fails to associate canonically one local object under an ambient object: one ends up rather with a family of canonical local objects under an ambient object, which is universal in some sense. The spectrum of an ambient object is a space whose points index this canonical family under it, equipped with a structural sheaf whose purpose is to gather those local objects, and it defines a left adjoint to a comparison functor between categories of ``structured spaces".\\

Until now, several and rather independent proposals to construct spectra in a general way have been done. Contrasting to the topos-theoretic approaches of \cite{Cole}, \cite{Coste}, \cite{Anel}, or \cite{lurie2009derived}, of which we will also provide a synthesis in \cite{survey}, Diers approach is more purely categorical in its premises, and strictly point-set in the way it processes to the construction. However, both notion of spectra follows a similar ``scenario": a first step identifying algebraic situations with a hidden geometric content, and a second step where this geometric content is used to construct a corresponding notion of spectrum. In the topos theoretic approach, the first step could be synthesized as revolving around what was variously called \emph{admissibility} in \cite{Cole}, \emph{geometry} in \cite{lurie2009derived}, \emph{triples} in \cite{Coste}, which was stated in terms suited for topos constructions; in Diers approach, the starting point for constructing spectra was the notion of \emph{right multi-adjoint}, which is at first sight far more abstract and algebraic than admissibility, and whose geometric meaning is more subtle. But modulo a slight additional assumption - we will refer as \emph{Diers condition} in the second part - which is easily encountered in practice, this situation leads in a very natural and concrete way to a notion of spectrum.  The topos theoretic approach is more abstract but also more ``universal" as it is based on syntactic data, and start from a complicated situation (admissibility) to process to a natural construction; Diers way is more concrete and based on semantical data, and while the construction of the spectrum, being point-set, could seem a bit ``handmade", the algebraic situation it starts with is far more natural. Moreover this method subsumes not only most of the usual examples of algebraic geometry or the structured versions of Stone dualities, but also a vast list of exotic examples which were unsuspected before Diers investigation, and are not necessarily suited for the general topos theoretic approach. Indeed, the later requires the categories it use as ``local data" to construct the spectrum to be axiomatisable by geometric theories, as well as the factorization system need to be left generated - in some sense, also axiomatisable. In particular local objects must form a (non full) subcategory of ambient objects, and are models of a \emph{geometric extension} of the theory behind the ambient objects. In Diers this condition is largely relaxed as local objects are related to ambient objects through a functor that is not required to be faithful nor injective on objects.  \\

Those two papers will hence deals with Diers approach. While Diers work on multiversal constructions in category theory has been rather well known amongst category theory community, his presentation of the spectral construction seems to have been poorly acknowledged, perhaps due to a restricted circulation of the main paper \cite{Diers}: this contrasts with the quality of the paper itself and the highly practical and comprehensive construction he proposed. We hope those two new papers will help to make more people aware of Diers work on spectra. \\  

This first paper will revolves essentially about the first step in Diers approach. More precisely, we will focus on its notion of interest and its variations, providing different presentations of this situation and a purely categorical analysis of it, postponing the actual construction of the spectrum and the geometric analysis in the second paper. Most of the content of this paper is expository, and aimed at synthesizing as much as possible of the ``algebraic" aspects of the construction, gathering and relating different notions dispatched in several papers as \cite{diers1977categories}, \cite{Diers-multipres}, \cite{Taylor}, \cite{adámek2019nice}... However we try to present them as explicitly and originally as possible, providing alternative presentation and proofs of some already known results, and also providing new observations at some points. The last section will also provide a totally new method, whose relevance will however appear later with a 2-categorical version. \\

Right multi-adjoint were introduced by Diers and extensively studied in \cite{diers1977categories}, \cite{Diers} or \cite{Diers-multipres}, amongst ``multi" versions of universal properties and usual categorical constructions. Multiversal properties are analogous of universal properties where, rather than having a unique solution representing a construction, one has a canonical small set of solution jointly assuming the universal property. The prototypical situation is the notion of multirepresentable functor into $\Set$, that is, a functor that decompose as a coproduct of representable functors: the other situations are constructed from this as well as universal construction are done by representing a functor into $\Set$. For instance, as well as a (co)limit over a diagram is an object representing the functor assigning to any object the set of (co)cones over the diagram with this object as tip, a \emph{multi (co)limit} exists when this functor is multi-representable, and the ``local representing objects" form a small family of (co)cone such that any other (co)cone factorizes uniquely through exactly one of them. The other main example of multi-construction is the notion of \emph{right multi-adjoint}, to which revolves the present paper. In an ordinary adjunction, any object in the category where the right adjoint lands admits \emph{one} unit uniquely factorizing any map from this object toward an object in the range of the right adjoint: and the left adjoint is used to provide the codomain of this unit. In a multi-adjonction, there is no global left adjoint, hence no uniquely defined unit under a fixed object: one rather has a small cone of \emph{local units} under a fixed object, which jointly play the role of the unit in the sense that any arrow from this object toward the righ adjoint uniquely factorizes through exactly one of those units, followed by a morphism in the range of the right adjoint. This is a special situation of the more general notion of \emph{local right adjoint}, were again one lacks a global left adjoint, but is able to construct local left adjoints to restrictions at slices: however in this case, while any object still posses a cone of local units, one cannot in general enforce the \emph{smallness} of this cone, and right multi-adjointess amounts in fact exactly to ``local adjointness plus small solution set condition". \\

The second part of this work concerns a characterization of right multi-adjoint through free coproduct completion, and we reprove as explicitly as possible a result, already known in \cite{diers1977categories} but proved in a different way, stating that a functor is right multi-adjoint if and only if its free coproduct extension is right adjoint. We will see in the second part that this result is the ``discrete version" of the spectral adjunction.\\

The third section of this work is about the orthogonality aspects of local adjunctions, and gives some characterizations and properties of ``diagonally universal morphism" as defined in \cite{Diers}, which will play a central role in the second part when defining the topology of the spectrum. \\

The fourth section contains some new results and improvements; we consider the case of accessible right multi-adjoint in the context of \emph{locally finitely multipresentable categories} -which were also introduced by Diers - and revisit some results expressing how to construct locally finitely multi-presentable categories from accessible right multi-adjoint satisfying some relative full faithfulness property; we also provide a "right accessible multi-adjoint theorem" in \cref{Accessible multi-adjoint functor theorem}, and finally improve a theorem of \cite{diers1977categories} on a method for constructing locally finitely multipresentable categories from some orthogonality conditions.\\

The last part contains new results. We examine a situation, in the context of a factorization system in the presence of a terminal object, producing a case of right multi-adjunction. However in this last section, to emphasize the geometric interpretation, we will work with a convention that produces actually left multi-adjonction. The interest of this construction will be revealed in a future work applying its bicategorical analog to the bicategory of Grothendieck toposes in order to construct notions of ``2-geometries" and spectra for toposes.


\section{Local right adjoints and stable functors }

In this first section, we recall our three notions of interest, namely local right adjoints, right multi-adjoints and stable functors. We first give some technical points about the behavior of the local units of the local adjunctions. We also prove that for a local right adjoint, the local adjunctions enjoy automatically a Beck-Chevalley condition, which was seemingly unnoticed until now. Then we turn to different characterization of local adjointness in term of nerves and initial family, and introduce the stronger notion of right multi-adjoint and recall a variant of Freyd adjoint functor theorem for multi-adjoint. Finally we turn to the notion of stable functor, as studied by Taylor in \cite{Taylor}, and also in \cite{weber2004generic}, and we prove equivalence with the notion of local right adjoint. 

\begin{definition}
A functor $ U : \mathcal{A} \rightarrow \mathcal{B}$ is said to be a \emph{local right adjoint} if for each object $ A $ of $ \mathcal{A}$ the restriction of $ U$ to the slice $ \mathcal{A}/A$ has a left adjoint
\[
\begin{tikzcd}
\mathcal{A}/A \arrow[rr, "U/A"', bend right] & \perp & \mathcal{B}/U(A) \arrow[ll, "L_A"', bend right]
\end{tikzcd} \] 
where we denote $ A_f$ the domain of the arrow $ L_A(f) $ in $ \mathcal{A}/A$. In the following we will also denote $ U/A$ as $ U_A$ for concision. The maps $ \eta_f^A$ for $ f: B \rightarrow U(A)$ are called \emph{local units} under $B$. 
\end{definition}

The definition of a local right adjoint means that for any arrow $ f : B \rightarrow U(A)$ in $ \mathcal{B}\downarrow U $ and $ u : A' \rightarrow A$ in $ \mathcal{A}/A$ we have triangles in $\mathcal{B}$ and $\mathcal{A}$ respectively
\[ 
\begin{tikzcd}
B \arrow[rr, "f"] \arrow[rd, "\eta^A_f"'] &                                 & U(A) \\
                                          & U(A_f) \arrow[ru, "U_AL_A(f)"'] &     
\end{tikzcd} \quad 
\begin{tikzcd}
                    & A_{U(u)} \arrow[ld, "\epsilon^A_u"'] \arrow[rd, "L_A(U(u))"] &   \\
A' \arrow[rr, "u"'] &                                                              & A
\end{tikzcd} \]
satisfying the triangle identities 
\[ 
\begin{tikzcd}
U_A \arrow[r, "\eta^A_{U_A}", Rightarrow] \arrow[rd, equal] & U_AL_AU_A \arrow[d, "U_A(\epsilon^A)", Rightarrow] \\
                                                                 & U_A                                               
\end{tikzcd} \quad 
\begin{tikzcd}
L_A \arrow[r, "L_A(\eta^A)", Rightarrow] \arrow[rd, equal] & L_A U_A L_A \arrow[d, "\epsilon^A_{L_A}", Rightarrow] \\
                                                    & L_A                                                  
\end{tikzcd} \]
In other words we have the following retractions 
\[ 
\begin{tikzcd}[row sep=small, column sep=large]
U(A') \arrow[rd, "\eta^A_{U(u)}" description] \arrow[dd, equal] \arrow[rrrd, "U_A(u)", bend left=20] &                                                                                   &&      \\
                                                                                          & U(A_{U(u)}) \arrow[ld, "U_A(\epsilon^A_u)" description] \arrow[rr, "U_AL_A(U(u))" description] && U(A) \\
U(A') \arrow[rrru, "U_A(u)"', bend right=20]                                                  &                                                                                   &&     
\end{tikzcd} \quad 
\begin{tikzcd}[row sep=small, column sep=large]
A_f \arrow[rd, "L_A(\eta^A_f)" description] \arrow[dd, equal] \arrow[rrrd, "L_A(f)", bend left=20] &                                                                                        &  &   \\
                                                                             & A_{U_AL_A(f)} \arrow[rr, "L_AU_AL_A(f)" description] \arrow[ld, "\epsilon^A_{L_A(f)}" description] &  & A \\
A_f \arrow[rrru, "L_A(f)"', bend right=20]                                      &                                                                                        &  &  
\end{tikzcd} \]
defining an isomorphism 
\[ \mathcal{A}/A[L_A(f), u] \simeq \mathcal{B}/U(A)[f, U(u)] \]
sending an arrow $ v : L_A(f) \rightarrow u$, resp. an arrow $ g : f \rightarrow U(u)$, to the composite triangle on the left, resp. on the right
\[  
\begin{tikzcd}[sep= large]
B \arrow[rd, "f"'] \arrow[r, "\eta^A_f"] & U(A_f) \arrow[d, "U_AL_A(f)" description] \arrow[r, "U_A(v)"] & U(A') \arrow[ld, "U_A(u)"] \\
                                         & U(A)                                                          &                           
\end{tikzcd} \quad 
\begin{tikzcd}[sep= large]
A_f \arrow[rd, "L_A(f)"'] \arrow[r, "L_A(g)"] & A_{U(u)} \arrow[r, "\epsilon^A_u"] \arrow[d, "L_A(U(u))" description] & A' \arrow[ld, "u"] \\
                                              & A                                                                     &                   
\end{tikzcd} \]

\begin{remark}
Beware that in general we cannot enforce the counits to be pointwise iso, that is, to require each $U_A$ to be full and faithful. Hence the factorization of a morphism in the range of $U$ may not be trivial. Morally, the factorization through the unit only takes in account the object of $ \mathcal{A}$ whose strict image is the codomain, while, even when the domain is in the image of $U$, the factorization may not remember from which precise object in $ \mathcal{A}$ it comes from. 
\end{remark}

\begin{remark}
For any $ u : A_1 \rightarrow A_2$ in $\mathcal{A}$, functoriality of $U$ makes the following square commute up to equality
\[ 
\begin{tikzcd}
\mathcal{A}/A_1 \arrow[r, "U_{A_1}"] \arrow[d, "A/u"'] \arrow[rd, "=", phantom] & \mathcal{B}/U(A_1) \arrow[d, "\mathcal{B}/U(u)"] \\
\mathcal{A}/A_2 \arrow[r, "U_{A_2}"']                                           & \mathcal{B}/U(A_2)                              
\end{tikzcd} \]
its corresponding mate 
\[ 
\begin{tikzcd}
\mathcal{A}/A_1 \arrow[d, "A/u"', ""{name=U, inner sep=0.1pt}] & \mathcal{B}/U(A_1) \arrow[d, "\mathcal{B}/U(u)"] \arrow[l, "L_{A_1}"'] \\
\mathcal{A}/A_2                   & \mathcal{B}/U(A_2) \arrow[l, "L_{A_2}", ""{name=D, below, inner sep=0.3pt}] 
\arrow[from=D, to=U, Rightarrow, bend right=30, "\sigma^u"']
\end{tikzcd} \]
defined as the composite
\[ \begin{tikzcd}[sep=large]
L_{A_2} \mathcal{B}/U(u) \arrow[r, " L_{A_2}\mathcal{B}/U(u)(\eta^{A_1}) "] & L_{A_2} \mathcal{B}/U(u) U_{A_1} L_{A_1} = L_{A_2}U_{A_2}\mathcal{A}/u L_{A_1} \arrow[r, "\epsilon^{A_2}_{\mathcal{A}/u L_{A_1}}"] & \mathcal{A}/uL_{A_1}
\end{tikzcd}   \]
This mates relates in a canonical way the unit of any $ f : B \rightarrow U(A_1)$ and the unit of the composite $ U(u)f : B \rightarrow U(A_2)$ as seen in the following diagram
\[ 
\begin{tikzcd}[row sep=large, column sep=huge]
A_f \arrow[rr, "L_{A_1}(f)"]                                                                                                                             &  & A_1 \arrow[d, "u"]                                                                                       \\
A_{U_{A_2}\mathcal{A}/uL_{A_1}(f)} \arrow[rr, "L_{A_2}U_{A_2}\mathcal{A}/uL_{A_1}(f)" description] \arrow[u, "\epsilon^{A_2}_{\mathcal{A}/uL_{A_1}(f)}"] &  & A_2                                                                                                      \\
A_{\mathcal{B}/U(u)U_{A_1}L_{A_1}(f)} \arrow[rru, "L_{A_2}\mathcal{B}/U(u)U_{A_1}L_{A_1}(f)" description] \arrow[u, equal]                          &  & A_{U(u)f} \arrow[u, "L_{A_2}(\mathcal{B}/U(u)(f))"'] \arrow[ll, "L_{A_2}\mathcal{B}/U(u)(\eta^{A_1}_f)"]
\end{tikzcd} \]
But surprisingly, this mates is automatically an isomorphism because of the universal property of the units, as stated in the following proposition: 
\end{remark}

\begin{theorem}\label{BC}
Let be $ U : \mathcal{A} \rightarrow \mathcal{B}$ a local right adjoint. Then for any $ u : A_1 \rightarrow A_2$ in $\mathcal{A}$, we have the Beck-Chevalley condition at $u$, that is, the canonical transformation $ \sigma^u$ is a point-wise isomorphism. 
\end{theorem}

\begin{proof}
Remark that for each $ u :A_1 \rightarrow A_2 $ and $ f : B \rightarrow U(A_1)$, the morphism $ \sigma^u_f : L_{A_2}(\mathcal{B}/U(u)(f)) \rightarrow L_{A_1}(f)$ is in $\mathcal{A}$, and we have a factorization 
\[ 
\begin{tikzcd}
B \arrow[rd, "\eta^{A_1}_f" description] \arrow[rr, "f"] \arrow[dd, "\eta^{A_2}_{U(u)f}"']                         &                                                    & U(A_1) \arrow[rd, "U(u)"] &        \\
                                                                                                                   & U(A_f) \arrow[ru, "U_{A_1}L_{A_1}(f)" description] &                           & U(A_2) \\
U(A_{U(u)f}) \arrow[rrru, "U_{A_2}L_{A_2}(U(u)f)" description, bend right=15, near end] \arrow[ru, "U(\sigma^u_f)" description] &                                                    &                           &       
\end{tikzcd} \]
Observe that $ \sigma^u_f$ is the unique arrow in $ \mathcal{A}$ provided by the universal property of the unit $ \eta^{A_2}_{U(u)f}$ at $ \eta^{A_1}_f$ seen as an arrow $ U(u)f \rightarrow U(u)U_{A_1}L_{A_1}(f)$ in $ \mathcal{B}/U(A_2)$. But on the other hand, by the universal property of $ \eta^{A_1}_f$ at $ \eta^{A_2}_{U(u)f}$ seen as an arrow $ f \rightarrow U_{A_1}L_{A_1}(f)U(\sigma^u_f)$ in $\mathcal{B}/U(A_1)$, there exists a unique arrow $ w : A_f \rightarrow A_{U(u)f}$ in $ \mathcal{A}$ such that
\[ 
\begin{tikzcd}[sep=large]
                                                                                 & B \arrow[rr, "f"] \arrow[rd, "\eta^{A_2}_{U(u)f}"'{inner sep=0.1pt}, near start] \arrow[ld, "\eta_f^{A_1}"'] &                                          & U(A_1)                                 \\
U(A_f) \arrow[rr, "U(w)"', dashed] \arrow[rrru, crossing over, "U_{A_1}L_{A_1}(f)" description, near end] &                                                                                            & U(A_{U(u)f}) \arrow[r, "U(\sigma^u_f)"'] & U(A_f) \arrow[u, "U_{A_1}L_{A_1}(f)"']
\end{tikzcd} \]
Now we prove that $ w$ and $ \sigma^u_f$ are mutual inverses in $\mathcal{A}$. First, as \[ U(\sigma^u_f) \eta^{A_2}_{U(u)f} = \eta^{A_1}_f\] and $ 1_{L_{A_1}(f)}$ is the unique map induced by $ \eta_f^{A_1} $ seen as an arrow $ f \rightarrow U_{A_1}L_{A_1}(f)$, then necessarily we have a retraction in $ \mathcal{A}$ 
\[ 
\begin{tikzcd}
A_f \arrow[rr, equal] \arrow[rd, "w"'] &                                    & A_f \\
                                & A_{U(u)f} \arrow[ru, "\sigma^u_f"'] &    
\end{tikzcd} \]
but again, as now $ \eta_{U(u)f}^{A_2} = U(w) \eta^{A_1}_f$ and $ 1_{L_{A_2}(U(u)f)}$ is the unique map induced by $ \eta^{A_2}_{U(u)f}$ as an arrow $ U(u)f \rightarrow U_{A_2}L_{A_2}(U(u)f)$, we have a retraction in $\mathcal{A}$
\[ 
\begin{tikzcd}
                                              & A_f \arrow[rd, "w"] &           \\
A_{U(u)f} \arrow[ru, "\sigma^u_f"] \arrow[rr, equal] &                     & A_{U(u)f}
\end{tikzcd} \]
and $\sigma^u_f$ defines both an iso $ A_{U(u)f} \simeq A_f$ in $\mathcal{A}$ and $ L_{A_2}\mathcal{A}/u(f) \simeq \mathcal{B}/U(u) L_{A_1}(f)$ which can be shown to be natural.
\end{proof}

\begin{remark}
Beck-Chevalley condition says that factorization through local unit are not modified by postcomposing with an arrow in the range of $U$: that is, for $f : B \rightarrow U(A_1)$ and $ u : A_1 \rightarrow A_2$, then the Beck chevalley transformation provides an isomorphism $ A_f \simeq A_{U(u)f}$.
\end{remark}

\begin{corollary}\label{units are candidate}
Let be $f : B \rightarrow U(A)$: then we have $ A_f \simeq A_{\eta^A_f} $ in $\mathcal{A}$ and $ \eta^{A_f}_{\eta^A_f} \simeq \eta^A_f$ in $B\downarrow U$.
\end{corollary}

\begin{proof}
Consider the following diagram
\[ 
\begin{tikzcd}
B \arrow[rr, "\eta^A_f"] \arrow[rd, "\eta^{A_f}_{\eta^A_f}"'] &                                                         & U(A_f) \arrow[r, "U_AL_A(f)"] & U(A) \\
                                                              & U(A_{\eta^A_f}) \arrow[ru, "U_{A_f}L_{A_f}(\eta_f^A)"'] &                               &     
\end{tikzcd} \]
Then by what precedes we have $ \sigma^{U_AL_A(f)}_{\eta^A_f}$ is an iso as seen in the following diagram 
\[ 
\begin{tikzcd}[sep =large]
B \arrow[rr, "\eta^A_f"] \arrow[rd, "\eta^{A_f}_{\eta^A_f}" description] \arrow[dd, "\eta^A_{U_AL_A(f)\eta_f^A}"']                       &                                                                    & U(A_f) \arrow[r, "U_AL_A(f)"] & U(A) \\
                                                                                                                                         & U(A_{\eta^A_f}) \arrow[ru, "U_{A_f}L_{A_f}(\eta_f^A)" description] &                               &      \\
U(A_{U_AL_A(f)\eta^A_f}) \arrow[ru, "\sigma^{U_AL_A(f)}_{\eta_f^A} \atop \simeq" description] \arrow[rrruu, "U_AL_A(U_AL_A(f)\eta^A_f)" description, bend right=25] &                                                                    &                               &     
\end{tikzcd} \]
But $ U_AL_A(f)\eta^A_f = f $, exhibiting an isomorphism 
\[ 
\begin{tikzcd}[column sep=large]
B \arrow[rd, "\eta^{A_f}_{\eta^A_f}" description] \arrow[d, "\eta_f^A"'] &                 \\
U(A_f) \arrow[r, "\sigma^{U_AL_A(f)}_{\eta_f^A}"'{inner sep=4pt}, "\simeq" description]                       & U(A_{\eta^A_f})
\end{tikzcd} \]
\end{proof}

A consequence is that local units that are related by an arrow in the range of $U$ must actually be isomorphic as objects under their domain:

\begin{corollary}\label{units form a set}
Let be $f_1 : B \rightarrow U(A_1)$ and $ f_2 : B \rightarrow U(A_2)$, such that there exists a morphism $u$ in $ \mathcal{A}$ and a triangle 
\[ 
\begin{tikzcd}
                               & B \arrow[ld, "\eta^{A_1}_{f_1}"'] \arrow[rd, "\eta^{A_2}_{f_2}"] &            \\
U(A_{f_1}) \arrow[rr, "U(u)"'] &                                                                  & U(A_{f_2})
\end{tikzcd} \]
then $u$ is an isomorphism. 
\end{corollary}

\begin{proof}
By \cref{BC}, we have that 
\[
  \eta^{A_{f_1}}_{\eta_{f_1}^{A_1}} \simeq \eta_{f_1}^{A_1} \hskip 2cm
  \eta^{A_{f_2}}_{\eta_{f_2}^{A_2}} \simeq \eta_{f_2}^{A_2}
 \]
 But by Beck-Chevalley condition at $u$, we also have 
\[
\eta^{A_{f_1}}_{\eta_{f_1}^{A_1}} 
 \simeq \eta^{A_{f_2}}_{U(u)\eta^{A_{1}}_{f_1}}
 = \eta^{A_{f_2}}_{\eta^{A_2}_{f_2}}     
\] 
\end{proof}

From their very universal property, locals units under a given object have to live isolated from each other, each one in its connected component:

\begin{corollary}\label{one unit per connected component}
If $ U : \mathcal{A} \rightarrow \mathcal{B}$ is local right adjoint, then for any $B$, any two $f_1 : B \rightarrow U(A_1)$, $ f_2 : B \rightarrow U(A_2)$ in the same connected component of $B \downarrow U$ factorize through the same unit, that is $ \eta^{A_1}_{f_1} \simeq \eta^{A_2}_{f_2} $, $A_{f_1} \simeq A_{f_2} $.
\end{corollary}

\begin{corollary}
For any $ f : B \rightarrow U(A)$, $ \eta^A_f$ is initial in $(B\downarrow U)\downarrow f$.
\end{corollary}

\begin{definition}
Let be $\mathcal{C}$ a category; \emph{multi-initial family} in $\mathcal{C}$ is a family of objects $ (X_i)_{i \in I} $ such that for any $C$ in $\mathcal{C}$ there is a unique $i \in I$ and a unique arrow $ X_i \rightarrow C$. A $X_i$ for $i \in I$ is a local initial object.  
\end{definition}

\begin{remark}
Observe that in this definition, if one has an arrow $ f : C_1 \rightarrow C_2$, then $C_1$ and $C_2$ lie under the same local initial object. More generally, if two objects $ C_1, \, C_2$ are in the same connected component and $ X_i \rightarrow C_1$, $X_j \rightarrow C_2$ are the initial maps, then they lie under the same initial object: for there is a zigzag 
\[ 
\begin{tikzcd}[sep =small]
    & B_1 \arrow[ld] \arrow[rd] &     & ... \arrow[ld] \arrow[rd] &         & B_n \arrow[ld] \arrow[rd] &     \\
C_1 &                           & B_2 &                           & B_{n-1} &                           & C_2
\end{tikzcd}\]
and by uniqueness of the local initial object over a given object, necessarily the local object over $ B_1$ is the same as the local initial object over $B_3$ because they both lies over $ B_2$ and so on. Conversely any two objects under a same local initial object are in the same connected component. Hence there is exactly one local initial object by connected component. 
\end{remark}

\begin{proposition}
Let $ U : \mathcal{A} \rightarrow \mathcal{B}$ be a local right adjoint and $ B$ in $\mathcal{B}$: then the coma category $ B \downarrow U$ has a multi-initial family. 
\end{proposition}

\begin{proof}
We claim that the (large) class of local units under $B$ is a multi-initial family in $ B\downarrow U$. First, for any $ f: B \rightarrow U(A)$, we have by local adjunction in $A$ an arrow $L_A(f)$ in $ B\downarrow U$; but now suppose there is an other $ g : B \rightarrow U(A')$ such that $ \eta^{A'}_g$ has a map $ g \rightarrow f$ in $ B \downarrow U$, that is there is a map $u$ in $\mathcal{A}$ such that $f = U(u)\eta^{A_1}_g$; but then by the universal property of the unit there is a unique factorization
\[ 
\begin{tikzcd}
B \arrow[rr, "f"] \arrow[rd, "\eta^A_f" description] \arrow[rdd, "\eta^{A'}_g"', bend right] &                                                      & U(A) \\
                                                                                             & U(A_f) \arrow[ru, "U_AL_A(f)" description] \arrow[d, dashed, "U(w)" description]           &      \\
                                                                                             & U(A_g) \arrow[ruu, "U_A(u)" description, bend right] &     
\end{tikzcd} \]
But by \cref{units form a set}, this forces $\eta^A_f \simeq \eta^{A'}_g $. \end{proof}

\begin{remark}
Observe that without further assumption, the multi-initial family of local units at a given $B$ may not be small. This is the point of the following notion.
\end{remark}

\begin{definition}
A functor $ U : \mathcal{A} \rightarrow \mathcal{B}$ is said to be a \emph{right multi-adjoint} if for any $B$ in $\mathcal{B}$ there is a \emph{small} multi-initial family in the comma $ B\downarrow U$. 
\end{definition}

\begin{remark}
Observe that this definition is indexed by the domain, that is, by the object $B$ in $ \mathcal{B}$, while the definition of local right adjoint was indexed by the objects $ A$ of $ \mathcal{A}$. However it is easy to see that any right multi-adjoint is in particular a local right adjoint: for any $ f : B \rightarrow U(A)$ define $ \eta^A_f$ to be the unique local initial object over $f$, and $ L_A(f)$ to be the unique map $ \eta^A_f \rightarrow f$. Then $ \eta^A_f$ has the universal property of the unit as any triangle $ f \rightarrow U(u)$ in $ \mathcal{B}/U(A)$
\[ 
\begin{tikzcd}
B \arrow[rr, "f"] \arrow[rd, "g"'] &                           & U(A) \\
                                   & U(A') \arrow[ru, "U(u)"'] &     
\end{tikzcd} \]
can be seen as a triangle $ g \rightarrow f $ in $B \downarrow U$ forcing $g$ and $f$ to be in the same connected component. Therefore $ \eta^A_f$ is also the initial object over $g$, inducing an unique arrow $ \eta^A_f \rightarrow g$ in $B \downarrow U$, that is a unique arrow $ L_A(f) \rightarrow u$ in $ \mathcal{A}/A$ as desired.
\end{remark}

Existence of a multi-initial object in the comma is reminiscent of the so called solution set condition in Freyd Adjoint Functor Theorem. Let us precise this fact, in order to retrieve an analogous multi-adjoint theorem, in the vein of \cite{Taylor}. First we need a weakening of the notion of initial family 

\begin{definition}
Let be $\mathcal{C}$ a category; a \emph{weakly initial family} is a family $ (X_i)_{i \in I}$ such that for any object $ C$, there some $i \in I$ and some arrow $ X_i \rightarrow C$.
\end{definition}

\begin{remark}
Observe that in this definition, there is no requirement of uniqueness, nor for the index, nor for the arrow, so there may be several weakly initial objects over an arbitrary one. In the following we are interested in \emph{small} weakly initial families:
\end{remark}

\begin{definition}
A functor $U : \mathcal{A} \rightarrow \mathcal{B}$ is said to satisfy the \emph{Solution Set Condition} if each of the coma $B \downarrow U$ admits a small weakly initial family: that is a family $ (n_i : B \rightarrow U(A_i))_{i \in I_B}$ such that any map from $ B$ to $U$ factorizes through some (non necessarily unique) $n_i$. 
\end{definition}

\begin{proposition}\label{Mradj = Lradj + SSC}
A functor $U : \mathcal{A} \rightarrow \mathcal{B}$ is a right multi-adjoint if and only if it is local right adjoint and satisfies the Solution Set Condition. 
\end{proposition}

\begin{proof}
It is obvious that a right multi-adjoint satisfies the Solution Set Condition as for any object $B$ the small multi-initial family of $B\downarrow U$ is in particular a small weakly initial family, and any right multi-adjoint is trivially local right adjoint. \\

For the converse, observe that the small family of local units of a right multi-adjoint produces a small (weakly) initial family, so that it always satisfies the solution set condition. Preservation of wide pullback just come from the fact that they are ordinary products in the slices, hence preserved by the restriction as it is right adjoint. 
\end{proof}

\begin{definition}
A \emph{connected limit} is a limit indexed by a connected category; a \emph{wide pullback} is a limit of a diagram over a set of arrows with common codomain. In particular any wide-pullback is a connected limit. 
\end{definition}

\begin{proposition}
A category has wide pullbacks if and only if each slice $\mathcal{A}/A$ has products. A functor preserves connected limits if and only if it preserves wide pullbacks. 
\end{proposition}

In Freyd adjoint functor theorem, assuming completeness and local smallness of the domain category $\mathcal{A}$, one can prove that a functor preserving limits is right adjoint if it satisfies the solution set condition : this is achieved by constructing an initial object in any $ B \downarrow U$ as the limit $ \underset{i \in I}{\lim} \, U(A_i) \simeq U(\underset{i \in I}{\lim} \,{A_i}) $. Similarly we can give a corresponding right multi-adjoint theorem, an accessible version of will be given in \cref{Accessible multi-adjoint functor theorem} :

\begin{proposition}\label{Mradj preserves connected limits}
Let $ \mathcal{A}$ be a category with wide pullbacks; then a functor $ U : \mathcal{A} \rightarrow \mathcal{B}$ is a right multi-adjoint if and only if it satisfies the solution set condition and preserves wide pullbacks. 
\end{proposition}

\begin{proof}
 Indeed, assuming that $\mathcal{A}$ has connected limit and that $U$ preserves them, the solution set condition will enable us to prove that its restriction on any slice $ U_A : \mathcal{A}/A \rightarrow \mathcal{B}/{U(A)}$ is right adjoint by constructing an initial object in any of the coma $ f \downarrow U_A$ whit $ f : B \rightarrow U(A) \in \mathcal{B}/{U(A)} $ (this is nothing but the category of factorizations of $f$ through $U$) as the following. If $ S_B = (n_i : B \rightarrow U(A_i))_{i \in I_B} $ is the small weakly initial family in $B\downarrow U$ given by the solution set condition, define $ S_f = S_B\downarrow f $ consisting of all the pairs $ (u,i)$ with $i \in I_B$ and $f = F(u)n_i$ a factorization of $f$ through $ n_i$: then we can do the wide pullback of the $A_i$ over $ A$ and it is preserved by $U$, so that we have 

\[ \begin{tikzcd}
 & U(\times_A (A_i)_{i \in S_f}) \simeq \times_{U(A)}(U(A_i))_{i \in S_f} \arrow[]{ld}[swap]{U(p_i)} \arrow[]{rd}{U(p_j)} & \\
 U(A_i) \arrow[]{rd}{} & B \arrow[dashed]{u}{\exists !} \arrow[]{d}{f} \arrow[]{r}{n_i} \arrow[]{l}{n_j} & U(A_j) \arrow[]{ld}{}\\ & U(A) & 
\end{tikzcd} \]
Now we claim that this unique arrow $B \rightarrow U(\times_A (A_i)_{i \in S_f})$ is the desired initial object in $ f \downarrow U_A$, as any factorization of $f$ factorizes itself through some of the $n_i$ in $S_f$ hence through the wide pullback.
For the converse we will make use of stability in order to recover the Solution Set Condition, while the preservation of wide pullback just come from the fact that they are ordinary products in the slices, hence preserved by the restriction as it is right adjoint. 
\end{proof}

Now we turn to another facet of the local adjunction.

\begin{definition}
Let $ U : \mathcal{A} \rightarrow \mathcal{B}$ be a functor; then the \emph{co-nerve} of $U$ is the functor \[ 
\begin{tikzcd}
\mathcal{B} \arrow[r, "\mathcal{N}^U"] & {[\mathcal{A}, \mathcal{S}et]}
\end{tikzcd} \]
sending each $B$ in $\mathcal{B}$ to the functor $ \mathcal{N}^U_B=  \mathcal{B}[B,U(-)] : \mathcal{A} \rightarrow \mathcal{S}et $.  
\end{definition}

For any $ B$ in $ \mathcal{B}$, we have a discrete opfibration 
\[ \pi_B : \int \mathcal{N}^U_B \rightarrow \mathcal{A} \]
whose objects are pairs $ (A,f)$ with $ A $ in $\mathcal{A}$ and $ f : B \rightarrow U(A)$, and morphisms $ (A_1, f_1) \rightarrow (A_2, f_2)$ are $ u : A_1 \rightarrow A_2$ with $ U(u) f_1 = f_2$. There is a general result saying that a functor is a right adjoint if the projection from the category of elements of its co-nerve at each object has a limit; in fact this says that the co-nerve functor is representable, that is, there is an initial object in the category of elements, which is the unit. We give here the corresponding statement for a right multi-adjoint.

\begin{definition}
Let $ U : \mathcal{A} \rightarrow \mathcal{B}$ be a functor; then the \emph{co-nerve} of $U$ is the functor \[ 
\begin{tikzcd}
\mathcal{B} \arrow[r, "\mathcal{N}^U"] & {[\mathcal{A}, \mathcal{S}et]}
\end{tikzcd} \]
sending each $B$ in $\mathcal{B}$ to the functor $ \mathcal{N}^U_B=  \mathcal{B}[B,U(-)] : \mathcal{A} \rightarrow \mathcal{S}et $.  
\end{definition}

For any $ B$ in $ \mathcal{B}$, we have a discrete opfibration 
\[ \pi_B : \int \mathcal{N}^U_B \rightarrow \mathcal{A} \]
whose objects are pairs $ (A,f)$ with $ A $ in $\mathcal{A}$ and $ f : B \rightarrow U(A)$, and morphisms $ (A_1, f_1) \rightarrow (A_2, f_2)$ are $ u : A_1 \rightarrow A_2$ with $ U(u) f_1 = f_2$. There is a general result saying that a functor is a right adjoint if the projection from the category of elements of its co-nerve at each object has a limit; in fact this says that the co-nerve functor is representable, that is, there is an initial object in the category of elements, which is the unit. We give here the corresponding statement for a right multi-adjoint.

\begin{definition}\label{multilimits}
Let $ F : I \rightarrow \mathcal{A}$ be a functor: then a \emph{multilimit} of $F$ is a small family of cones \[ (p_ i^j : L_j \rightarrow F(i))_{i \in I \atop j \in J} \]
such that for any cone $ (f_i : X \rightarrow F(i))_{i \in I}$ there is a unique $j \in J$ and a unique factorization of the cone $(f_i)_{i \in I}$ through the cone $ (p^j_i)_{i \in I} $.\\

A functor $U : \mathcal{A} \rightarrow \mathcal{B}$ preserves multilimits (or also, is \emph{multicontinuous}) if for any multilimit $ (p_i^j : L_j \rightarrow F(i))_{i \in I,\, j \in J}$ in $\mathcal{A}$, there is a multilimit $ ( q^k_i : M_k \rightarrow UF(i))_{i \in I, k\in K}$ in $\mathcal{B}$ and for each $ k \in K$ we have 
\[ M_k \simeq \coprod_{j \in J_k} U(L_j) \]
where $J_k$ is the set of $j \in J$ such that the cone $ (U(p^j_i) : U(L_j) \rightarrow UF(i))_{i \in I}$ factorizes through $ M_k $.
\end{definition}

Before going further, we think relevant to introduce the dual notion of multilimits, for we are going to use them also at the end later in this paper and also in a closure theorem for multireflective subcategories in a moment:

\begin{definition}\label{multicolimits}
Let $ F : I \rightarrow \mathcal{A}$ be a functor. Then a \emph{multicolimit} of $F$ is a small family of cocones \[ ((q^j_i : F(i) \rightarrow X_j)_{i \in I})_{ j \in J} \] that is multi-initial the the category of cocones over $F$: that is, any other cocone $(f_i : F(i) \rightarrow A)_{i \in I}$ factorizes uniquely through one of the $(q^j_i : F(i) \rightarrow X_j)_{i \in I}$ for a unique $j \in J$.\\ 

A category is (finitely) \emph{multicocomplete} if any (finite) diagram admits a multi-colimit.
Dually, a \emph{multicomplete} category is a category where any (finite) diagram has a multi-limits. \\

A functor $U :  \mathcal{A} \rightarrow \mathcal{B}$ is \emph{multicocontinuous} if for any multicolimit $ ((q^j_i : F(i)_i \rightarrow X_j)_{i \in I})_{ j \in J}$ in $\mathcal{A}$, then the composite $ UF : I \rightarrow \mathcal{B}$ has a multicolimit $ ((s^k_i : UF(i) \rightarrow Y_k)_{i \in I})_{k \in K}$ such that for any $ k \in K$ we have a coproduct decomposition \[ Y_k \simeq \prod_{J_k} U(X_j) \]
where $ J_k= \{ j \in J \mid \, (U(q^j_i) : UF(i) \rightarrow U(X_j))_ {i \in I} \textrm{ factorizes through } (s^k_i)_{i \in I} \} $. 
\end{definition}

\begin{remark}\label{universal property of mcolim}
The universal property of the multicolimits can be encapsulated, that for any other $ X$ in $\mathcal{B}$, in the following isomorphism
\[  \underset{j \in J}{\coprod} \, \mathcal{B}[F(i), X] \simeq \underset{i \in I}{\lim} \, \mathcal{B}[X_j, X] \]
\end{remark}

\begin{remark}
Remark that any (co)limit is in particular a multi-(co)limit with a single cone. In particular (co)completeness implies multi-(co)completeness, and (co)continuity implies multi-(co)continuity. Then in particular any corepresentable functors are multicontinuous.
\end{remark}

Let us get back to right multi-adjoints. The following observation just gives the obvious analog of the characterization of right adjoint in terms of the existence of the limit of the projection of the category of elements of the nerve:

\begin{proposition}
A functor $ U : \mathcal{A} \rightarrow \mathcal{B}$ is a local right adjoint if and only if for any $B$ in $\mathcal{B}$, each connected component of $ \int \mathcal{N}^U_B$ has an initial object. Moreover, $ U$ is a right multi-adjoint if $ \int \mathcal{N}^U_B$ has a set of connected components. Equivalently, $ U$ is a right multi-adjoint if and only if the functor $ \pi_B : \int \mathcal{ N }^U_B \rightarrow \mathcal{A} $ has a multi-limit in $\mathcal{A}$ and $U$ preserves it. 
\end{proposition}

The first half of this fact is tautological; for the second part, one can adapt \cite{borceux1994handbook}[proposition 3.3.2]. \\

Now from what was said, it appears that a right multi-adjoint is a functor such that the associated conerve in any object is ``locally representable". Indeed, any arrow from an object $B$ toward $U$ is determined first by the connected component it lies in, which corresponds to the local unit it factorizes through, and secondly by a choice of map in $\mathcal{A}$. This amounts to the following:

\begin{proposition}
Let $ \mathcal{U}$ be a multi-right adjoint: then for each $\mathcal{B}$ one has 
\[ \mathcal{N}^U_B \simeq \coprod_{i \in I_B} \mathcal{A}[A_i, -] \]
with $ I_B$ the set of connected components of $B$ and $ n_i: B \rightarrow U(A_i)$ the initial object of the connected component $i$. 
\end{proposition}

We finally end this subsection on a result on multireflective subcategories. Recall that full reflectives subcategories inherits limits and colimits from their ambient category in the following way. Limits can be computed directly in the subcategory and are created by the inclusion, which is proven by proving that the reflection of a limit computed in the ambient category is a limit in the subcategory, and observing the later must be preserved by inclusion, so that this reflection was actually an iso. For colimits, one compute the colimit in the ambient category, then take its reflection. In the context of multireflection, the correct analog of those statement are in term of limits and multicolimits.

\begin{theorem}\label{closures properties of mreflective subcategories}
Let be $ U : \mathcal{A} \hookrightarrow \mathcal{B}$ a full multireflective subcategory; then: \begin{itemize}
    \item if $ \mathcal{B}$ has colimits, then $\mathcal{A}$ has multicolimits
    \item if $\mathcal{B}$ has limits, then $ \mathcal{A}$ has connected limits and they are created by $U$.
\end{itemize} 
\end{theorem}

\begin{proof}
For multicolimits, take a functor $ F : I \rightarrow \mathcal{A}$; then one can compute the colimit $ (q_i : UF(i) \rightarrow \colim_{i \in I} UF(i))_{i \in I}$ in $\mathcal{B}$. Then consider its small cone of local units $ n_x : (\colim_{i \in I} UF(i) \rightarrow U(A_x))_{x \in I_{\colim_{i \in I} UF(i)}}$. Then by fullness of $U $ each composite $ n_xq_i :  UF(i) \rightarrow U(A_x) $ comes from a unique map $ q^x_i : F(i) \rightarrow A_x$ and $(q^x_i : F(i) \rightarrow A_x)_{i \in I, x \in I_{\colim_{i \in I} UF(i) }}$ is a multicolimit of $F$. \\

Now for connected limits, if we suppose $ F : I \rightarrow \mathcal{A}$ with $ I$ connected, then the category $ \lim_{i \in I} \, UF(i) \downarrow UF$ is connected; then by \cref{one unit per connected component}, all the limit projections $ p_i$ factorize through a same local unit $ n_F$
\[\begin{tikzcd}
	{\underset{i \in I}{\lim} \, UF(i)} & {UF(i)} \\
	{U(A_F)}
	\arrow["{p_i}", from=1-1, to=1-2]
	\arrow["{n_F}"', from=1-1, to=2-1]
	\arrow["{U(u^F_i)}"', from=2-1, to=1-2]
\end{tikzcd}\]
Then $( U(u^F_i) :A_F \rightarrow A_i)_{i \in I} $ is a limiting cone in $\mathcal{A}$: indeed, any cone $ (u_i : A \rightarrow F(i))_{i \in I}$ in $\mathcal{A}$ is transported by $U$ to a cone in $\mathcal{B}$, were it induces a unique arrow $ (U(u_i))_{i \in I} U(A) \rightarrow \lim_{i \in I} UF(i)$ and the composite $ n_F (U(u_i))_{i \in I} : U(A) \rightarrow U(A_F)$ comes uniquely from an arrow $ A \rightarrow A_F$ as desired. Hence $ A_F$ is a limit of the connected diagram $F$, but as $U$ preserves connected limits, $ n_F $ was actually an isomorphism $ \lim_{i \in I} \, UF(i) \simeq U(A_F)$. \end{proof}

\begin{remark}
In fact, concerning connected limit, the condition that $ U$ is full can be slightly relaxed into the condition of being \emph{relative full and faithful}, which will be defined later in \cref{relff}
\end{remark}

Finally we come to an alternative notion encapsulating the property of being a local right adjoint, but in a way that is more related to factorization systems. This was studied in \cite{Taylor} and \cite{weber2004generic}, and we prove there that this notion coincides with local right adjointness. It relies on an atlernative presentation of local unit in a more ``orthogonality structure" spirit. 

\begin{definition}
A candidate (diagonally universal toward $ U$ in the terminology of Diers), is a morphism $ n : B \rightarrow U(A)$ such that for any square of the following form there exist an unique morphism $ w : A \rightarrow A_1$  such that $ U(w)$ diagonalizes uniquely the square and the left triangle already commutes in $ \mathcal{A}$
\[  
\begin{tikzcd}[column sep=large]
B \arrow[]{r}{f} \arrow[]{d}[swap]{n} & U(A_1) \arrow[]{d}{U(v)} & &  & A_1 \arrow[]{d}{v}  \\ U(A) \arrow[]{r}[swap]{U(u)} \arrow[]{ru}{ U(w)} & U(A_2)  & & A \arrow[dashed]{ru}{\exists ! w} \arrow[]{r}{u} & A_2
\end{tikzcd}\]
\end{definition}
\begin{definition}
A functor $ U : \mathcal{A} \rightarrow \mathcal{B}$ is \emph{stable} when any morphism $ f : B \rightarrow U(A)$ factorizes uniquely through the range of $ U $ as 
\[ 
\begin{tikzcd}
B \arrow[rr, "f"] \arrow[rd, "n_f"'] &                              & U(A) \\
                                     & U(A_f) \arrow[ru, "U(u_f)"'] &     
\end{tikzcd} \]
where $n_f : B \rightarrow U(A_f) $ is a candidate. We refer to this factorization as \emph{the stable factorization of} $f$ and to $ n_f$ as \emph{the candidate of} $f$. 
\end{definition}
\begin{remark}
Then the candidate for $f$ is the initial object in the category of factorizations of $f$ through the range of $U$. Indeed, for any other factorization through $U$
\[ 
\begin{tikzcd}
B \arrow[rr, "f"] \arrow[rd, "g"'] &                           & U(A) \\
                                   & U(A') \arrow[ru, "U(u)"'] &     
\end{tikzcd}\]
one gets a square as below, where $n_f$ produces a unique $w$ such that $ U(w)$ is a filler
\[ 
\begin{tikzcd}
B \arrow[r, "g"] \arrow[d, "n_f"']                                 & U(A') \arrow[d, "U(u)"] \\
U(A_f) \arrow[r, "U(u_f)"'] \arrow[ru, "U(v)" description, dashed] & U(A)                   
\end{tikzcd} \]
\end{remark}

\begin{proposition}
For a functor $ U : \mathcal{A} \rightarrow \mathcal{B}$ and $ B$ in $\mathcal{B}$ we have the following 
\begin{itemize}
    \item If a candidate $ n_1 : B \rightarrow U(A_1)$ admits an arrow $ n_2 \rightarrow n_1$ from another candidate $n_2 : B \rightarrow U(A_2)$ in $B \downarrow U$, then we have $ n_1 \simeq n_2$ in $ B \downarrow U$ and $ A_1 \simeq A_2$ in $\mathcal{A}$.
    \item In particular, any two candidates in a same connected component of $ B \downarrow U$ are isomorphic.
    \item If $f : B \rightarrow U(A)$ admits a stable factorization, then it is unique up to unique isomorphism.
    \item In particular, when $U$ is stable, the stable factorization of any arrow is unique up to unique isomorphism.
\end{itemize}  
\end{proposition}

\begin{proof}
The first item is easily shown to implies the other ones. Suppose we have $ n_1$, $n_2$ and a triangle 
\[ 
\begin{tikzcd}
B \arrow[r, "n_1"] \arrow[d, "n_2"']  & U(A_1) \\
U(A_2) \arrow[ru, "U(u)" description] &       
\end{tikzcd}\]
Then we have a unique filler
\[ 
\begin{tikzcd}
B \arrow[r, "n_2"] \arrow[d, "n_1"']                                               & U(A_2) \arrow[d, "U(u)"] \\
U(A_1) \arrow[r, "U(1_{A_1})"', equal] \arrow[ru, "U(v)" description, dashed] & U(A_1)                  
\end{tikzcd} \]
But now there is a unique filler of the square 
\[ 
\begin{tikzcd}
B \arrow[r, "n_1"] \arrow[d, "n_2"']                                               & U(A_1) \arrow[d, "U(v)"] \\
U(A_2) \arrow[r, "U(1_{A_2})"', equal] \arrow[ru, "U(w)" description, dashed] & U(A_2)                 
\end{tikzcd} \]
so that $ u : A_2 \rightarrow A_2$ is both a retraction and a section in $A\mathcal{A}$, hence an isomorphism, so that $n_1 \simeq n_2$ in $B \downarrow U$. 
\end{proof}

\begin{proposition}
For any square as below
\[
\begin{tikzcd}
B_1 \arrow[r, "f_1"] \arrow[d, "f"'] & U(A_1) \arrow[d, "U(u)"] \\
B_2 \arrow[r, "f_2"']                & U(A_2)                  
\end{tikzcd}\]
the stable factorizations of $ f_1$ and $f_2$ are related by a unique morphism in $\mathcal{A}$ such that 
\[ 
\begin{tikzcd}[sep=large]
B_1 \arrow[r, "n_{f_1}"] \arrow[d, "f"'] & U(A_{f_1}) \arrow[d, "{U(w_{g,u})}" description] \arrow[r, "U(u_{f_1})"] & U(A_1) \arrow[d, "U(u)"] \\
B_2 \arrow[r, "n_{f_2}"']                & U(A_{f_2}) \arrow[r, "U(u_{f_2})"']                           & U(A_2)                  
\end{tikzcd} \]
\end{proposition}

\begin{proof}
 The desired $ w_{g,u}$ is the filler of the square 
 \[ 
\begin{tikzcd}
B_1 \arrow[r, "n_{f_1}"] \arrow[d, "n_{f_2}f"'] & U(A_{f_1}) \arrow[d, "U(uu_{f_1})"] \\
U(A_{f_2}) \arrow[r, "U(u_{f_2})"']             & U(A_2)                             
\end{tikzcd}\]
\end{proof}

\begin{theorem}\label{stable is MRadj}
Stable functors and local right adjoints coincide.
\end{theorem}

\begin{proof}
 Let $ U : \mathcal{A} \rightarrow \mathcal{B} $ be a stable functor. For each $A$ defines the functor 
\[ \begin{tikzcd}
\mathcal{A}/A \arrow[bend right=15]{rr}[swap]{U/A} & \perp & \mathcal{B}/U(A) \arrow[bend right=15]{ll}[swap]{L_A} 
\end{tikzcd} \]
where $L_a$ returns the left part of the initial factorization of an arrow by its associated candidate:
\[ \begin{array}{cccc}
   L_A : & \mathcal{B}/{U(A)} & \rightarrow & \mathcal{A}/A  \\
     & \begin{tikzcd}
     B \arrow[]{rr}{f} \arrow[]{rd}[swap]{n_f}  & & U(A) \\ & U(A_f) \arrow[]{ru}[swap]{U(u_f)} & 
     \end{tikzcd} & \mapsto &A_f \stackrel{u_f}{\rightarrow } A
\end{array} \]

We can easily prove this functor is left adjoint to $ U/A$, but it is more direct to observe that the family of candidates under $B$ is a multi-initial family in $B \downarrow U$. Hence $U$ is a local right adjoint. \\

For the converse, suppose $ U$ is a local right adjoint. We claim that candidates are arrows $ n : B \rightarrow U(A)$ such that $ L_A(n)$ provides an iso $ A_f \simeq A$ in $\mathcal{A}$, hence $ \eta^A_n \simeq n$ in $B \downarrow U$. Let be a square 
\[ 
\begin{tikzcd}
B \arrow[d, "n"'] \arrow[r, "g"] & U(A_1) \arrow[d, "U(u)"] \\
U(A) \arrow[r, "U(v)"']          & U(A_2)                  
\end{tikzcd}\]
Recall by \cref{BC} we also have that composing with $U(v)$ does not modify the unit, as we have an isomorphism $ \sigma^u_n : \eta^{A_2}_{U(v)n} \simeq \eta^{A}_f$. But the triangle 
\[ 
\begin{tikzcd}
B \arrow[r, "g"] \arrow[rd, "U(u)n"'] & U(A_1) \arrow[d, "U(u)"] \\
                                      & U(A_2)                  
\end{tikzcd} \]
provides us with a unique arrow $ w : A_{U(v)n} \rightarrow A_1$ such that 
\[ 
\begin{tikzcd}[column sep=huge]
                                                   & A_1 \arrow[d, "U(u)"] \\
A_{U(v)n} \arrow[r, "L_A(U(u)n)"'] \arrow[ru, "w"] & A_2                  
\end{tikzcd} \quad 
\begin{tikzcd}[column sep=huge]
B \arrow[d, "\eta^{A_2}_{U(v)n}"'] \arrow[r, "g"]                   & U(A_1) \arrow[d, "U(u)"] \\
U(A_{U(v)n}) \arrow[r, "U_{A_2}L_{A_2}(U(u)n)"'] \arrow[ru, "U(w)" description] & U(A_2)                  
\end{tikzcd} \]
Then by inserting the local inverses of $ \eta^A_n$ and $ \sigma^u_n$ in the square above and using the universal property of $ \eta^{A_n}_{\eta^A_n}$ at the triangle

\[ 
\begin{tikzcd}[row sep=large, column sep=huge]
B \arrow[d, "n"'] \arrow[rd, "\eta^{A}_n" description] \arrow[rrr, "g"] \arrow[rrd, "\eta^{A_n}_{\eta_f^A}" description] &                                                                                                            &                                                                                                                                            & U(A_1) \arrow[dddd, "U(u)"] \\
U(A) \arrow[d, equal] \arrow[rd, "U(L_A(n)^{-1})" description]                                                      & U(A_n) \arrow[l, "U(L_A(n))" description] \arrow[d, equal] \arrow[r, "U((\sigma^u_n)^{-1})" description] & U(A_{U(v)n}) \arrow[ld, "U(\sigma^u_n)" description] \arrow[d, equal] \arrow[ru, "U(w)" description, dashed]                             &                             \\
U(A)  \arrow[dd, equal] \arrow[rd, "U(L_A(n)^{-1})" description]                  &\arrow[l, "U(L_A(n))" description] U(A_n) \arrow[r, "U((\sigma^u_n)^{-1})" description] \arrow[d, equal]                                    & U(A_{U(v)n}) \arrow[ruu, "U(w)" description, dashed] \arrow[ld, "\sigma^u_n" description] \arrow[rdd, "U_{A_2}L_{A_2}(U(v)n)" description] &                             \\
                                                                                                                         & U(A_n) \arrow[ld, "U(L_A(n))" description]                                                            &                                                                                                                                            &                             \\
U(A) \arrow[rrr]                                                                                                         &                                                                                                            &                                                                                                                                            & U(A_2)                     
\end{tikzcd} \]
provides us with a triangle as below and a diagonalization in $\mathcal{B}$
\[ 
\begin{tikzcd}
                                                               & A_1 \arrow[d, "u"] \\
A \arrow[r, "v"'] \arrow[ru, "w (\sigma^u_f)^{-1}L_A(n)^{-1}"] & A_2               
\end{tikzcd} \]
In particular, local units are candidates by \cref{units are candidate}. Hence for any arrow $ f : B \rightarrow U(A)$, the factorization through the unit as $ f = U_AL_A(f) \eta^A_f$ provides a stable factorization through a candidate. 
\end{proof}

This achieve to prove that stable functors and local right adjoint can be used indifferently and are two ways of encapsulating the same property. \\

However in the following, and especially in the second paper, we will give more interest to right multi-adjoint for the smallness condition allows us to manipulate local units without size issue.

\section{Right multi-adjoints through free product completion}

In this section, we give the characterization of right multi-adjoints though the free product completion, following loosely \cite{diers1977categories} and \cite{tholen_1984}. In the second part of this work, we are going to show how the notion of spectrum is a way to turn a local right adjoint into a right adjoint, the spectrum functor being the desired left adjoint. But this construction, motivated by geometric and duality theoretic conceptions, process by extracting as much as possible topological and geometric information from a situation of local adjunction: in some way, it exploits the defect of universality on the algebraic side in order to produce richer structure on the geometric side. In this section, we recall another way to turn a situation of multi-adjunction into a honest adjunction, which is purely algebraic and purer in some sense, but also devoid of any geometric content for this very reason. The relation between those two approach will be studied in more detail in the second part of this work, where this approach through free product completion will appear as the ``discrete version" of the spectral adjunction. \\

The main intuition of this part is that, for a right multi-adjoint $ U : \mathcal{A} \rightarrow \mathcal{B}$, the cone of local units under a given object $B$ defines a family of objects in $\mathcal{A}$ given by the codomains of those local units. Hence, at the level of families of objects, the multiversality of the construction can be fixed and $U$ will induce an honest adjunction between categories of families of objects of $\mathcal{A}$ and $\mathcal{B}$. The good notion of ``category of families" here is the one provided by the free product completion, the beginning of this part is devoted to.

\begin{definition}
For a category $\mathcal{A}$, the \emph{free product completion} of $\mathcal{A}$ is the category $ \Pi\mathcal{A}$ whose \begin{itemize}
    \item objects are functors $ A_{(-)} : I \rightarrow \mathcal{A}$ (also denoted $ (A_i)_{i \in I})$ with $ I$ a set,
    \item and arrows $(A_i)_{i \in I} \rightarrow (B_j)_{j \in J} $ consist of the data of an application $ \alpha : J \rightarrow I$ and a natural transformation  
    \[ 
\begin{tikzcd}[row sep=small, column sep=large]
I \arrow[rd, "A_{(-)}", ""{name=U, below, inner sep=0.1pt}]                       &             \\
                                              & \mathcal{A} \\
J \arrow[ru, "B_{(-)}"', ""{name=D, inner sep=0.1pt}] \arrow[uu, "\alpha"] &       \arrow[from=U, to=D, Rightarrow, "f"', bend right=20]     
\end{tikzcd} \]
that is, a $J$-indexed family $ (f_j : A_{\alpha(j)} \rightarrow B_j)_{j \in J}$. 
\end{itemize}
\end{definition}

\begin{proposition}
We have the following properties of the free product completion, for a given category $ \mathcal{A}$:\begin{itemize}
    \item $\Pi \mathcal{A} $ has small products
    \item There is a codense full embedding $ \iota_\mathcal{A} : \mathcal{A} \hookrightarrow \Pi\mathcal{A}$ whose essential image is the subcategory $ (\Pi \mathcal{A})_{coco}$ of co-connected objects.
    \item Moreover, the embedding $ \mathcal{A} \rightarrow \Pi \mathcal{A}$ has a right adjoint if and only if $\mathcal{A}$ already had products 
    \item We have a full embedding $ \Pi \mathcal{A} \hookrightarrow [\mathcal{A}, \mathcal{S}et]^{op}$ whose essential image consists of all small products of representable
    \item For any category $ \mathcal{B}$ with small products, we have an equivalence of categories 
    \[ [\mathcal{A}, \mathcal{B}] \simeq  [\Pi \mathcal{A}, \mathcal{B}]_{\Pi} \]
    (where $  [\Pi \mathcal{A}, \mathcal{B}]_{\Pi}$ is the category of functors preserving small products) sending any $ F: \mathcal{A} \rightarrow \mathcal{B}$ on its right Kan extension $ \ran_{\iota_\mathcal{A}} F  $ and any $ G : \Pi \mathcal{A} \rightarrow \mathcal{B}$ on its restriction $ G \iota_\mathcal{A}$.
\end{itemize}
\end{proposition}

\begin{proof}
 For the first item: the product in $ \Pi \mathcal{A}$  of a family of families $ ((A^j_i)_{i \in I_j})_{j \in J}$ has as indexing set the disjoint union $ \coprod_{j \in J}I_j$ and whose member of index $ (j,i)$ is the object $ A_{(j,i)} = A^i_j$; the projections are given for each $ j \in J$ as the transformation 
 \[ 
\begin{tikzcd}[column sep=huge]
I_j \arrow[rrd, "(A^j_i)_{ i \in I_j}", ""{name= U, inner sep=0.1pt, below}] \arrow[d, "q_j"']                                                &&             \\
\underset{j \in J}{\coprod} I_j \arrow[rr, "{(A_{(j,i)})_{(j,i) \in \underset{j \in J}{\coprod} I_j}}"', ""{name=D, near start, inner sep=0.1pt}] && \mathcal{A} \arrow[from=D, to=U, Rightarrow, "p_j", bend left=15] 
\end{tikzcd} \]
where $ p_i $ is the pointwise equality $ A_{(j,i)} = A^j_i $. \\

For the second item, the embedding sends an object $ A$ of $\mathcal{A}$ to the one element family $ A : 1 \rightarrow \mathcal{A}$ and a morphism $ f : A_1 \rightarrow A_2 $ to the natural transformation 
\[ 
\begin{tikzcd}
1 \arrow[rr, "A_2"', bend right, ""{name=D, inner sep=0.1pt}] \arrow[rr, "A_1", bend left, ""{name=U, below, inner sep=0.1pt}] &  & \mathcal{A} \arrow[Rightarrow, from =U, to=D, "f"']
\end{tikzcd} \]
Now we prove objects in the image of this embedding are coconnected, which says that for a family of families $ ((A^j_i)_{i \in I_j})_{j \in J}$, we have 
\[ \Pi\mathcal{A}[ \prod_{j \in J} (A^j_i)_{i \in I_j}, \iota_\mathcal{A}(A)] \simeq \coprod_{j \in J} \Pi \mathcal{A}[(A^j_i)_{i \in I}, \iota_\mathcal{A}(A)] \]
Indeed, any arrow $\Pi_{j \in J} (A^j_i)_{i \in I_j} \rightarrow  \iota_\mathcal{A}(A) $ defines an arrow $ 1 \rightarrow \coprod_{j \in J}I_j $ pointing at some pair $ (j,i)$ with $j \in J$ and $i \in I_j$, and a natural transformation 
\[ 
\begin{tikzcd}
                                                                                                       & 1 \arrow[rdd, "A", bend left] \arrow[ldd, "{(j,i)}"', bend right] \arrow[d, dashed] &             \\
                                                                                                       & I_j \arrow[rd, "(A_i)_{ i \in I_j}" description] \arrow[ld, "q_j" description]      &             \\
\underset{j \in J}{\coprod} I_j \arrow[rr, "{(A_{(j,i)})_{(j,i) \in \underset{j \in J}{\coprod} I_j}}"'] &                                                                                     & \mathcal{A}
\end{tikzcd} \]
which is nothing but an arrow $ f : A_{(j,i)} \rightarrow A$. But in such a case, as $ 1$ is a connected object in $\mathcal{S}et$, this arrow $ (j,i) : 1 \rightarrow \coprod_{j \in J}I_j$ factorizes through $ I_j$ for some $j \in J$, pointing the corresponding $i \in I_j$, and the natural transformation $ f $ factorizes through the componentwise identity $ p^j : A_{(i,j)} = A^j_i$ so we have an arrow $ (A^j_i)_{i \in I_j} \rightarrow \iota_\mathcal{A}(A) $ as desired. Conversely, any coconnected object is of the form $ \iota_\mathcal{A}(A)$: indeed a family $ (A_i)_{i \in I} : I \rightarrow \mathcal{A}$ is nothing but the product in $\Pi\mathcal{A}$ of the family $(\iota_\mathcal{A}(A_i))_{i \in I}$ as the set $ I $ decomposes as the coproduct $ \coprod_I 1$ in $\mathcal{S}et$; and any family should be indexed by a connected set to be a coconnected object in $ \Pi\mathcal{A}$, but $1$ is the only connected set. This also suffice to prove that any object is a product of objects in the range of $ \iota_A$.  \\

Now suppose that $\mathcal{A}$ has products. Then for any family $ (A_i)_{i \in I}$ in $ \Pi\mathcal{A}$ one can compute the product in $\mathcal{A}$, $ \prod_{i \in I}A_i$. Now for an object $ A $ in $ \mathcal{A}$, we have 
\[ \Pi\mathcal{A}[\iota_\mathcal{A}(A), (A_i)_{i \in I} ] \simeq \mathcal{A}[A, \prod_{i \in I} A_i] \]
sending a family of arrows $ (A \rightarrow A_i)_{i \in I}$ to the universal map $ A \rightarrow \prod_{i \in I}$. The unit of this adjunction is iso as $ \iota_\mathcal{A}$ is full and faithful, while the counit is the transformation
\[ 
\begin{tikzcd}[column sep=large]
* \arrow[rrd, "\underset{i \in I}\prod A_i",""{name=U, inner sep=0.1pt, below}]    &  &             \\
I \arrow[rr, "(A_i)_{i \in I}"', ""{name=D, inner sep=0.1pt}] \arrow[u, "!"] &  & \mathcal{A} \arrow[Rightarrow, from=U, to=D, "\epsilon_{(A_i)_{i \in I}}"']
\end{tikzcd} \]
where $ \epsilon_{(A_i)_{i \in I}}$ has the projection $p_i : \prod_{i \in I}A_i \rightarrow A_i$ has component in $i$. For the converse, it is easy to see that any right adjoint of the embedding $ \iota_\mathcal{A}$ sends a family on an object in $\mathcal{A}$ with the universal property of the product.  \\

The embedding $ \Pi\mathcal{A} \hookrightarrow [\mathcal{A}, \mathcal{S}et]^{op}$ justs sends a family $(A_i)_{i \in I} $ to the coproduct $ \coprod_{i \in I} \hirayo^*_{A_i}$ and an arrow $ (\alpha, (f_j)_{j \in J}) : (A_i)_{i \in I} \rightarrow (B_j)_{j \in J}$ to the opposite of the induced map between the corresponding coproducts in $[\mathcal{A}, \mathcal{S}et]$ as depicted below
\[ 
\begin{tikzcd}
\hirayo_{A_{\alpha(j)}}^* \arrow[r, "q_{\alpha(j)}"]              & \underset{i \in I}{\coprod} \hirayo^*_{A_i}                                                                              \\
\hirayo_{B_j}^* \arrow[u, "\hirayo^*_{f_j}"] \arrow[r, "q'_{j}"'] & \underset{j \in J}{\coprod}\hirayo^*_{B_j} \arrow[u, "\langle q_{\alpha(j)} \hirayo^*_{f_j} \rangle_{j \in J}"', dashed]
\end{tikzcd}\]

Finally, for a functor $U : \mathcal{A} \rightarrow  \mathcal{B}$ with $ \mathcal{B}$ having products; we claim that the right Kan extensions of $U$ is pointwise and can be computed as \[ \ran_{\iota_{\mathcal{A}}}U (A_i)_{i \in I} = \prod_{i \in I}U(A_i) \]
Indeed for any $(A_i)_{i \in I}$ the comma category $(A_i)_{i \in I}\downarrow \iota_\mathcal{A} $ has a small initial $I$-indexed subcategory consisting of the objects $(i, 1_{A_i})$ for $ i \in I$, and this subcategory is discrete. Hence calculating the poinwise right Kan extension resumes to calculating the product above. Moreover, as $\iota_\mathcal{A} $ is full and faithful, restricting back this Kan extension along $ \iota_\mathcal{A}$ gives back $U$, in fact up to equality in this context. 
\end{proof}

\begin{proposition}
The embedding $ \mathcal{A} \hookrightarrow \Pi\mathcal{A}$ creates connected limits in $ \mathcal{A}$. Moreover, $ \Pi \mathcal{A}$ is complete if and only if $ \mathcal{A}$ has connected limits.
\end{proposition}

\begin{proof}
 Let be $ D$ a connected category and $ F : D \rightarrow \mathcal{A}$; we prove that the singleton $ \iota_\mathcal{A}(\lim \, F)$ is the limit of $\iota_\mathcal{A} F$ in $\Pi \mathcal{A}$. Let be a cone $ (\alpha_d, (f_d)_{d \in D} : (A_i)_{i \in I} \iota_\mathcal{A}F$ in $\Pi\mathcal{A}$ consisting for each $d \in D$ of an arrow $ f_d : A_{\alpha_d} \rightarrow F(d)$ where $ \alpha_d : 1 \rightarrow I$ points to some index; but as $ D$ is connected and $ I$ is a set, necessarily the $ \alpha_d$ are all equal to the same index $\alpha$, so that we actually have a cone $( f_d : A_\alpha \rightarrow F(d))_{d \in D}$ in $\mathcal{A}$, inducing a unique arrow $ f : A_\alpha \rightarrow \lim \, F$ in $\mathcal{A}$. This defines a unique arrow $ (\alpha, f) : (A_i)_{i \in I} \rightarrow \iota_ \mathcal{A}(\lim\, F)$. By what precedes, it is also clear that any connected cone that $ \iota_\mathcal{A}$ sends to a limiting cone was already limiting.\\ 
 
 Now, recall that a category is complete if and only if it has connected limits and products. But $\Pi \mathcal{A}$ always has products, so we just have to show that $ \Pi\mathcal{A}$ has connected limits if and only if $ \mathcal{A}$ does. Let be $ {D}$ a connected category, and $ F : {D} \rightarrow \Pi\mathcal{A}$ a functor, with $ F_d : I_d \rightarrow \mathcal{A}$ its component in $d$ and with transition morphism
 \[ 
\begin{tikzcd}[row sep=small]
I_{d_1} \arrow[rrd, "F_{d_1}", ""{name=U, inner sep=0.1pt, below}]                    &  &             \\
                                                  &  & \mathcal{A} \\
I_{d_2} \arrow[uu, "I_s"] \arrow[rru, "F_{d_2}"', ""{name=D, inner sep=0.1pt}] &  & \arrow[Rightarrow, from=U, to=D, "F_s"']           
\end{tikzcd} \]
 for each $ s : d_1 \rightarrow d_2$ in $D$. Then $F$ defines an oplax cocone $ (F_d : I_d \rightarrow \mathcal{A})_{d \in D}$ in $ \mathcal{C}at$, defining uniquely a functor 
 \[ \int I \stackrel{\langle F_d \rangle_{d \in D}}{\longrightarrow } \mathcal{A} \] 
 where $ \int I$ is the category of elements of the functor $ I : D^{op} \rightarrow \mathcal{S}et$ returning the indexing set $I_d : I_{d_2} \rightarrow I_{d_1}$ and the transition map $ I_s$ for $s : d_1 \rightarrow d_2$: it is indeed well known that the category of elements is the oplax colimit in $ \mathcal{C}at$, and we see here the $ I_d$ as discrete categories. Now, as $D$ was small and each $ I_d$ was a set, the category $ \int I$ has a small set $ \pi_0 (\int I) $ of connected components. In this context, one can describe the connected components as follows. In set, the colimit of the diagram $I$ is the quotient 
 \[ \underset{d \in D}\colim \, I_d \simeq \coprod_{d \in D} I_d /\sim_D  \]
 where $ (d,i) \sim_D (d',i')$ if there is a zigzag in $D$ relating $i $ and $i'$: this exactly amounts to say that $ (d,i) $ and $ (d',i')$ are in the same connected component of $\int I$, so we also have that the connected components of $\int I$ are exactly equivalence classes $ [(d,i)]_{\sim_D}$ and \[\coprod_{d \in D} I_d /\sim_D \simeq \pi_0(\int I )  \] 
 Now, if we restrict the induced functor $ \langle F _d\rangle_{d \in D}$ along the inclusion of a connected component 
 \[ 
\begin{tikzcd}[row sep=small]
\int I \arrow[rrd, "\langle F _d\rangle_{d \in D}"]                                                  &  &             \\
                                                                                                     &  & \mathcal{A} \\
{[(d,i)]_{\sim_D}} \arrow[uu, "{i_{[(d,i)]_{\sim_D}}}", hook] \arrow[rru, "{F_{[(d,i)]_{\sim_D}}}"'] &  &            
\end{tikzcd} \]
we can compute the limit $ \lim F_{[(d,i)]_{\sim_D}}$ in $\mathcal{A}$, and this limit is preserved by the inclusion functor $ \iota_\mathcal{A}$. So the desired limit of $F$ in $\Pi\mathcal{A}$ is the family 
\[ \pi_0(\int I) \rightarrow \mathcal{A}  \]
sending the connected component $ [(d,i)]_{\sim_D}$ to the connected limit $ \lim F_{[(d,i)]_{\sim_D}} $, and this actually coincides with the product in $ \Pi \mathcal{A}$ of the family $ (\lim F_{[(d,i)]_{\sim_D}} : 1 \rightarrow \mathcal{A})_{[(d,i)]_{\sim_D} \in \pi_0(\int I)} $. \\
\end{proof}

\begin{proposition}
Any functor $ U : \mathcal{A} \rightarrow \mathcal{B}$ extends uniquely into a functor $ \Pi U $, called its \emph{free product extension}, making the square below to commute up to equality
\[ 
\begin{tikzcd}
\mathcal{A} \arrow[d, "\iota_\mathcal{A}"', hook] \arrow[r, "U"] & \mathcal{B} \arrow[d, "\iota_\mathcal{B}", hook] \\
\Pi\mathcal{A} \arrow[r, "\Pi U"']                               & \Pi\mathcal{B}                                  
\end{tikzcd}\]
\end{proposition}

\begin{proof}
 The functor $\Pi U$ just is the right Kan extension $ \ran_{\iota_\mathcal{A}} \iota_\mathcal{B} U$, and is defined by sending a family $ (A_i)_{i \in I}$ to $ (U(A_i))_{i \in I}$. 
\end{proof}

The following proposition is tautological:

\begin{proposition}
$\mathcal{A}$ has a multi-initial family if and only if $ \Pi \mathcal{A}$ has an initial object. 
\end{proposition}

The following proposition is the core idea of \cite{diers1977categories}[part 4], though we present here a quite different proof. 

\begin{proposition}
For a functor $ U : \mathcal{A} \rightarrow \mathcal{B}$, the following are equivalent:\begin{enumerate}
    \item $U$ is a right multi-adjoint
    \item $U $ has a relative left adjoint along $ \iota_\mathcal{A}$
    \item its free product extension $\Pi U : \Pi\mathcal{A} \rightarrow  \Pi\mathcal{B}$ is a right adjoint
\end{enumerate}  
\end{proposition}

\begin{proof}
 Suppose that $ U$ is a right multi-adjoint, with $I_B$ the set of local units $ \eta_x : B \rightarrow U(A_x)$ and $ \pi_B : I_B \rightarrow \mathcal{A}$ its projection sending $ x $ to $ A_x$. Define a functor $ L : \mathcal{B} \rightarrow \Pi \mathcal{A}$ sending an object $B$ to the family $\pi_B : I_B \rightarrow \mathcal{A} $, and an arrow $ f : B_1 \rightarrow B_2  $ to the transformation 
 \[ 
\begin{tikzcd}[row sep=small]
I_{B_1} \arrow[rrd, "\pi_{B_1}", ""{name=U, below, inner sep=0.1pt}]                    &  &             \\
                                                    &  & \mathcal{A} \\
I_{B_2} \arrow[uu, "I_f"] \arrow[rru, "\pi_{b_2}"', ""{name=D, inner sep=0.1pt}] &  &  \arrow[Rightarrow, from=U, to=D, "L_f"]          
\end{tikzcd} \]
where $I_f$ sends $x$ to the index of the unit $ \eta^{A_x}_{n_x f} : B \rightarrow U(A_{n_x f}) = n_{I_f(x)}$ and $ L_f $ has component $ L_{A_x}(n_x f) :A_{n_x f} \rightarrow A_x$ as provided in each $ x \in I_{B_2}$ by the factorization  
\[ 
\begin{tikzcd}[column sep=huge]
B_1 \arrow[d, "\eta^{A_x}_{n_xf}"'] \arrow[r, "f"] & B_2 \arrow[d, "n_x"] \\
U(A_{n_xf}) \arrow[r, "U_{A_x}L_{A_x}(n_xf)"']     & U(A_x)              
\end{tikzcd} \]
Observe that the local units of $B$ define in particular a morphism of families 
\[ 
\begin{tikzcd}
1 \arrow[r, "B", ""{name=U, inner sep=0.1pt, below, near start}]                       & \mathcal{B}                 \\
I_B \arrow[r, "\pi_B"'] \arrow[u, "!"] & \mathcal{A} \arrow[u, "U"', ""{name=D, inner sep=0.1pt, below, near start}] \arrow[Rightarrow, from= U, to=D, "n"', bend right=35]
\end{tikzcd} \]
corresponding to the family $ (n_x : B \rightarrow U(A_x))_{x \in I_B}$.
Now it is easy to see that this functor is a relative left adjoint to $U$ along $ \iota_\mathcal{A}$, that is, that for any $B$ in $\mathcal{B}$ and $ A$ in $\mathcal{A}$ we have 
\[ \Pi\mathcal{A}[L(B), \iota_\mathcal{A} (A)] \simeq \mathcal{B}[B, U(A)] \]
Indeed, any arrow $f : B \rightarrow U(A)$ factorizes through a unique $ n_x : B \rightarrow U(A_x)$, where $x$ is the index of the unit $ \eta^{A}_f$, while $ L_A(f) : A_x \rightarrow A$  provides a morphism in $\Pi\mathcal{A}$  
\[ 
\begin{tikzcd}[ row sep=small ]
I_B \arrow[rrd, "\pi_B", ""{name=U, below, inner sep=0.1pt}]            &  &             \\
                                    &  & \mathcal{A} \\
1 \arrow[rru, "A"', ""{name=D, inner sep=0.1pt}] \arrow[uu, "x"] &  & \arrow[Rightarrow, from=U, to=D, "L_A(f)"']           
\end{tikzcd} \]\\
while any arrow $ (x, u) : L(B) \rightarrow \iota_\mathcal{A}(A)$ can be pasted with the family of units 
\[ 
\begin{tikzcd}[column sep=large]
1 \arrow[rr, "B",""{name=U, inner sep=0.1pt, below}]                      && \mathcal{B}                 \\
I_B \arrow[rr, "\pi_B", ""{name=A, below, inner sep=0.1pt}] \arrow[u, "!"] && \mathcal{A} \arrow[u, "U"', ""{name=D, inner sep=0.1pt, below, near start}] \\
1 \arrow[u, "x"] \arrow[rru, "A"', ""{name=B, inner sep=0.1pt}]     &&    \arrow[Rightarrow, from= U, to=D, "n"', bend right=20]                 \arrow[Rightarrow, from=A, to=B, "L_A(f)"'{near end}]       
\end{tikzcd} \]
to return an arrow $ B \rightarrow U(A)$, and one just has to use the universal properties of the local units to see that those processes are mutual inverses.\\

This functor $ L$ extends to $\Pi\mathcal{B}$ as follows: for a family $ (B_i)_{i \in I}$, define $ L(B_i)_{i \in I}$ as the family 
\[  \coprod_{i \in I} I_{B_i} \stackrel{ \langle \pi_{B_i} \rangle_{i \in I} }{\longrightarrow} \mathcal{A}  \]
sending $ (i,x)$ with $ i \in I$ and $ x \in I_{B_i}$ to the associated $ A_x$; for an arrow $ (\alpha, f=(f_i)_{i \in I}) : (B_i)_{i \in I_1} \rightarrow (B'_i)_{i \in I_2} $, that is a family $ (f_i : B_{\alpha(i)} \rightarrow B'_i)_{i \in I_2}$, each pair $ (i,x) \in \coprod_{i \in I_2}I_{B'_i} $ defines uniquely some $ I_f(x) \in I_{B_{\alpha(i)}}$ indexing the unit through which factorizes the composite $ n_x f_i : B_{\alpha(i)} \rightarrow U(A_x)$, that is such that 
\[ 
\begin{tikzcd}[column sep=large]
B_{\alpha(i)} \arrow[d, "{n_{I_{(\alpha, f)}(x)}}"'] \arrow[r, "f_i"] & B'_i \arrow[d, "n_x"] \\
U(A_{I_f(x)}) \arrow[r, "U(L_{A_x}(n_xf))"']                                      & U(A_x)               
\end{tikzcd} \]
and define $ I_{(\alpha,f)} : \coprod_{i \in I_2}I_{B'_i} \rightarrow \coprod_{i \in I}I_{B_i}$ as sending $ (i,x)$ to this $ I_{(\alpha,f)}(x)$, and define the desired morphism $L(\alpha, f)$ as 
\[ 
\begin{tikzcd}[row sep=small]
\underset{i \in I_1}\coprod I_{B_i} \arrow[rrd, "\langle \pi_{B_i} \rangle_{i \in I_1} ", ""{name=U, below, inner sep=0.1pt}]                                    &  &             \\
                                                                                                                             &  & \mathcal{A} \\
\underset{i \in I_2}\coprod I_{B'_i} \arrow[rru, "\langle \pi_{B'_i} \rangle_{i \in I_2} "', ""{name=D, inner sep=0.1pt}] \arrow[uu, "{I_{(\alpha, f)}}"] &  &  \arrow[Rightarrow, from=U, to=D, " L_f"']          
\end{tikzcd} \]
where $ L_f$ denotes the family $ (L_{A_x}(n_xf) : A_{I_{(\alpha, f)}} \rightarrow A_x)_{(i,x) \in \coprod_{i \in I_2}I_{B'_i}}$. In particular we have a morphism 
\[ \begin{tikzcd}
I \arrow[r, "(B_i)_{i \in I}", ""{name=U, inner sep=0.1pt, below, near start}]                       & \mathcal{B}                 \\
\underset{i \in I}{\coprod}I_{B_i} \arrow[r, "\pi_{(B_i)_{i \in I}}"'] \arrow[u, "\pi_I"] & \mathcal{A} \arrow[u, "U"', ""{name=D, inner sep=0.1pt, below, near start}] \arrow[Rightarrow, from= U, to=D, "n"', bend right=35]
\end{tikzcd} \]
where $ \pi_I :\coprod_{i \in I}I_{B_i} \rightarrow I $ is the projection sending $ (i,x)$ on $i \in I$, $\pi_{(B_i)_{i \in I}} : \coprod_{i \in I}I_{B_i} \rightarrow \mathcal{A}$ sends $ (i,x)$ on $ A_x$, and $ n$ has as components \[( n_{(i,x)}= n_x : B_i \rightarrow U(A_x))_{(i,x) \in \underset{i \in I}\coprod I_{B_i}}\]
Then for any $ (A_j)_{j \in J} $ in $\Pi\mathcal{A}$ and $ (B_i)_{i \in I}$ in $ \Pi\mathcal{B}$ we have an isomorphism 
\[ \Pi \mathcal{A}[ L(B_i)_{i \in I}, (A_j)_{j \in J}] \simeq \Pi \mathcal{B}[(B_i)_{i \in I}, (U(A_j))_{j \in J}] \]
Indeed a morphism of family $ (\alpha, f) :(B_i)_{i \in I} \rightarrow (U(A_j))_{j \in J} $, that is a family $ (f_j : B_{\alpha(j)} \rightarrow U(A_j))_{j \in J}$, defines an application $ \xi : J \rightarrow \coprod_{i \in I} I_{B_i} $ sending $ i$ to the index of the local unit $ n_{\xi(i)} = \eta^{A_i}_{f_i} : B_{\alpha(i)} \rightarrow A_{\xi(i)}$, and a morphism of families 
\[ 
\begin{tikzcd}[row sep=small, column sep=huge]
\underset{i \in I}\coprod I_{B_i} \arrow[rrd, "\pi_{(B_i)_{i \in I}}", ""{name=U, near start, below, inner sep=0.1pt}] &  &             \\
                                                                       &  & \mathcal{A} \\
J \arrow[uu, "\xi"] \arrow[rru, "(A_j)_{j \in J}"', ""{name=D, near start, inner sep=0.1pt}]                    &  &        \arrow[Rightarrow, from=U, to=D, "L_{(A_j)_{j \in J}}(f_j)_{j \in J}" description]    
\end{tikzcd} \]
where $L_{(A_j)_{j \in J}}(f_j)_{j \in J}$ consists of all right part in $\mathcal{A}$ of their factorization \[ (L_{A_j}(f_j) : A_{\xi(j)} \rightarrow A_j)_{j \in J} \]
For the converse, use the same argument as for the proof of the first implication by pasting. \\

Now we prove that if $ U$ is such that $ \Pi U$ is right adjoint to a functor $L$, then $ U$ is multi-adjoint. Observe that with this hypothesis, we have in particular for any $B$ in $\mathcal{B}$ a unit $ \eta_{(*,B)} : (*,B) \rightarrow \Pi U L(*,B)$. So if $I_B$ denote the indexing set of $\Pi U L(*,B)$ and $ A_i$ is the object in $\mathcal{A}$ corresponding to the $i$th index of $I_B$ in this family, then we have a cone $ (\eta_i : B \rightarrow U(A_i))_{i \in I_B}$. Now we prove this is a cone of local units. For any $ A$ in $\mathcal{A}$, the unit property of $ \Pi UL(*,B)$ provides a factorization
\[\begin{tikzcd}[column sep=large]
	{*} && {\mathcal{B}} \\
	& {I_B} & {} \\
	{*} && {\mathcal{A}}
	\arrow["{U}"', from=3-3, to=1-3]
	\arrow["{B}"{name=0}, from=1-1, to=1-3]
	\arrow["{A}"{name=1, swap}, from=3-1, to=3-3]
	\arrow[Rightarrow, from=1-1, to=3-1, no head]
	\arrow[from=1-1, to=2-2]
	\arrow["{i_f}" description, from=3-1, to=2-2, dashed]
	\arrow["{(A_i)_{i \in I_B}}"{name=2, description}, from=2-2, to=3-3]
	\arrow[Rightarrow, "{n}"', from=0, to=2-3, shift left=2, curve={height=12pt}, shorten <=8pt, shorten >=10pt]
	\arrow[Rightarrow, "{f}"', from=0, to=1, shift right=10, shorten <=6pt, shorten >=6pt, crossing over]
	\arrow[Rightarrow, "{L_A(f)}"', from=2, to=1, shorten <=4pt, shorten >=4pt, dashed]
\end{tikzcd}\] 
for $i_f : * \rightarrow I_B$ pointing at the index of the local unit $ \eta^A_f$ and $ L_A(f)$ returning the image of $f$ along the local left adjoint at $A$. 
\end{proof}

\section{Factorization aspects}

As suggested by the definition of candidates in the notion of stability, orthogonality structures are hidden in the notion of local adjunction. In the same vein, one could ask whether the stable factorization of arrows toward $U$ can be generalized to any arrow, that is, if the orthogonality structure provided by the candidates on the left and the morphisms in the range of $U$ on the right can be completed into a factorization system. In the context studied in \cite{Diers}, this is possible through a small object argument in the context of locally finitely presentable category. This step is essential in general in the construction of spectra, and also takes place in the topos-theoretic approach of \cite{Coste}, though it is mostly left implicit. The reference for this is \cite{Anel}, we mostly follow there modulo some adaptations, and in combination with elements from \cite{Diers}. \\

\begin{definition}
Let $ U : \mathcal{A} \rightarrow \mathcal{B}$ a local right adjoint. A morphism $n : B \rightarrow C$ is said to be \emph{ diagonally universal } if it is left orthogonal to morphisms in the range of $U$, that is, if for any morphism $ u : A_1 \rightarrow A_2$ in $\mathcal{A}$ and any square as below, there exists a unique filler $d : C \rightarrow U(A_1)$ making both the upper and lower triangles to commute
\[ 
\begin{tikzcd}
B \arrow[d, "n"'] \arrow[r, "f"]              & U(A_1) \arrow[d, "U(u)"] \\
C \arrow[r, "g"'] \arrow[ru, "d" description] & U(A_2)                  
\end{tikzcd} \]
\end{definition}

As a left class in an orthogonality structure, diagonally universal morphisms enjoy the following properties, which are standards and then do not need proofs here.

\begin{proposition}
We have the following:\begin{itemize}
    \item diagonally universal morphisms are stable under composition and contain isomorphisms
    \item if $ n : B \rightarrow C$ is diagonally universal and $ m : C \rightarrow D$ is such that $ mn $ is diagonally universal, then $m$ is also diagonally universal. In particular, diagonally universal morphisms are stable under retracts.
    \item diagonally universal morphisms are stable under pushout along arbitrary morphisms
    \item diagonally universal morphisms are stable under colimits and retracts in the arrow category $ \mathcal{B}^2$ 
\end{itemize}
\end{proposition}

\begin{remark}
Beware that without further assumption, a diagonally universal morphism with codomain in the range of $U$, that is of the form $ n : B \rightarrow U(A)$, needs not to be a candidate, as the filler needs not to arise from a morphism in $\mathcal{A}$. 
\end{remark}

\begin{definition}\label{relff}
A functor $ U : \mathcal{A} \rightarrow \mathcal{B}$ is said to be \emph{relatively full and faithful} if for any triangle as below 
\[\begin{tikzcd}
	{U(A_1)} && {U(A_2)} \\
	& {U(A)}
	\arrow["{U(u_1)}"', from=1-1, to=2-2]
	\arrow["{U(u_2)}", from=1-3, to=2-2]
	\arrow["f", from=1-1, to=1-3]
\end{tikzcd}\]
then $ f$ comes uniquely from some $ u : A_1 \rightarrow A_2$ such that $ U(u)=f$.
\end{definition}

\begin{proposition}
For any $ u : A_1 \rightarrow A_2$, $ U(u)$ is diagonally universal if and only if $U(u)$ is an isomorphism. If moreover $ U$ is relatively full and faithful, then $U(u)$ is diagonally universal if and only if $u$ is an isomorphism.
\end{proposition}

\begin{proof}
The unique filler of the square 
 \[ 
\begin{tikzcd}
U(A_1) \arrow[d, "U(u)"'] \arrow[r, equal]                  & U(A_1) \arrow[d, "U(u)"] \\
U(A_2) \arrow[r, equal] \arrow[ru, "d" description, dashed] & U(A_2)                  
\end{tikzcd}\]
is both a right and left inverse to $ U(u)$; and if moreover $U$ is relatively full and faithful, it comes from a unique morphism $ d = U(v)$ which is both a section of $ u$ from the lower triangle, but it is also a retraction because there is a unique morphism in $\mathcal{A}$ lifting $U(u)$ in the upper triangle, and this must be $u$. 
\end{proof}

\begin{proposition}
A morphism $ n : B \rightarrow U(A)$ is diagonally universal if and only if $ U_AL_A(n)$ is an isomorphism. 
\end{proposition}

\begin{proof}
The unique lifter $d$ in the following square 
\[ 
\begin{tikzcd}
B \arrow[d, "n"'] \arrow[r, "\eta^A_n"]                        & U(A_n) \arrow[d, "U_AL_A(n)"] \\
U(A) \arrow[r, equal] \arrow[ru, "d" description, dashed] & U(A)                         
\end{tikzcd} \]
 is a section of $ U_AL_A(n)$. This provide also a filler of the square
 \[ 
\begin{tikzcd}
B \arrow[d, "\eta^A_n"'] \arrow[r, "\eta^A_n"]        & U(A_n) \arrow[d, "U_AL_A(n)"] \\
U(A_n) \arrow[r, "U_AL_A(n)"'] \arrow[ru, "d U_AL_A(n)" description] & U(A)                         
\end{tikzcd} \] 
But $ 1_{U(A)} = U(1_A)$ is the only filler of the square because $ \eta^A_n$ is a candidate. So $ d$ is also a retraction of $U_AL_A$, which is hence an isomorphism. Conversely, if $ U_AL_A(n)$ has an inverse, then one can use the candidate property at $ \eta^A_f$ to get a filler in any square with a morphism in the range of $U$ on the right. 
\end{proof}

\begin{remark}
Beware that $U$ needs not be conservative, so that the inverse of $ U_AL_A(n)$ needs not comes from a morphism in $\mathcal{A}$ making $ L_A(f)$ an isomorphism itself, so that the remark above does not says that $ n \simeq \eta^A_n$ in $ B \downarrow U$; in particular $n : B \rightarrow U(A)$ may be diagonally universal without being a candidate. However, in case where $ U$ is relatively full and faithful, the filler we had above must come from a unique morphism $ d = U(v)$ which must satisfies the same commutations, hence provides an inverse of $L_A(n)$: hence the following corollary.
\end{remark}

\begin{corollary}\label{relffdiaguniv}
Suppose that $ U$ is relatively full and faithful; then for a morphism $ n : B \rightarrow U(A)$, the following are equivalent:\begin{itemize}
    \item $n$ is diagonally universal 
    \item $n$ is a candidate 
    \item $ L_A(n)$ is an isomorphism
\end{itemize}
\end{corollary}

We defined $ \mathcal{D}iag$ as the left orthogonal $ ^\perp\!U({\mathcal{A}}^2)$. Hence we end with an orthogonality strucutre $ (\mathcal{D}iag, \mathcal{D}iag^\perp)$ where $\mathcal{D}iag^\perp$ is the double-orthogonal $ (^\perp\!U({\mathcal{A}}^2))^{\perp}$. Arrows in $\mathcal{D}iag^\perp$ lies now out of the essential image of $U$ and may have arbitrary domain and codomain. However, we have the following fullness property of the essential image of $U$ relatively to $ \mathcal{D}iag^\perp$:

\begin{proposition}
Let be $ u : U(A_1) \rightarrow U(A_2)$ be an arrow in $ \mathcal{D}iag^\perp$. Then $u \simeq U(L_{A_2}(u))$ in $\mathcal{B}/U(A_2)$ and $ \eta^{A_2}_u$ is an isomorphism. In particular $u$ is an arrow in the essential image of $U$.
\end{proposition}

\begin{proof}
Indeed, $u$ is right orthogonal to the local unit in its stable factorization, so there exists a unique $w$ as below
\[ 
\begin{tikzcd}[column sep=large]
U(A_1) \arrow[r, equal] \arrow[d, "\eta_u^{A_2}"']      & U(A_1) \arrow[d, "u"] \\
U(A_u) \arrow[r, "U(L_{A_2}(u))"'] \arrow[ru, "w" description] & U(A_2)               
\end{tikzcd} \]
which is both diagonally universal by left cancellation, and in $\mathcal{D}iag^\perp$ by right cancellation, and is hence an isomorphism. In particular $ \eta^{A_2}_u$ is an iso, being section of an iso.
\end{proof}

Now, we explain how, in a suitable context, the stable factorization of morphisms towards $U$ extends to a factorization system in $\mathcal{B}$, where the diagonally universal morphisms form the left class. To do so, we are going to adapt \cite{Anel} version of the small object argument in the context of locally presentable categories. \\

We recall first some general properties of the left and right classes of a factorization system:

\begin{proposition}
For a factorization system $(\mathcal{L}, \mathcal{R})$: \begin{multicols}{2}
 \begin{itemize}
    \item $\mathcal{L}$ contains all isomorphisms and is closed under composition,
    \item $\mathcal{L}$ is right-cancellative: if one has 
\[\begin{tikzcd}
	{C_1} & {C_2} \\
	& {C_3}
	\arrow["{l}", from=1-1, to=1-2]
	\arrow["{f}", from=1-2, to=2-2]
	\arrow["{l'}"', from=1-1, to=2-2]
\end{tikzcd}\]with $l,\;l'$ in $\mathcal{L}$ then $ f $ is also in $\mathcal{L}$
    \item  $\mathcal{L}$ is closed under colimits in $\overrightarrow{\mathcal{C}}$
 \end{itemize}

 \columnbreak
 \begin{itemize}
    \item $\mathcal{R}$ contains all isomorphisms and is closed under composition,
    \item $\mathcal{R}$ is left-cancellative: if one has 
\[\begin{tikzcd}
	{C_1} & {C_2} \\
	& {C_3}
	\arrow["{f}", from=1-1, to=1-2]
	\arrow["{r}", from=1-2, to=2-2]
	\arrow["{r'}"', from=1-1, to=2-2]
\end{tikzcd}\]with $r,\;r'$ in $\mathcal{R}$ then $ f $ is also in $\mathcal{R}$
    \item  $\mathcal{R}$ is closed under limits in $\overrightarrow{\mathcal{C}}$
 \end{itemize}  
  \end{multicols}
\end{proposition}

We also have the following constrains on the mutual factorizations of left and right maps:

\begin{lemma}\label{constrain on factorization}
If $ (\mathcal{L},\mathcal{R})$ is an orthogonality structure in $\mathcal{C}$ and we have a factorization as below with $ r$ in $\mathcal{R}$ and $l$ in $\mathcal{L}$
\[\begin{tikzcd}
	A & B \\
	C
	\arrow["r", from=1-1, to=1-2]
	\arrow["l"', from=1-1, to=2-1]
	\arrow["f"', from=2-1, to=1-2]
\end{tikzcd}\]
Then $l$ is a split monomorphism, and $f$ factorizes through $r$. Dually, for any factorization as below 
\[\begin{tikzcd}
	A & B \\
	& C
	\arrow["l"', from=1-1, to=2-2]
	\arrow["r", from=1-2, to=2-2]
	\arrow["f", from=1-1, to=1-2]
\end{tikzcd}\]
then $ r$ is a split epimorphism and $ f$ factorizes through $l$.
\end{lemma}

\begin{proof}
Those statement just comes from the diagonalization of the squares
\[\begin{tikzcd}
	A & A \\
	C & B
	\arrow["l"', from=1-1, to=2-1]
	\arrow["f"', from=2-1, to=2-2]
	\arrow["r", from=1-2, to=2-2]
	\arrow[Rightarrow, no head, from=1-1, to=1-2]
	\arrow["d"{description}, from=2-1, to=1-2]
\end{tikzcd} \hskip1cm  \begin{tikzcd}
	A & B \\
	C & C
	\arrow["r", from=1-2, to=2-2]
	\arrow["f", from=1-1, to=1-2]
	\arrow["l"', from=1-1, to=2-1]
	\arrow[Rightarrow, no head, from=2-1, to=2-2]
	\arrow["d"{description}, from=2-1, to=1-2]
\end{tikzcd} \]
\end{proof}

We also have the following usefull property:

\begin{lemma}\label{Diers observation on coequalizers}
Let be $ (\mathcal{L}, \mathcal{R})$ an orthogonality structure in a category $ \mathcal{C}$ with coequalizers: then if a parallel pair $ a,a' : C \rightrightarrows D$ in $\mathcal{C}$ is equalized by a morphism $ l : B \rightarrow C $ in $\mathcal{L}$, then its coequalizer $ q_{a,a'} : D \rightarrow \coeq(a,a')$ is in $\mathcal{L}$. \end{lemma}

\begin{proof}
Let be a square as below with $ r \in \mathcal{R}$
\[\begin{tikzcd}
	D & A \\
	{\coeq(a,a')} & {A'}
	\arrow["{q_{(a,a')}}"', from=1-1, to=2-1]
	\arrow["f", from=1-1, to=1-2]
	\arrow["r", from=1-2, to=2-2]
	\arrow["g"', from=2-1, to=2-2]
\end{tikzcd}\]
Then we have $ fan = fa'n$ and $ g q_{(a,a')} a = g q_{(a,a')} a'$, so that we have a commutative square
\[\begin{tikzcd}
	B &&& A \\
	C & D & {\coeq(a,a')} & {A'}
	\arrow["g"', from=2-3, to=2-4]
	\arrow["r", from=1-4, to=2-4]
	\arrow["{fan=fa'n}", from=1-1, to=1-4]
	\arrow["n"', from=1-1, to=2-1]
	\arrow["{q_{(a,a')}}"', from=2-2, to=2-3]
	\arrow["{a'}"', shift right=1, from=2-1, to=2-2]
	\arrow["a", shift left=1, from=2-1, to=2-2]
\end{tikzcd}\]
and both $ fa$, $fa'$ provides diagonalization of this square: but such a diagonalization must be unique, so that $ fa = fa'$. Hence there exists a unique $d$ factorizing $f$ through the coequalizer as below
\[\begin{tikzcd}
	D & A \\
	{\coeq(a,a')}
	\arrow["{q_{(a,a')}}"', from=1-1, to=2-1]
	\arrow["f", from=1-1, to=1-2]
	\arrow["d"', from=2-1, to=1-2]
\end{tikzcd}\]
Moreover, we have $ udq_{(a,a')} = uf = g q_{(a,a')}$, but as a coequalizer, $q_{(a,a')}$ is an epimorphism: thence necessarily $ ud= g$ so that $ d$ is the desired lift of the square above. Uniqueness of such a lift process from the uniqueness of the solution in the universal property of the coequalizer. \end{proof}

\begin{remark}
Of course we have the dual statement saying that the equalizer of a parallel pair coequalized in $\mathcal{R}$ must be in $\mathcal{R}$.
\end{remark}

\begin{corollary}\label{simultaneous left equalization and right coequalization}
Any two arrows which is simultaneously equalized by a morphism in $\mathcal{L}$ and coequalized by a morphism in $ \mathcal{R}$ must be equal.
\end{corollary}

\begin{proof}
Suppose $ a,a': C \rightrightarrows D$ are equalized by some $ l : B \rightarrow C$ in $\mathcal{L}$ and coequalized by some $r : D \rightarrow A$ in $ \mathcal{R}$. Then the coequalizer $ q_{(a,a')} : D \rightarrow \coeq(a,a')$ is both in $ \mathcal{L}$ and factorizes $r$ as below
\[\begin{tikzcd}
	B & C & D & A \\
	&& {\coeq(a,a')}
	\arrow["a", shift left=1, from=1-2, to=1-3]
	\arrow["{a'}"', shift right=1, from=1-2, to=1-3]
	\arrow["l", from=1-1, to=1-2]
	\arrow["r", from=1-3, to=1-4]
	\arrow["{q_{(a,a')}}"', two heads, from=1-3, to=2-3]
	\arrow["f"', dashed, from=2-3, to=1-4]
\end{tikzcd}\]
Then from \cref{constrain on factorization} we know $ q_{(a,a')}$ to be a split monomorphism; but a coequalizer being always an epimorphism, this forces that actually $ q_{(a,a')}$ is an isomorphism, so that $ a =a'$. 
\end{proof}

In the following we suppose that $ \mathcal{B}$ is a locally finitely presentable category and $ U : \mathcal{A} \rightarrow \mathcal{B}$ is a local right adjoint. Then denote $ \mathcal{D} $ the class of diagonally universal morphisms between finitely presented objects. This coincides with the intersection of the class of diagonally universal morphisms and the class of finitely presented morphisms, that is, $\mathcal{D} = \mathcal{D}iag \cap \mathcal{B}_\omega^2$. We are going to use $ \mathcal{D}$ to left-generate a factorization system, which will enjoy some degree of accessibility.

\begin{proposition}
The class $ \mathcal{D}$ has the following properties: \begin{itemize}
    \item $\mathcal{D}$ is closed under composition and contains isomorphisms between finitely presented objects
    \item $\mathcal{D}$ is right-cancellative
    \item $ \mathcal{D}$ is closed under pushouts along arbitrary finitely presented morphisms
    \item $\mathcal{D}$ is closed under finite colimit in the arrow category $ \mathcal{B}^2$
    \item Any filtered colimit in $\mathcal{B}^2$ of morphisms in $\mathcal{D}$ is diagonally universal.
\end{itemize}
\end{proposition}

\begin{proof}
The two first propositions are obvious. The third is an easy consequence of the universal property of the pushout. The fourth comes from the fact that $ \mathcal{B}^2_\omega$ is itself closed under finite colimits in $\mathcal{B}^2$ because $ \mathcal{B}_\omega$ is so in $\mathcal{B}$ and colimits in the arrow category are sent to colimit by domain and codomain functors, while $\mathcal{D}iag$ is also closed under colimit as a left class in an orthogonality structure. This last argument also proves the last item. 
\end{proof}

Now we invoke results of \cite{Anel} to construct a factorization system $ (\Ind(\mathcal{D}), \mathcal{D}^\perp)$. In our context, we can use the small class $\mathcal{D}$ of finitely presented diagonally universal morphism to left-generate a factorization system. We recall here the process:

\begin{definition}
A \emph{saturated class} is a set $ \mathcal{V} \subseteq {B}^2_{\omega}$ of finitely -presented maps such that:\begin{itemize}
    \item $\mathcal{V}$ contains isomorphisms and is stable by composition,
    \item $\mathcal{V}$ is right-cancellative
    \item $ \mathcal{V}$ is closed under finite colimits in $ { \mathcal{B}}^2$
    \item $\mathcal{V}$ is closed under pushouts along arbitrary maps between finitely presented objects
\end{itemize}
\end{definition}

A saturated class is always small, as lying in the essentially small generator ${\mathcal{B}}^2_{\omega}$. In our case, the class $ \mathcal{D}$ of finitely presented diagonally universal morphisms is a saturated class. 

\begin{proposition}
Any set of finitely presented maps $V \subseteq {\mathcal{B}^2_{\omega}}$ can be completed into a saturated class $ \mathcal{V}$ such that $ V^\perp = \mathcal{V}^\perp$. 
\end{proposition}

\begin{proof}
It is clear that $ \mathcal{V}^\perp \subseteq V^\perp$ for $ V \subseteq \mathcal{V}$. Moreover it is also easy to see that taking the closure by composition and iso does not modify the right class. Stability of the right class after closing the right class under finite colimit is a special case of stability of left class under colimits, but let us give the detailed proof: let be a map $l $ in $V^\perp$. If $ (l_i : K_i \rightarrow K_i')_{i \in I} $ is a diagram in ${\mathcal{B}}^2$, then in the following diagram we have for each $ i \in I$ a lifting
\[\begin{tikzcd}
	{K_i} & {\underset{i \in I}{\colim} \; K_i} & B \\
	{K_i'} & {\underset{i \in I}{\colim} \; K_i'} & {B'}
	\arrow["{k_i}"', from=1-1, to=2-1]
	\arrow["{q_i}", from=1-1, to=1-2]
	\arrow["{q_i'}"', from=2-1, to=2-2]
	\arrow["u", from=1-2, to=1-3]
	\arrow["v"', from=2-2, to=2-3]
	\arrow["l", from=1-3, to=2-3]
	\arrow[""{name=0, anchor=center, inner sep=0}, "{\underset{i \in I}{\colim} \; k_i}"{description}, from=1-2, to=2-2]
	\arrow["{d_i}"{description}, shorten >=10pt, dashed, no head, from=2-1, to=0]
	\arrow[shorten <=10pt, dashed, from=0, to=1-3]
\end{tikzcd}\]
and then by universal property of the colimit, this induces a unique map $ \langle d_i \rangle_{i \in I}$ with $\langle d_i \rangle_{i \in I} q'_i = d_i $. We prove this map is a filler. First, universal properties of colimits give us the following sequence of equalities: 
\[ u = \langle u q_i \rangle_{i \in I} = \langle d_i q_i' l_i \rangle_{i \in I} = \langle d_i \rangle_{i \in I} \langle q'_i l_i \rangle_{i \in I} = \langle d_i \rangle_{i \in I} \underset{i \in I}{\colim} \; l_i \]
so the upper triangle commutes; similarly we have
\[ v = \langle vq_i' \rangle_{i \in I} = \langle l d_i \rangle_{i \in I} = l \langle d_i \rangle_{i \in I} \]
so that the lower triangle commutes also. Finally, the stability of the right class after closing $V$ under pushouts along arbitrary maps is an easy consequence of the universal property of the pushout.
\end{proof}

\begin{definition}
Let be $ B$ in $\mathcal{B}$: define $ \mathcal{D}iag_B$ the category of diagonally universal morphisms with domain $B$, and $ \mathcal{D}_B \hookrightarrow \mathcal{D}iag_B$ the full subcategory whose objects are arrows $ n : B \rightarrow C$ such that there exists some finitely presented diagonally universal morphism $ k : K \rightarrow K'$ in $ \mathcal{D}$ and $ a : K \rightarrow B$ such that we can exhibit $ n$ as a pushout
\[ 
\begin{tikzcd}
K \arrow[d, "a"'] \arrow[r, "k"] \arrow[rd, "\ulcorner", phantom, very near end] & K' \arrow[d, "k_*a"] \\
B \arrow[r, "n"']                                                 & C                   
\end{tikzcd} \]
\end{definition}

\begin{remark}
By left cancellation, $\mathcal{D}iag_B$ is itself a full subcategory of the coslice $ B \downarrow\mathcal{B}$. 
\end{remark}

\begin{proposition}
The category $ \mathcal{D}_B$ is closed under finite colimits in $B\downarrow\mathcal{B}$.
\end{proposition}

\begin{proof}
Let be $(n_i, a_i)_{i \in I}$ with 
\[\begin{tikzcd}
	{K_i} & {K_i'} \\
	{B}
	\arrow["{a_i}"', from=1-1, to=2-1]
	\arrow["{l_i}", from=1-1, to=1-2]
\end{tikzcd}\] a finite diagram in $\mathcal{D}_B$. We can use the fact that $ \mathcal{D}$ is closed under finite colimit to compute the finite colimit of $F$ as 
\[ 
\begin{tikzcd}
K_i \arrow[dd, "a_i"'] \arrow[rd, "q_i " description] \arrow[rr, "l_i"] &                                                                                                                                & K_i' \arrow[rd, "q_i'"] &                                    \\
                                                                            & {\underset{i \in I}{\colim}\, K_i} \arrow[ld, "\langle a_i \rangle_{i \in I}"] \arrow[rr, "{\underset{i \in I}{\colim} \, l_i}"'] &                             & {\underset{i \in I}{\colim}\, K'_i} \\
B                                                                           &                                                                                                                                &                             &                                   
\end{tikzcd} \]
where $(q_i, q_i') : n_i \rightarrow \colim_{i \in I} n_i$ is the inclusion in the colimit computed in the category of arrows and $ \colim_{i \in I}l_i$ is still in $\mathcal{D}$. Then by commutation of pushouts with colimits we have 
\[ \langle a_i {\rangle_{i \in I}}_* \underset{i \in I}{\colim} \, l_i =  \underset{i \in I}{\colim} \, {a_i}_*l_i \]
and this map is still in $\mathcal{D}_B$. 
\end{proof}

As a consequence, for any $ f : B \rightarrow C$, the category $ \mathcal{D}_B \downarrow f$ of diagonally universal morphisms under $B$ above $f$ is filtered. Moreover, recall that the codomain functor $ B \downarrow \mathcal{B}$ preserves filtered colimits. Now we can construct the factorization of any arrow $f$ in $\mathcal{B}$:

\begin{proposition}\label{SOA}
For any $ f : B \rightarrow C$ in $ \mathcal{B}$ we have a factorization 
\[ 
\begin{tikzcd}
B \arrow[rr, "f"] \arrow[rd, "{\colim \, \mathcal{D}_B\downarrow f }"'] &                                                 & C \\
                                                        & {\underset{\mathcal{D}_B\downarrow f}{\colim} \,C} \arrow[ru, "r_f"'] &       
\end{tikzcd} \]
with $ l_f$ in $\Ind(\mathcal{D})$ and $ r_f$ is in $\mathcal{D}^\perp$.
\end{proposition}

For a complete proof of the statement, see \cite{Anel}[section 2.3] and in particular [theorem 14].

\begin{proposition}
The factorization above is orthogonal, that is, $ \Ind(\mathcal{D}) = ^\perp(\mathcal{D}^{\perp})$.
\end{proposition}

\begin{proof}
First recall that diagonally universal are closed under (filtered) colimits. For $l \in  ^\perp(\mathcal{D}^\perp)$ with factorization $ f = r_fl_f$, $ f$ is left orthogonal to its own right part, therefore there is a unique filler in the diagram below
\[ 
\begin{tikzcd}
                       & B \arrow[dd, "f" description] \arrow[ld, "l_f"'] \arrow[r, "l_f"] & {C_f} \arrow[dd, "r_f"] \\
C_f \arrow[rd, "r_f"'] &                                                                   &                      \\
                       & C \arrow[r, equal] \arrow[ruu, description, dashed]      & C                   
\end{tikzcd} \]
But now left cancellation of left maps, together with cancellation of right maps, enforces that this filler is an isomorphism, which forces $r_f$ to be iso, so that $f$ is in $\Ind(\mathcal{D})$. 
\end{proof}


\begin{remark}

At this point we should also give some details about the right classes, in particular $ \mathcal{D}iag^\perp$ under the conditions of \cref{LMP from accessible Mradj}. We have $ \mathcal{D}iag^\perp = (^\perp U({\mathcal{A}}^2))^\perp$, exhibiting it as the closure of $U({\mathcal{A}}^2)$ under the axioms of the right classes. But from the specific assumptions the closure operations simplifies:\begin{itemize}
    \item by functoriality $ U({\mathcal{A}}^2)$ is allready stable under composition
    \item it is also left-cancellative for $U$ is relatively full and faithful
    \item and closed in ${\mathcal{B}}^2$ under connected limits for $ {\mathcal{A}}^2$ has connected limits and $ U$ preserves them.  
\end{itemize}  
Therefore one has to add: \begin{itemize}
    \item all the isomorphisms of $\mathcal{B}$, comprising those between objects out of the essential image of $U$; in fact it is sufficient to add the identities
    \item in particular one must add the terminal map $ 1_1$, the identity of $1$, in order to have all limits in ${\mathcal{B}}^2$
    \item and add all pullbacks along arbitrary morphisms of $ \mathcal{B}$ as below
\[\begin{tikzcd}
	{f^*U(A)} & {U(A)} \\
	B & {U(A')}
	\arrow["f"', from=2-1, to=2-2]
	\arrow["{U(u)}", from=1-2, to=2-2]
	\arrow["{f^*U(u)}"', from=1-1, to=2-1]
	\arrow[from=1-1, to=1-2]
	\arrow["\lrcorner"{anchor=center, pos=0.125}, draw=none, from=1-1, to=2-2]
\end{tikzcd}\]
    \item finally one must also add retract in ${\mathcal{B}}^2$ of maps in $ U({\mathcal{A}}^2)$ as below
\[\begin{tikzcd}[row sep=small]
	B && B \\
	& {U(A)} \\
	{B'} && {B'} \\
	& {U(A')}
	\arrow[Rightarrow, no head, from=1-1, to=1-3]
	\arrow[Rightarrow, no head, from=3-1, to=3-3]
	\arrow["f"{description}, from=1-1, to=2-2]
	\arrow["g"{description}, from=2-2, to=1-3]
	\arrow["{f'}"', from=3-1, to=4-2]
	\arrow["g'"', from=4-2, to=3-3]
	\arrow["{U(u)}"{description, pos=0.3}, from=2-2, to=4-2, crossing over]
	\arrow["r"{description}, from=1-1, to=3-1]
	\arrow["r"{description}, from=1-3, to=3-3]
\end{tikzcd}\]
\end{itemize}

\end{remark}

\begin{remark}
Hence we have for any local right adjoint $U$ a factorization system $(\Ind(\mathcal{D}), \mathcal{D}^\perp)$. Observe that in the general case, the local units under a given object may not be obtained as filtered colimit of finitely presented diagonaly universal morphisms under them. If one successively take the stable factorization and the $  (\Ind(\mathcal{D}), \mathcal{D}^\perp) $-factorization
\[ 
\begin{tikzcd}
B \arrow[d, "{\underset{\mathcal{D}_B\downarrow \eta^A_f}{\colim} \, n}"'] \arrow[r, "f"] \arrow[rd, "\eta^A_f" description] & U(A)                           \\
{\underset{\mathcal{D}_B\downarrow \eta^A_f}{\colim} \, C} \arrow[r, "u_{\eta^A_f}"']                                                 & U(A_f) \arrow[u, "U(L_A(f))"']
\end{tikzcd} \]
then 
\[u_f = U(L_A(f)) u_{\eta^A_f} \]
is in $ \Ind(\mathcal{D})^\perp$ by  uniqueness of the factorization because $ \mathcal{D}iag^\perp \subseteq \Ind(\mathcal{D})^\perp$, so that the right part of $ \eta^A_f$ is also the right part of $f$, so that we know that the functor \[\mathcal{D}_B \downarrow \eta^A_f \rightarrow \mathcal{D}_B \downarrow f\] is cofinal since it induces the same colimit. Remark also that $ u_{\eta^A_f}$ is in $ \mathcal{D}iag \cap \Ind(\mathcal{D})^\perp$, which however does not forces it however to be an isomorphism.\\

Moreover, this situation cannot even be improved in the case of a right multi-adjoint, where the local units under a given object form a small set. This is why, as we shall see in the second part, we have to impose explicitly the condition that local units are filtered colimits of finitely presented diagonally universal morphisms above them amongst conditions isolated by Diers to enable the construction of a spectra from a right multi-adjoint.
\end{remark}

\begin{definition}
A local right adjoint functor $ U $ is said to be \emph{diagonally axiomatisable} if we have $ \mathcal{D}iag = \Ind(\mathcal{D})$. It is said to satisfy \emph{Diers condition} if for any $B$ and any $f : B \rightarrow U(A)$, the local unit $ \eta^A_f $ is in $\Ind(\mathcal{D}iag_B)$.
\end{definition}

Recall that $ \mathcal{B}^2$ also is locally finitely presentable, with $ \mathcal{B}^2_\omega$ as generator of finitely presented objects. We have now a functor preserving finite colimits
\[ \mathcal{D} \stackrel{\iota_\mathcal{D}}{\hookrightarrow} \mathcal{B}^2_\omega \]
which extends into pair of adjoint functors
\[ 
\begin{tikzcd}
\mathcal{B}^2 \quad \arrow[rr, " {\iota_\mathcal{D}}_*"', bend right=20] & \perp & \Ind(\mathcal{D}) \arrow[ll, "\iota_\mathcal{D}^*"', hook, bend right=20]
\end{tikzcd}  \]
where $ \Ind(\mathcal{D})$ itself is locally finitely presentable. This gives rise to an idempotent comonad, one could see as returning the left part of a factorization system. In fact, the adjunction $ \iota_\mathcal{D}^*\dashv {\iota_\mathcal{D}}_*$ defines a morphism of locally finitely presentable categories, for the left adjoint $\iota^*_\mathcal{D}$ restricts to finitely presented objects, as morphisms in $\mathcal{D}$ are in particular finitely presented in ${\mathcal{B}}^2$.

\begin{remark}
It is well known that a left adjoint between locally finitely presented categories sends finitely presented object to finitely presented objects if and only if its right adjoint preserves filtered colimits. In our case, this says that the functor $ \iota_{D*}$, which returns the left part of the factorization, preserves filtered colimits. This means that for any filtered diagram $ f_{(-)} : I \rightarrow {\mathcal{B}}^2$ of arrows $ f_i : B_i \rightarrow B'_i$ in ${\mathcal{B}}^2$, the left part of the filtered colimit is the filtered colimit of the left parts
\[ l_{\underset{i \in I}{\colim} \; f_i} \simeq \underset{i \in I}{\colim} \; l_{f_i} \]
In particular, if we apply this to the canonical diagram of an arrow $ f \simeq \colim \; \mathcal{B}_\omega^2\downarrow f$, we have 
\[  l_{f} \simeq \underset{\mathcal{B}_\omega^2\downarrow f}{\colim} \; l_{k} \]
where the colimit is indexed by all the $ (k, a, a')$ in $ \mathcal{B}_\omega^2\downarrow f$. Hence this proves that arrows of the form $ l_k$ for $ k \in \mathcal{K}$ form a generator in $\Ind(\mathcal{D})$. 
\end{remark}

Moreover, objects in $\Ind(\mathcal{D})$ are then models of a finite limit theory, which motivates the following definition:

\begin{definition}
A diagonally universal morphism is said to be \emph{axiomatisable} if it lies in $ \Ind(\mathcal{D})$.  
\end{definition}

\begin{remark}
Observe however that in general one cannot force arbitrary diagonally universal morphisms to be decomposable as a filtered colimit of morphisms in $ \mathcal{D}$. That is, we only have $ \Ind(\mathcal{D}) \subseteq \mathcal{D}iag$ (and hence $ \mathcal{D}iag^\perp \subseteq \Ind(\mathcal{D})^\perp$) in the general case. For this reason, with the inclusion of right class a morphism between orthogonality structures, the factorization system $  (\Ind(\mathcal{D}), \Ind(\mathcal{D})^\perp) $ is the free left generated factorization system associated to the orthogonality structure $ (\mathcal{D}iag, \mathcal{D}iag^\perp)$. 
\end{remark}

\begin{remark}
Diers condition says that local units coincide with the diagonally axiomatisable morphism in the induced factorization. This is a strictly weaker condition than being diagonally axiomatisable as it does not requires any diagonally universal morphism to be axiomatisable.
\end{remark}

However in some situation we can produce local right adjoints with the desired property if we start from a left generated factorization system and a class of objects enjoying an adequate ``gliding" condition:

\begin{definition}
Let be a functor $ U : \mathcal{A} \rightarrow \mathcal{B}$ and $ \mathcal{R} $ a class of maps in $\mathcal{B}$. We say that $U$ \emph{lifts $\mathcal{R}$ maps} if for any $A$ in $\mathcal{A}$ and any $ r: B \rightarrow U(A)$, there exists $ u : A_0 \rightarrow A$ and an isomorphism $ \alpha: U(A_0) \simeq B$ such that $r \alpha=U(u)$.
\end{definition}

Let be $ \mathcal{V}$ a saturated class in a locally finitely presentable category $\mathcal{B}$ and $ (\mathcal{L}, \mathcal{R})$ the associated left generated system, with $ \mathcal{L} = \Ind(\mathcal{V})$ and $ \mathcal{R} = \mathcal{V}^\perp$. Now suppose that $U_0 : \mathcal{A}_0 \rightarrow \mathcal{B}$ is a functor lifting $ \mathcal{R}$-maps. Then define $\iota_0: \mathcal{A} \hookrightarrow \mathcal{A}_0$ as the wide subcategory whose arrows are those whose image under $U_0$ are in $\mathcal{R}$ and $ U $ as the restriction $ U=U_0 \iota_0$.

\begin{proposition}
Suppose that $ U_0$ is relatively full and faithful; then the induced functor $U: \mathcal{A} \rightarrow \mathcal{B}$ is stable and diagonally axiomatisable. 
\end{proposition}

\begin{proof}
For any $ f : B \rightarrow U(A)$, consider the axiomatisable factorization
\[\begin{tikzcd}
	{B} && {U(A)} \\
	& {C_f}
	\arrow["{f}", from=1-1, to=1-3]
	\arrow["{l_f}"', from=1-1, to=2-2]
	\arrow["{r_f}"', from=2-2, to=1-3]
\end{tikzcd}\]
where $ l_f $ is obtained as the filtered colimit $ l_f = \colim \mathcal{D}_B \downarrow f$ in $B \downarrow \mathcal{B}$. For $U_0$ lifts along $\mathcal{R}$-maps, there exists $u_f : A_f \rightarrow A$ in $\mathcal{A}$ sent by $U$ to $r_f$, and moreover, this morphism is essentially unique in $\mathcal{A}$ as $U_0$ is relatively full and faithful. But then $l_f$ is the local unit of $U$, or equivalently, is a candidate for $U$ for it is diagonally universal with its image in the range of $U$ by \cref{relffdiaguniv}.
 \end{proof}

There is also a converse property:

\begin{proposition}
Let $U : \mathcal{A} \rightarrow \mathcal{B}$ a diagonally axiomatizable right multi-adjoint: then $U$ lifts its local maps. 
\end{proposition}

\begin{proof}
For any local map of the form $r : C \rightarrow U(A)$, any precomposition with a diagonally universal morphism with codomain $C$ as below 
\[\begin{tikzcd}
	B & C & {U(A)}
	\arrow["l", from=1-1, to=1-2]
	\arrow["r", from=1-2, to=1-3]
\end{tikzcd}\]
must also coincides up to iso with its own factorization, for in the square below
\[\begin{tikzcd}
	B & {U(A_{rl})} \\
	C & {U(A)}
	\arrow["l"', from=1-1, to=2-1]
	\arrow["r"', from=2-1, to=2-2]
	\arrow["{\eta^A_{rl}}", from=1-1, to=1-2]
	\arrow["{UL_A(rl)}", from=1-2, to=2-2]
\end{tikzcd}\]
we have both a diagonal $ C \rightarrow U(A_{rl})$ induced from the fact that $ l$ is diagonally universal, and a diagonal $ U(A_{rl}) \rightarrow C$ from the fact that $ \eta^A_{rl}$ is diagonally universal and $ r$ is a local map, hence is in $ \mathcal{D}iag^\perp$ by diagonal axiomatizability. It is easy to see that those maps are mutual inverse. Applying the result before with $ l = 1_C$, we get an isomorphism $ \eta^A_{r} : C \simeq U(A_r)$, and we have $ r = UL_A(r)\eta^A_r$. 
\end{proof}

To complete this section, we should give a word on the way conversely, a factorization system together with a convenient class of objects induces a stable functor. A typical example of non full multireflection is the following as pointed out in \cite{Taylor} - where we recall that a class of maps closed under composition can be seen as a subcategory with is bijective on objects:

\begin{proposition}
For any factorization system $(\mathcal{L}, \mathcal{R})$ (resp. left generated factorization system), the inclusion $\mathcal{R} \hookrightarrow \mathcal{B}$ is a relatively full and faithful stable functor (resp. right multi-adjoint).\\

Conversely any stable functor (resp. right multi-adjoint) which is surjective on objects and faithful is the inclusion of a right class (resp. of a right class in a left generated factorisation system).
\end{proposition}

\begin{proof}
If we have a factorization system, any morphism $ f : B \rightarrow \iota_\mathcal{R}(A) $ factorizes uniquely since $ \iota_\mathcal{R}(r)l$ and $ l$ is orthogonal to any morphism in $ \mathcal{R}$ so we can see it as the desired candidate. \\

If $ U : \mathcal{A} \rightarrow \mathcal{B} $ is stable, faithful and surjective on objects, any object is some $ U(A)$ and any $ f : B \rightarrow U(A)$ factorizes through a candidate which is left orthogonal to the morphisms in the range of $U$, hence the class of candidates constitutes the left part and the morphisms in $ U(A^2)$ the right part: this is a factorization on the whole category as $U$ is surjective on objects.
\end{proof}

The following terminology was suggested by Anel, and was also identified in \cite{Diers}[part 4] amongst condition to produce a spectral construction: 

\begin{definition}\label{gliding}
Let be $\mathcal{R}$ a class of maps in a category and $ \mathcal{A}$ a class of objects. We say that $\mathcal{A}$ has the \emph{gliding property relatively to} $\mathcal{R}$ if for any arrow $ l : B \rightarrow A $ in $\mathcal{R}$ with $ A$ a n object of $\mathcal{A}$, then $ B$ must also be in $\mathcal{A}$.
\end{definition}

\begin{theorem}\label{Stable functor from gliding}
Let $\mathcal{B}$ be a category equiped with a factorization system $(\mathcal{L}, \mathcal{R})$ and $\mathcal{A}$ be a class of objects of $\mathcal{B}$ with the gliding condition relative to $\mathcal{R}$. Then the inclusion $\mathcal{A}^\mathcal{R} \hookrightarrow \mathcal{B}$ of $ \mathcal{A}$ objects equiped with arrows of $\mathcal{R} $ between them defines a relatively full and faithful stable functor.
\end{theorem}

\begin{proof}
This follows from the previous proposition: for any $B$ in $ \mathcal{B}$ and any arrow $ f : B \rightarrow A$ with $A $ an object in $\mathcal{A}$, as there exists a unique factorization $ B \stackrel{n_f}{\rightarrow} A_f \stackrel{u_f}{\rightarrow} A$ with $n_f$ in $ \mathcal{E}$ and $ u_f $ in $\mathcal{R}$, then $A_f$ is also an object in $\mathcal{A}$; this factorization is initial amongst those through a morphism in $\mathcal{R}$ on the right. Moreover $ n_f$ is a morphism in $ \mathcal{L}$ with an objects in $\mathcal{A}$ as codomain, and such arrows are exactly the candidates for the inclusion as for any square with $ A_0, A_1,A_2$ objects in $\mathcal{A}$, we have the diagonalization
\[ 
\begin{tikzcd}
B \arrow[]{d}[swap]{n} \arrow[]{r}{} & A_1 \arrow[]{d}{u } \\ A_0 \arrow[dashed]{ru}[description]{d} \arrow[]{r}[swap]{u' } &A_2 
\end{tikzcd} 
\]
But recall that $ \mathcal{R}$ is left-cancellative as any right class, so that $ u $ must itself be in $\mathcal{R}$. Hence in the factorization above $ n_f$ is a candidate and the inclusion is stable; left-cancellativity of $\mathcal{R}$ also enforces that this inclusion is relatively full and faithful.
\end{proof}

\section{Locally multi-presentable categories and accessible right multi-adjoint}

Locally multipresentables categories were introduced by Diers in \cite{DIA_1983__9__A1_0} (under the name ``cat\'{e}gories localisantes") as an a generalization of locally presentable categories encompassing a wide class of non-locally presentable categories, as local rings, fields, integral domains, local lattices... They are defined in the same way as locally presentable categories, but in the language of multi-colimits. For their tight relation with multiadjoitness and their reccuring role as the categories of local objects in spectral situations, we recall Diers theory of locally multipresentable categories, also present in \cite{Diers-multipres}. \\

First of all, recall the dual notion of multi-limit as introduced in section 1.

\begin{definition}
A category $\mathcal{A}$ is said to be \emph{locally finitely multipresentable} if \begin{itemize}
\item it has filtered colimits
    \item it has a small generator of finitely presented object such that any object $A$ decomposes as the filtered colimit $ A = \colim \, \mathcal{A}_\omega \downarrow A$ 
    \item it is multicocomplete
\end{itemize} 
\end{definition}

As well as locally finitely presentable can be characterized as the finitely accessible categories that are moreover either complete or cocomplete, locally finitely multipresentable categories are also characterized 

\begin{proposition}
Locally finitely multipresentable categories are exactly the finitely accessible categories with connected limits, where finite connected limits commutes with filtered colimits.
\end{proposition}

\begin{remark}
Recall that a category is complete if and only if it has connected limits and a terminal object. Then a locally finitely multipresentable category is a locally finitely presentable category if and only if it has a terminal object.
\end{remark}

\begin{remark}
If $\mathcal{A}$ is locally finitely multipresentable, then:\begin{itemize}
    \item the arrow category $\mathcal{A}^2$ also is locally finitely multipresentable
    \item for any object $ A$ in $\mathcal{A}$ the coslice $ A\downarrow \mathcal{A}$ also is finitely multipresentable and moreover, the codomain functor $ \cod : A \downarrow \mathcal{A} \rightarrow \mathcal{A}$ is finitary and right multiadjoint.
\end{itemize}
\end{remark}

\begin{proof}
For the first item, observe that the category $ \mathcal{A}^2_\omega$ is a generator of finitely presented objects. Now multicolimits in $\mathcal{A}^2$ are computed as follows: for $I$ a finite category and $ (f_i : A_i \rightarrow A'_i)_{i \in I}$ a $I$-indexed diagram in $\mathcal{A}^2$, the multicolimit $ ( g'_{ij} : A'_i \rightarrow B'_j)_{i \in I, j \in J}$ of the codomains induces a multicocone $ (g'_{ij}f_i : A_i \rightarrow B_j)_{i \in I, j \in J}$, which defines in each $ j \in J$ a cocone over the $(A_i)_{i \in I}$, and if we chose a multicolimit $ (g_{ik} : A_i \rightarrow B_k)_{i \in I, k \in K}$, then for each $ i \in I$ and $j \in J$ there is a unique $ k \in K$ such that the cocone $ (g_{ij}'f_i)_{i \in I} $ factorizes uniquely through $g_{ik}$ 
\[\begin{tikzcd}
	{A_i} && {A_i'} \\
	{B_{k}} && {B_j'}
	\arrow["{g_{ij}'}", from=1-3, to=2-3]
	\arrow["{g_{ik}}"', from=1-1, to=2-1]
	\arrow["{f_i}", from=1-1, to=1-3]
	\arrow["{\overline{f_j}}"', dashed, from=2-1, to=2-3]
\end{tikzcd}\]
and the family $ (\overline{f}_{j} : B_k \rightarrow B_j')_{j \in J}$ is a multicolimit in $\mathcal{A}^2$. Then in particular for $ \mathcal{A}_\omega$ is closed under finite multicolimits, $ \mathcal{A}_\omega^2$ is so.\\

For the second item, we refer to \cite{diers1977categories}[Proposition 8.4] concerning existence of multicolimits. However we emphasize the following: for any object $A'$ in $ \mathcal{A}$, the canonical cone of local unit relatively to the codomain functor $ A \downarrow \mathcal{A } \rightarrow \mathcal{A}$ is made of the inclusions $ ( q_i : A' \rightarrow B_i)_{i \in I}$ into a multicoproduct of $ A' $ with $ A$. In particular, as we shall see in \cref{LMP from accessible Mradj}, this allows us to exhibits the generator of finitely presented objects $(A \downarrow \mathcal{A})_\omega$ as consisting of members of the multicoproducts of $A$ with finitely presented objects of $\mathcal{A}$. 
\end{proof}

For the sake of completeness, we recall here elements of (finite) \emph{Diers duality}, which is the multi-version of (finite) Gabriel-Ulmer duality: 

\begin{proposition}[Diers duality]
Let $ \mathcal{A}$ be a locally finitely multipresentable category. Then $ \mathcal{A}_\omega^{\op}$ is finitely multicomplete and we have an equivalence of category $ \mathcal{A} \simeq \mathcal{M}lex[\mathcal{A}_\omega^{\op}, \mathcal{S}et]$ where $\mathcal{M}lex$ denotes the 2-category category of small finitely multicomplete categories, small finitely multicomplete functors and natural transformation. Conversely for any small multicomplete category $ \mathcal{C}$, $\mathcal{M}lex[\mathcal{C}, \mathcal{S}et] $ is locally finitely multipresentable and $ \mathcal{C}^{\op}$ is equivalent to its small generator of finitely presented objects.
\end{proposition}


Another analog to locally presentable categories is the relation with adjointness. First, there is a well known variant of the adjoint functor theorem saying that a functor between locally presentable categories have a left adjoint if and only if it is accessible and preserves limits. Here we provide an analogous statement for right multiadjoint with a locally multipresented domain:

\begin{theorem}\label{Accessible multi-adjoint functor theorem}
Let be $U : \mathcal{A} \rightarrow \mathcal{B}$ a functor with $ \mathcal{A}$ and $ \mathcal{B}$ locally finitely multipresentable categories. Then $U$ is a right multiadjoint if and only if it is accessible and preserves connected limits.
\end{theorem}

\begin{remark}
Beware that we cannot control the rank of accessibility of $U$, which will end up $\lambda$-accessible for some $\lambda \geq \aleph_0$ we do not know, even though $\mathcal{A}$ and $\mathcal{B}$ are finitely accessible. In the following proof, we make use of the following general fact: for $ \kappa \leq \lambda$ two cardinals, $ \lambda$-filtered categories are in particular $\kappa$-filtered; hence if a category has $ \kappa$-filtered colimits, it has in particular $ \lambda$-filtered colimits, and similarly, a functor preserving $ \kappa$-filtered colimits preserves in particular $ \lambda$-filtered colimits. In particular, $ \kappa$-presented objects are also $ \lambda$-presented. 
\end{remark}

\begin{proof}
For the indirect sense, it is known (see \cite{makkai1989accessible} and \cite{TholenRosickySolutionSet}) that accessible functors between accessible categories satisfy the solution set condition. As $\mathcal{A}$ and $\mathcal{B}$ are in particular (finitely) accessible and $U$ is accessible, $U$ satisfies the solution set condition, and if moreover it preserves connected limits, then by \cref{Mradj preserves connected limits} it is right multi-adjoint.\\

For the direct sense, we propose this adaptation of \cite{adamek1994locally}[Theorem 1.66] in the context of multi-adjointness and multipresentability. Suppose that $U$ is right multiadjoint (and hence preserves connected limits by \cref{Mradj preserves connected limits}). As $ \mathcal{B}_\omega$ is small and $ \mathcal{A}$ is accessible, we can chose a cardinal $ \lambda \geq \aleph_0$ such that for any $K$ in $\mathcal{B}_\omega$ and any local unit $ n_x : K \rightarrow U(A_x)$ with $x \in I_K$, the object $ A_x$ is $\lambda$-presented in $\mathcal{A}$. Then we prove that $ U$ is $\lambda$-accessible as follows. \\

Let be $F : I \rightarrow \mathcal{A}$ a $ \lambda$-filtered diagram: in particular, $ I$ is finitely filtered, and the colimit $( q_i : F(i) \rightarrow  \colim \; F)_{i \in I}$ exists in $\mathcal{A}$. We prove that $ (U(q_i) : F(i) \rightarrow U(\colim \; F))_{i \in I}$ is a colimit in $\mathcal{B}$ - and in particular, for it is $\lambda$-filtered, this colimit is also finitely filtered. By \cite{adamek1994locally}[Exercise 1.o(1)] we know that it is sufficient to check that\begin{itemize}
    \item for any finitely presented object $K $ in $\mathcal{B}$ and $ f : K \rightarrow U(\colim \; F)$, there is some lift $ a :K \rightarrow UF(i)$ of $ a$ along $q_i$ for some $i \in I$ 
    \item and that for any two such lifts $ b : K \rightarrow UF(i) $ and $ b' : K \rightarrow F(i') $ there is a common refinement $ d : i \rightarrow i'' $ and $ d' : i' \rightarrow i'' $ such that $ UF(d) b= UF(d') b'$.
\end{itemize}  
For $ a : K \rightarrow U(\colim \; F)$, consider the local factorization 
\[\begin{tikzcd}
	K & {} & {U(\colim \; F)} \\
	& {U(A_{x(a)})}
	\arrow["{n_{x(a)}}"', from=1-1, to=2-2]
	\arrow["{U(u_a)}"', from=2-2, to=1-3]
	\arrow["a", from=1-1, to=1-3]
\end{tikzcd}\]
As $K$ is finitely presented, $ A_{x(a)}$ is $ \lambda$-presented in $\mathcal{A}$, and as $ I$ is $\lambda$-filtered, there is a lift 
\[\begin{tikzcd}
	& {F(i)} \\
	{A_{x(a)}} & {\colim \; F}
	\arrow["{q_i}", from=1-2, to=2-2]
	\arrow["{u_a}"', from=2-1, to=2-2]
	\arrow["b", dashed, from=2-1, to=1-2]
\end{tikzcd}\]
whose image along $U$ provide the desired lift by precomposing with the local unit $ n_{x(a)}$
\[\begin{tikzcd}
	K && {U(\colim \; F)} \\
	& {U(A_{x(a)})} & {U(F(i))}
	\arrow["{U(q_i)}"', from=2-3, to=1-3]
	\arrow["{U(u_a)}"{description}, from=2-2, to=1-3]
	\arrow["{U(b)}"', dashed, from=2-2, to=2-3]
	\arrow["{n_{x(a)}}"', from=1-1, to=2-2]
	\arrow["a", from=1-1, to=1-3]
\end{tikzcd}\]
Now for two such lifts $ b : K \rightarrow UF(i)$, $b' : K \rightarrow F(i')$ of $a$, as the image of the inclusions $ U(q_i)$, $U(q_{i'})$ are in the range of $U$, we know by \cref{BC} that up to canonical isomorphism, $b $ and $b'$ factorize through the same local unit as $ a$, that is $ x(b) = x(a) = x(b')$ and we have the commutation below
\[\begin{tikzcd}[column sep=large]
	&& {UF(i)} \\
	K & {U(A_{x(a)})} && {U(\colim \; F)} \\
	&& {UF(i')}
	\arrow["{n_{x(a)}}"{description}, from=2-1, to=2-2]
	\arrow["{U(u_a)}"{description}, from=2-2, to=2-4]
	\arrow["{U(u_b)}"{description}, from=2-2, to=1-3]
	\arrow["{U(q_i)}", from=1-3, to=2-4]
	\arrow["{U(u_{b'})}"{description}, from=2-2, to=3-3]
	\arrow["{U(q_{i'})}"', from=3-3, to=2-4]
	\arrow["b", curve={height=-12pt}, from=2-1, to=1-3]
	\arrow["{b'}"', curve={height=12pt}, from=2-1, to=3-3]
\end{tikzcd}\]
This returns in $\mathcal{A}$ the following commutation
\[\begin{tikzcd}
	& {F(i)} \\
	{A_{x(a)}} && {\colim \; F} \\
	& {F(i')}
	\arrow["{u_a}"{description}, from=2-1, to=2-3]
	\arrow["{u_b}", from=2-1, to=1-2]
	\arrow["{q_i}", from=1-2, to=2-3]
	\arrow["{u_{b'}}"', from=2-1, to=3-2]
	\arrow["{q_{i'}}"', from=3-2, to=2-3]
\end{tikzcd}\]
But now, for $A_{x(a)}$ is $\lambda$-presented and $I$ is $\lambda$-directed, the lifts $ u_{b}$ and $u_{b'}$ factorize through a common refinement $ d : i \rightarrow i'' $ and $ d' : i' \rightarrow i'' $ such that $  F(d)u_{b} = F(d') u_{b'}$, which provides in $\mathcal{B}$ the desired common refinement for $b$ and $ b'$. \\

Then, we know that the cocone $ (U(q_i) : UF(i) \rightarrow U(\colim \; F))_{i \in I} $ is a colimit in $\mathcal{B}$; this proves that $U$ preserves $\lambda$-filtered colimits, that is, is $\lambda$-accessible. 
\end{proof}

\begin{remark}
In \cref{LMP from accessible Mradj}, we are going to prove that if we suppose that $U$ is \emph{finitely} accessible, then the localizations $ A_x$ for $x \in I_K$ and $K$ finitely presented in $\mathcal{B}$ are finitely presented. In fact, in the proof above, $\lambda$ was dependent of the rank of presentability of the localizations of finitely presented objects: if all the $A_x$ for $x \in I_K$ and $K$ in $\mathcal{B}_{\omega}$ happen to be finitely presented, then the functor $U$ can be certified as finitely accessible. 
\end{remark}

Recall that any locally finitely presentable functor is a right adjoint preserving filtered colimits; the last condition amounts to saying that its left adjoint sends finitely presented objects to finitely presented objects. Now recall from \cite{adamek1994locally}[Theorem 1.39] that a full, reflective subcategory $\mathcal{A}$ of a locally finitely presentable category $\mathcal{B}$ which is moreover closed under filtered colimits is locally finitely presentable - and the inclusion functor is locally finitely presentable; moreover, one could use as the generator of finitely presented objects in $\mathcal{A}$ the reflections of the finitely presented objects of $\mathcal{B}$. In \cite{diers1977categories}[Theorem 8.3.1] is given the following analog - we propose here a slightly different adaptation of.

\begin{proposition}\label{LMP from accessible Mradj}
Let $\mathcal{B}$ be a locally finitely presentable category and $ U : \mathcal{A} \hookrightarrow \mathcal{B}$ such that\begin{itemize}
    \item $U$ is a right multiadjoint
    \item $U$ is moreover faithful, and relatively full and faithful,
    \item $ \mathcal{A}$ closed under filtered colimits and $ U$ preserves them.
\end{itemize} 
Then $\mathcal{A}$ is a locally finitely multipresentable category. Moreover, the full, dense generator $ \mathcal{A}_\omega$ of finitely presented object of $\mathcal{A}$ has as objects local reflections of finitely presented objects, that is 
\[ \mid \!\mathcal{A}_\omega\! \mid \,= \coprod_{K \in \mathcal{B}_\omega} \{ A_x \mid \; n_x : K \rightarrow U(A_x) \in I_K \} \]
where $I_K$ is the set of local units of $K$ for $U$. 
\end{proposition}

\begin{proof}
First, let us prove that $\mathcal{A}$ is finitely accessible. For it is supposed to have filtered colimits, we have to check that $\mathcal{A}_\omega$ is a small dense generator of finitely presented objects. Let be $ K$ in $\mathcal{B}_\omega$ and $ n_x : K \rightarrow U(A_x)$ a local unit under $K$. We first check that $ A_x$ is finitely presented in $\mathcal{A}$: let be a filtered diagram $ F : I \rightarrow \mathcal{A}$, and an arrow $ a : A_x \rightarrow \colim F$. Then for $U$ is supposed to preserves filtered colimits and $ K$ is finitely presented, there is some $i $ in $ I $ such that we have a lift
\[\begin{tikzcd}
	K & {UF(i)} \\
	{U(A_x)} & {\colim \; UF}
	\arrow["{U(a)}"', from=2-1, to=2-2]
	\arrow["{U(q_i)}", from=1-2, to=2-2]
	\arrow["{n_x}"', from=1-1, to=2-1]
	\arrow["b", dashed, from=1-1, to=1-2]
\end{tikzcd}\]
Then the local reflection of $b$ provides us with a filler and morover we have a commutation
\[\begin{tikzcd}
	& {F(i)} \\
	{A_x} & {\colim \; F}
	\arrow["a"', from=2-1, to=2-2]
	\arrow["{q_i}", from=1-2, to=2-2]
	\arrow["{L_{F(i)}(b)}", from=2-1, to=1-2]
\end{tikzcd}\]
Hence $ A_x$ is finitely presented. Now we prove that for each $ A$ in $\mathcal{A}$, we have a filtered colimit $ A \simeq \mathcal{A}_\omega\downarrow A$. First, observe that there exist a canonical arrow in $\mathcal{A}$ as below 
\[\begin{tikzcd}
	& {A_x} \\
	A && {\colim \; \mathcal{A}_\omega\downarrow A}
	\arrow["{q_a}", from=1-2, to=2-3]
	\arrow["a"', from=1-2, to=2-1]
	\arrow["{\langle a \rangle_{a\in \mathcal{A}_\omega\downarrow A }}", from=2-3, to=2-1]
\end{tikzcd}\]
induced from the fact that the diagram $ \dom : \mathcal{A}_\omega \downarrow A \rightarrow \mathcal{A}$ is equipped with a canonical cocone of tip $A$. Moreover, for $U$ preserves filtered colimit, this triangle is sent by $U$ on a triangle 
\[\begin{tikzcd}
	& {U(A_x)} \\
	{U(A)} && {\colim \; U(\mathcal{A}_\omega\downarrow A)}
	\arrow["{U(q_a)}", from=1-2, to=2-3]
	\arrow["{U(a)}"', from=1-2, to=2-1]
	\arrow["{\langle U(a) \rangle_{a\in \mathcal{A}_\omega\downarrow A }}", from=2-3, to=2-1]
\end{tikzcd}\]
But in $\mathcal{B}$, we have a filtered colimit  $ U(A) \simeq \colim \mathcal{B}_\omega \downarrow U(A)$, where each $ b : K \rightarrow U(A)$ is the inclusion at the index it defines in this canonical colimit. Moreover the local adjoint of $U$ over $A$ restricts to finitely presented objects as a functor 
\[ \mathcal{B}_\omega \downarrow U(A) \rightarrow \mathcal{A}_\omega \downarrow A \]
sending any $ b : K \rightarrow U(A)$ to its local part $ L_A(b) : A_{x(b)} \rightarrow A$ for the unique $x(b)$ in $I_K$ corresponding to the local factorization of $b$. For each $b$ the corresponding $ L_A(b)$ is equipped with a canonical inclusion $ q_{L_A(b)} : A_{x(b)} \rightarrow \colim \mathcal{A}_\omega \downarrow A$ and one has $ L_A(b) = \langle a \rangle_{ \mathcal{A}_\omega \downarrow A} q_{L_A(b)}$. Then one has a factorization of the inclusion $b$
\[\begin{tikzcd}
	& K & {U(A_{x(b)})} \\
	{U(A)} &&& {\colim \; U(\mathcal{A}_\omega\downarrow A)}
	\arrow["{U(q_{L_A(b)})}", from=1-3, to=2-4]
	\arrow["{n_{x(b)}}", from=1-2, to=1-3]
	\arrow["b"', from=1-2, to=2-1]
	\arrow["{U(\langle a \rangle_{\mathcal{A}_\omega \downarrow A})}", from=2-4, to=2-1]
	\arrow["{UL_A(b)}"{description}, from=1-3, to=2-1]
\end{tikzcd}\]
which entails by the universal property of the colimit that we have a retraction 
\[\begin{tikzcd}[row sep=small]
	{U(A)} \\
	& {U(\colim \; \mathcal{A}_\omega\downarrow A)} \\
	{U(A)}
	\arrow[Rightarrow, no head, from=1-1, to=3-1]
	\arrow["{\langle U(q_{L_A(b)})n_{x(b)}\rangle_{b \in \mathcal{B}_\omega \downarrow U(A)}}", dashed, from=1-1, to=2-2]
	\arrow["{U(\langle a \rangle_{a \in \mathcal{A}_\omega\downarrow A})}", from=2-2, to=3-1]
\end{tikzcd}\]
Moreover, for $U$ is relatively full and faithful, the map $\langle U(q_{L_A(b)})n_{x(b)}\rangle_{b \in \mathcal{B}_\omega \downarrow U(A)} $ must actually comes from a unique section $ u_A $ in $\mathcal{A}$ 
\[\begin{tikzcd}[row sep=small]
	A \\
	& {\colim \; \mathcal{A}_\omega\downarrow A} \\
	A
	\arrow[Rightarrow, no head, from=1-1, to=3-1]
	\arrow["{u_a}", from=1-1, to=2-2]
	\arrow["{\langle a \rangle_{a \in \mathcal{A}_\omega\downarrow A}}", from=2-2, to=3-1]
\end{tikzcd}\]
On the other hand, the composites $ U( q_{L_A(b)}) n_x : K \rightarrow \colim \; U(\mathcal{A}_\omega \downarrow A)$ define altogether a universal map
\[\begin{tikzcd}
	& K & {U(A_{x(b)})} \\
	{U(A)} &&& {\colim \; U(\mathcal{A}_\omega\downarrow A)}
	\arrow["{U(q_{L_A(b)})}", from=1-3, to=2-4]
	\arrow["{n_{x(b)}}", from=1-2, to=1-3]
	\arrow["b"', from=1-2, to=2-1]
	\arrow["{\langle U(q_{L_A(b)}) n_{x(b)} \rangle_{b \in \mathcal{B}_\omega \downarrow U(A)}}"', from=2-1, to=2-4]
\end{tikzcd}\]
But as we have $\langle U(q_{L_A(b)}) n_{x(b)} \rangle_{b \in \mathcal{B}_\omega \downarrow U(A)} = U(u_A) $ and for each $b : K \rightarrow U(A)$, $ n_{x(b)}$ is a local unit, and the local factorization provides a diagonalization of this square 
\[\begin{tikzcd}
	& {A_{x(b)}} \\
	A && {\colim  \; \mathcal{A}_\omega \downarrow A}
	\arrow["{q_{L_A(b)}}", from=1-2, to=2-3]
	\arrow["{L_A(b)}"', dashed, from=1-2, to=2-1]
	\arrow["{u_A}"', from=2-1, to=2-3]
\end{tikzcd}\]
Moreover, for $U$ is relatively full and faithful, we knwo that each $U_A$ is full and faithful, so that the counits of the local adjunctions are isomorphisms, and for any $ a : A_x \rightarrow A$ in $ \mathcal{A}_\omega \downarrow A$, we have $ a \simeq L_A(U_A(a))$, and hence $a = L_A(U_A(a)n_x)$, and $ q_a = q_{L_A(U_A(a))}$. Hence for any $ a \in \mathcal{A}_\omega \downarrow A$ we have a factorization of the inclusion $ q_a$ given as 
\[\begin{tikzcd}
	& {A_x} \\
	A && {\colim  \; \mathcal{A}_\omega \downarrow A}
	\arrow["{q_{a}}", from=1-2, to=2-3]
	\arrow["{L_A(U(a)n_x)}"', dashed, from=1-2, to=2-1]
	\arrow["{u_A}"', from=2-1, to=2-3]
\end{tikzcd}\]
which entails by the universal property of the colimit that we have a retraction 
\[\begin{tikzcd}[row sep=small]
	&& {\colim  \; \mathcal{A}_\omega \downarrow A} \\
	A \\
	&& {\colim  \; \mathcal{A}_\omega \downarrow A}
	\arrow["{u_A}"', from=2-1, to=3-3]
	\arrow[Rightarrow, no head, from=1-3, to=3-3]
	\arrow["{\langle a \rangle_{a \in \mathcal{A}_\omega \downarrow A}}"', from=1-3, to=2-1]
\end{tikzcd}\]
Hence we have an isomorphism $ A \simeq \colim  \; \mathcal{A}_\omega \downarrow A$ as desired. \\

Now we prove that $\mathcal{A}$ has multicolimits. Let be $F : I \rightarrow \mathcal{A}$. Then, for $\mathcal{B}$ is locally finitely presentable, we can compute in $\mathcal{B}$ the colimit $ \colim \; UF$; now consider the local units under $ \colim \; UF$, and observe that for each of those local unit $ x \in I_{ \colim \; UF}$ we have a cocone over $UF$
\[\begin{tikzcd}
	{UF(i)} && {UF(j)} \\
	& { \colim \; UF} \\
	& {U(A_x)}
	\arrow["{UF(d)}", from=1-1, to=1-3]
	\arrow["{q_i}"', from=1-1, to=2-2]
	\arrow["{q_j}", from=1-3, to=2-2]
	\arrow["{n_x}", from=2-2, to=3-2]
\end{tikzcd}\]
and we claim that each of those cocones comes uniquely from a cocone in $\mathcal{A}$, which altogether form the multicolimit. For a given cocone $ (f_i :F(i) \rightarrow A)_{i \in I}$ in $\mathcal{A}$, take the universal map $ \langle U(f_i) \rangle_{i \in I} : \colim \; UF \rightarrow U(A)$; for $U$ is right multi-adjoint, this arrows factorizes uniquely through one of the local unit of index $x(\langle U(f_i) \rangle_{i \in I}) \in I_{ \colim \; UF}$. Then for each $i \in I$ we have a factorization of $U(f_i)$
\[\begin{tikzcd}
	{UF(i)} \\
	{\colim \; UF} && {U(A)} \\
	{U(A_{x(\langle U(f_i) \rangle_{i \in I})})}
	\arrow["{U(f_i)}", from=1-1, to=2-3]
	\arrow["{\langle U(f_i) \rangle_{i \in I}}"{description}, from=2-1, to=2-3]
	\arrow["{U_A(L_A(\langle U(f_i) \rangle_{i \in I}))}"', from=3-1, to=2-3]
	\arrow["{q_i}"', from=1-1, to=2-1]
	\arrow["{n_{x(\langle U(f_i) \rangle_{i \in I})}}"', from=2-1, to=3-1]
\end{tikzcd}\]
which forces the composite $n_{x(\langle U(f_i) \rangle_{i \in I})} q_i$ to come from a unique arrow $ \gamma^x_i : F(i) \rightarrow A_{x(\langle U(f_i) \rangle_{i \in I})}$ in $\mathcal{A}$ for $U$ is relatively full and faithful. Hence in particular each composite $ n_x q_i$ for $ i \in I$ and $ x \in I_{\colim \; UF}$ comes from a unique $ \gamma^x_i : F(i) \rightarrow A_x$ in $\mathcal{A}$, and the family of cocones 
\[ ((\gamma^x_i : F(i) \rightarrow A_x)_{i \in I})_{x \in I_{\colim \; UF}} \]
is exhibited as a multicolimit for $F$ in $\mathcal{A}$. Hence $\mathcal{A}$ is multicocomplete, and being finitely accessible, it is locally finitely multipresentable.\\

In particular observe that $\mathcal{A}$ has connected limits - and observe that $U$ preserves them as a right multi-adjoint, so that they are computed as limits in $\mathcal{B}$.\end{proof} 

In this context, we can in particular compute explicitly the factorization through local units associated to $U$ by using the canonical cone of its domain:

\begin{corollary}
With the hypothesis of \cref{LMP from accessible Mradj}, for any $f : B \rightarrow U(A)$, we have 
\[ A_f \simeq \underset{b \in \mathcal{B}_\omega\downarrow B}{\colim} \; A_{fb} \quad \quad \eta^A_{f} = \underset{b \in \mathcal{B}_\omega\downarrow B}{\colim} \; \eta^A_{fb} \quad \quad L_A(f) = \langle L_A(fb) \rangle_{b \in \mathcal{B}_\omega\downarrow B} \]
\end{corollary}

\begin{proof}
We have $ B \simeq  {\colim} \; \mathcal{B}_\omega\downarrow B$, so that $ f = \langle fb \rangle_{b \in \mathcal{B}_\omega\downarrow B}$. Now for a given $b : K \rightarrow B$, consider the factorization
\[\begin{tikzcd}[column sep=large]
	K & B && {U(A)} \\
	& {U(A_{\eta^A_fb})} & {U(A_f)}
	\arrow["f", from=1-2, to=1-4]
	\arrow["b", from=1-1, to=1-2]
	\arrow["{\eta^A_f}"{description}, from=1-2, to=2-3]
	\arrow["{UL_A(f)}"', from=2-3, to=1-4]
	\arrow["{\eta^{A_f}_{\eta^A_fb}}"', from=1-1, to=2-2]
	\arrow["{UL_{A_f}(\eta^A_f b)}"', from=2-2, to=2-3]
\end{tikzcd}\]
but by \cref{BC} we know that actually 
\[  A_{\eta^A_fb} \simeq A_{fb} \quad \quad \eta^{A_f}_{\eta^A_fb} = \eta^A_{fb} \quad \quad L_A(f) L_{A_f}(\eta^A_fb) = L_A(fb)  \]

\end{proof}

In \cite{diers1980quelques}, a specific process to construct locally finitely multipresentable categories is provided, which encompasses most of the interesting examples and enjoys Diers condition. Though this process is originally done inside an ambient locally multipresentable category, we consider here the special case of an ambient locally presentable category, where the construction simplifies slightly and meets the version of the small object argument presented above. \\

In the following, let $\mathcal{B}$ be a locally finitely presentable category. We consider a small class $\Gamma$ of small cones in $\mathcal{B}_{\omega}$, with \[ V_\Gamma = \{ g_i \mid \; (g_i : K \rightarrow K_i)_{i \in I} \in \Gamma, i \in I \} \] the set of all arrows involved in cones of $\Gamma$. \\

\begin{definition}
An object $ A$ in $\mathcal{B}$ is said to be \emph{$\Gamma$-local} if for any cone $ (g_i : K \rightarrow K_i)_{i \in I}$ in $\Gamma$, we have a surjection
\[ \begin{tikzcd}
\underset{i \in I}{\coprod} \mathcal{B}[K_i, A] \arrow[ two heads]{rr}{ \langle \mathcal{B}[g_i, A] \rangle_{i \in I} } & & \mathcal{B}[K, A]
\end{tikzcd}  \]
\end{definition}

\begin{remark}
Observe that this notion is slightly weaker than the one considered in \cite{diers1977categories}[Section 8.6], which we treat later under the name of \emph{strongly $\Gamma$-local}. Our notion, beside being more general, is more suited to encompass the situations corresponding to geometric extensions of the finite limit theory behind $ \mathcal{B}$ as we shall see below.
\end{remark}

We have also a notion of local morphism, which is actually the notion of right map relatively to the class of all maps involved in the cones of $\Gamma$; together with local objects they form a certain subcategory of $\mathcal{B}$ we are going to focus on: 

\begin{definition}
A morphism is said to be \emph{$\Gamma$-local} if it is in $ V_\Gamma^\perp$. We denote as $ \mathcal{B}_\Gamma$ the category of $\Gamma$-local objects and $\Gamma$-local morphisms between them, equiped with a faitfull, injective on object inclusion
\[ \begin{tikzcd}
 \mathcal{B}_\Gamma \arrow[hook]{r}{U_\Gamma}& \mathcal{B} \end{tikzcd} \]
\end{definition}

\begin{remark}
Observe that the category hence obtained is not a full subcategory: this allows in particular to select a distinguished class of morphisms between models of an extension of the theory behind $\mathcal{B}$, while considering only a finite limit extension of the theory of $\mathcal{B}$ in the same signature does not restrict morphisms between the models of the extension, producing always a full subcategory: this is because morphisms are determined by the signature and not by the axioms of the theory. 
\end{remark}

We know then from \cref{SOA} that $ \Gamma$-local morphisms are the right class of a left generated factorization system $ (\Ind(\overline{V_\Gamma}), V^\perp_\Gamma) $ where $\overline{V_\Gamma} $ is the saturated class generated from $ V_\Gamma$. \\

Then we have the following property, which is slightly more general than  \cite{diers1980quelques}[Theorem 3.2] and whose proof simplifies also a bit thanks to the prior treatment of the small object argument:

\begin{theorem}\label{Local objects are LMP}
Let $\Gamma$ be any small class of small cones in $\mathcal{B}_\omega$. Then the inclusion $ U_\Gamma$ is
\begin{itemize}
    \item accessible,
    \item relatively full and faithful,
    \item right multi-adjoint,
    \item and satisfies Diers condition
\end{itemize} 
In particular it exhibits $ \mathcal{B}_\Gamma$ as a locally finitely multipresentable category. 
\end{theorem}

\begin{proof}
First observe that $U_\Gamma$ is faithful; moreover, it is easy to see it is relatively full and faithful as a consequence of left cancellability local maps enjoy as a right class. \\

We now prove that $\mathcal{B}_\Gamma$ has filtered colimits and that they are preserved by $U_\Gamma$. \\

Let be $ F : I \rightarrow \mathcal{B}_\Gamma$ a small diagram with $I$ filtered. Consider its filtered colimit $ \colim U_\Gamma F$ in $ \mathcal{B}$. Now take a cone $ (g_j : K \rightarrow K_j)_{j\in J}$ in $ \Gamma$.  Now for the joint surjectivity, consider an arrow $ a : K \rightarrow \colim \; U_\Gamma F$; as $K$ is finitely presented, we have a lift $ b : K \rightarrow F(i)$ for some $i \in I$, and by localness of $ F(i)$ there is some $j$ in $J$ and a factorization 
\[\begin{tikzcd}
	K && {F(i)} \\
	{K_j} && {\colim \; U_\Gamma F}
	\arrow["{g_j}"', from=1-1, to=2-1]
	\arrow["{q_i}", from=1-3, to=2-3]
	\arrow["b", from=1-1, to=1-3]
	\arrow["a"{description}, from=1-1, to=2-3, near end]
	\arrow["{\bar{b}}"{description}, from=2-1, to=1-3, near start, crossing over]
\end{tikzcd}\]
and we have $ a = q_i b = q_i \bar{b} g_j$. Hence the localness of $\colim \; U_\Gamma F$.  \\

 We must now prove the inclusions $ q_i$ are $\Gamma$-local. Consider a square as below, for $ (g_j : K \rightarrow K_j)_{j \in I}$ and $ j \in J$:
\[\begin{tikzcd}
	K & {F(i)} \\
	{K_j} & {\underset{i \in I}{\colim} \, F(i)}
	\arrow["a", from=1-1, to=1-2]
	\arrow["b"', from=2-1, to=2-2]
	\arrow["{g_j}"', from=1-1, to=2-1]
	\arrow["{q_i}", from=1-2, to=2-2]
\end{tikzcd}\]
Then for $ K_j$ is finitely presented and $I$ filtered, there is some $ i'$ such that $ b $ factorizes through $ q_{i'}$ as $ b = q_{i'}b'$. But then $ q_i a = b g_j = q_{i'}b'g_j$, so that there exists some span $ s : i \rightarrow i''$, $s' : i' \rightarrow i''$ in $I$ such that we have a factorization as follows
\[\begin{tikzcd}
	K & {F(i)} \\
	& {F(i'')} && {\underset{i \in I}{\colim} \, F(i)} \\
	{K_j} & {F(i')}
	\arrow["{g_j}"', from=1-1, to=3-1]
	\arrow["{q_i}", curve={height=-12pt}, from=1-2, to=2-4]
	\arrow["b"', from=3-1, to=3-2]
	\arrow["{q_{i'}}"', curve={height=12pt}, from=3-2, to=2-4]
	\arrow["a", from=1-1, to=1-2]
	\arrow["{F(s)}"{description}, from=1-2, to=2-2]
	\arrow["{F(s')}"{description}, from=3-2, to=2-2]
	\arrow["{q_{i''}}"{description}, from=2-2, to=2-4]
\end{tikzcd}\]
But then for $ g_j$ is in $V_\Gamma$ and $ F(s)$ is $\Gamma$-local we have a unique diagonalization as below
\[\begin{tikzcd}
	K && {F(i)} \\
	{K_j} & {F(i')} & {F(i'')} && {\underset{i \in I}{\colim} \, F(i)}
	\arrow["{g_j}"', from=1-1, to=2-1]
	\arrow["{q_i}", curve={height=-12pt}, from=1-3, to=2-5]
	\arrow["b"', from=2-1, to=2-2]
	\arrow["a",from=1-1, to=1-3]
	\arrow["{F(s)}"{description}, from=1-3, to=2-3]
	\arrow["{F(s')}"', from=2-2, to=2-3]
	\arrow["{q_{i''}}"{description}, from=2-3, to=2-5]
	\arrow["d"{description}, dashed, from=2-1, to=1-3]
\end{tikzcd}\]
where $d$ provides also a diagonalization of the original square above.\\

Moreover such a diagonalization must always be unique: suppose indeed one has two parallel diagonalization
\[\begin{tikzcd}
	K & {F(i)} \\
	{K_j} & {\underset{i \in I}{\colim} \, F(i)}
	\arrow["a", from=1-1, to=1-2]
	\arrow["{g_j}"', from=1-1, to=2-1]
	\arrow["{q_i}", from=1-2, to=2-2]
	\arrow["d", shift left=1, from=2-1, to=1-2]
	\arrow["b"', from=2-1, to=2-2]
	\arrow["{d'}"', shift right=1, from=2-1, to=1-2]
\end{tikzcd}\]
Then $ g_j$ equalize $ d,d'$; but we also have that $ q_i $ coequalizes $ d,d'$, and from the fact that $ K_j$ is finitely presented and $I$ is filtered, this implies there exists some arrow $ s : i \rightarrow i'$ in $I$ such that $ F(s) $ coequalizes $ d,d'$: but $F(s)$ is $ \Gamma$-local, that is, in $V_\Gamma^{\perp}$, so by \cref{simultaneous left equalization and right coequalization}, this forces $d,d'$ to be equal. This finish to prove that $ \mathcal{B}_\Gamma$ has filtered colimit and that $ U_\Gamma$ preserves them, hence is accessible.  \\

Now we have to prove $U_\Gamma$ is a right multi-adjoint. We consider as already constructed the factorization system $(\Ind(\overline{V_{\Gamma}}), V_\Gamma^\perp) $. We prove that $ \Gamma$-local objects have the gliding property relatively to local maps (see \cref{gliding}). Suppose that $ A$ is $\Gamma$-local and $ C$ is an object equipped with a local arrow $ u : C \rightarrow A$. Now for any cone $ (g_i : K \rightarrow K_i)_{i \in I}$ in $\Gamma$ and any $f : K \rightarrow C$ we have a factorization for some $ g_i$:
\[\begin{tikzcd}
	K & C \\
	{K_i} & A
	\arrow["{g_i}"', from=1-1, to=2-1]
	\arrow["f", from=1-1, to=1-2]
	\arrow["u", from=1-2, to=2-2]
	\arrow["a"', dashed, from=2-1, to=2-2]
\end{tikzcd}\]
Then for $ g_i \perp u$ this induces a unique factorization $ K_i \rightarrow C$ of $f$ as desired. Then $ C$ is $\Gamma$-local. \\

Hence from \cref{Stable functor from gliding} we know that $ U_\Gamma $ is stable, as $ \mathcal{B}_\Gamma$ consists precisely of a class of objects that have the gliding property to a right class. Hence for any arrow $ f : B \rightarrow A $ with $A$ a $\Gamma$-local object, the $(\Ind(\overline{V_\Gamma}), V_\Gamma^\perp)$-factorization 
\[\begin{tikzcd}
	B && A \\
	& {A_f}
	\arrow["f", from=1-1, to=1-3]
	\arrow["{u_f}"', from=2-2, to=1-3]
	\arrow["{n_f}"', from=1-1, to=2-2]
\end{tikzcd}\]
returns a local object $A_f$. Moreover, as $U_\Gamma$ is relatively full and faithful, arrow $ n : B \rightarrow A$ in $\Ind(\overline{V_\Gamma}) $ with $ A$ $\Gamma$-local are exactly the candidate for $U$ : hence $ U_\Gamma$ automatically satisfies Diers condition.\\

By \cref{stable is MRadj}, $ U_\Gamma$ is local right adjoint. But now, for $ U_\Gamma$ is accessible, it satisfies the solution set condition: then by \cref{Mradj = Lradj + SSC}, $ U_\Gamma$ is a right multi-adjoint. As a consequence, $ \mathcal{B}_\Gamma$ is locally finitely multipresentable, from \cref{LMP from accessible Mradj}.
 \end{proof}

\begin{remark}
Beware however that without additional assumption, $U_\Gamma$ may not be diagonally axiomatizable; it may happens actually that the class of $\Gamma$-local object is empty, or at least does not contain enough objects to ensure that $ ^\perp U_\Gamma(\mathcal{B}_\Gamma^2) = \Ind(\overline{V_\Gamma})$. \\

Moreover, we also must restrict to the local maps between local objects in order to get a multireflective subcategory, for there is no reason for local objects to have the gliding property along arbitrary maps. 
\end{remark}

\begin{definition}
An object $A$ is said to be \emph{strongly $\Gamma$-local} if it is $ \Gamma$-local and satisfies moreover the condition that for any cone $ (g_i : K \rightarrow K_i)_{i \in I} $ in $\Gamma$ and any $i \in I$, one has an injection
 we have in each $i \in I$ an injection
\[ \begin{tikzcd}[column sep=large]
\mathcal{B}[K_i, A] \arrow[hook]{r}{\mathcal{B}[g_i, A] } & \mathcal{B}[K, A]
\end{tikzcd}  \]
Again one can define $ \mathcal{B}_\Gamma^{\textrm{strong}} \hookrightarrow \mathcal{B}$ the subcategory of strongly $ U^{\textrm{strong}}_\Gamma : \Gamma$-local objects and local maps between them.
\end{definition}

Then we have the analog of \cref{Local objects are LMP}, which is the original content of \cite{diers1980quelques}[Theorem 3.2]:

\begin{proposition}
The inclusion $ U^{\textrm{strong}}_\Gamma$ is accessible, relatively full and faithful, and axiomatizable right multi-adjoint, and it exhibits $ \mathcal{B}^{\textrm{strong}}_\Gamma$ as a locally finitely multipresentable category. 
\end{proposition}

\begin{proof}
The proof would be essentially the same as for \cref{Local objects are LMP}: we just have to control the additional injectivity condition.\\

Take a filtered diagram $ F : I \rightarrow \mathcal{B}_\Gamma$. We know that $ \colim_I F$ is local, and that the inclusion are local. We must prove that the colimit is strongly local. Suppose that, for some $ j \in J$ one has two arrows $ a, a'$ such that $ a g_j = a'g_j$; as $K_j$ is finitely presented, there are respectively $ i, i' $ in $I$ and arrows $ b, \, b'$ factorizing respectively $ a$ and $a'$ as below
\[\begin{tikzcd}
	& {F(i)} & {F(i')} \\
	{K_i} && {\colim \;U_\Gamma F}
	\arrow["a", shift left=1, from=2-1, to=2-3]
	\arrow["{a'}"', shift right=1, from=2-1, to=2-3]
	\arrow["{q_i}"', from=1-2, to=2-3]
	\arrow["{q_{i'}}", from=1-3, to=2-3]
	\arrow["{b}", from=2-1, to=1-2]
	\arrow["{b'}"{description}, crossing over, from=2-1, to=1-3]
\end{tikzcd}\]
But then observe that $ q_i b g_j = ag_j = a'g_j = q_{i'}b'g_j$, and as $K$ is finitely presented and $F$ is filtered, there is some $i''$ in $I$ and a cospan $ (d : i \rightarrow i'', d' : i' \rightarrow i'') $ in $I$ such that we have a common factorization as below
\[\begin{tikzcd}
	& K \\
	{F(i)} & {F(i'')} & {F(i')} \\
	& {\colim U_\Gamma F}
	\arrow["{q_i}"{description}, from=2-1, to=3-2]
	\arrow["{q_{i'}}"{description}, from=2-3, to=3-2]
	\arrow["{F(d)}"{description}, from=2-1, to=2-2]
	\arrow["{F(d')}"{description}, from=2-3, to=2-2]
	\arrow["{q_{i''}}"{description}, from=2-2, to=3-2]
	\arrow["{ag_j}"', from=1-2, to=2-1]
	\arrow["{a'g_j}", from=1-2, to=2-3]
\end{tikzcd}\]
Now by localness of $F(i'')$, we know that $ F(d)b = F(d')b'$ for those arrows are equalized by precomposition with $g_j$. But then we have 
\[ a = q_i b = q_{i''} F(d) b = q_{i''} F(d') b' = q_{i'} b' = a'  \]
which proves the injectivity of the applications $ \mathcal{B}[g_j, \colim \; U_\Gamma F]$.\\

Now it remains to prove that $ U^{\textrm{strong}}_\Gamma$ is a right multi-adjoint. Again we just have to prove that strongly $\Gamma$-local objects have the gliding property relatively to $\Gamma$-local maps. If one has $ u : C \rightarrow A$ with $u$ local and $ A$ a strongly $\Gamma$-local object, then we know from the last item of \cref{Local objects are LMP} that $ C$ is $\Gamma$-local; now we must strengthen this condition. For any cone $ (g_i : K \rightarrow K_i)_{i \in I}$ in $\Gamma$, suppose we have two parallel extension of a same $f$ as below
\[\begin{tikzcd}
	K & C \\
	{K_i}
	\arrow["{g_i}"', from=1-1, to=2-1]
	\arrow["f", from=1-1, to=1-2]
	\arrow["a", shift left=1, from=2-1, to=1-2]
	\arrow["{a'}"', shift right=1, from=2-1, to=1-2]
\end{tikzcd}\]
Then one has $ lf=lag_i = la'g_i$, which enforces that $ la=la'$ as $ A$ is strongly $\Gamma$-local: but then $ a,a'$ are simultaneously equalized by $g_j$ which is in $V_\Gamma$, and coequalized by $u$ which is in $V_\Gamma^\perp$, so that they are equal by \cref{simultaneous left equalization and right coequalization}.
\end{proof}

\section{Co-stable functors from factorization systems}


We saw in the previous section how stable functors towards a locally presentable category induced a factorization system through a small object argument. 
In this part we also make use of a factorization system, but for a different purpose: in the context of a category equipped with a factorization system and terminal object, we show how to construct a \emph{costable} inclusion from a class of ``left objects' (aka, objects with a left map as their terminal map). The motivation for this construction will be made clearer in a future work, where the bicategorical version of this process will be developed and applied in the bicategory of Grothendieck topoi relatively to several factorization systems, and provide a notion of ``2-geometry" for Grothendieck topoi. However, as we chose to remain purely 1-categorical in this paper, we do not go into the 2-dimensional version of this construction and give it as an autonomous construction to prepare for its future involvement.\\

Throughout this section, we want to emphasize the ``geometric" intuition leading our work. Hence, while the remaining of the paper followed ``algebraic convention", where we studied stables functor in to categories of objects to be seen as algebraic (for instance, living in a locally presentable category), this section will produce co-stable functors rather than stable ones, and the object of the ambient category should be seen as spaces. \\

Before anything, let us precise that we call a functor $ F :  \mathcal{A} \rightarrow \mathcal{C}$ \emph{co-stable } if the corresponding functor between the opposite categories $ F^{op} : \mathcal{A}^{op} \rightarrow \mathcal{B}^{op}$ is stable. \\

We fix a category $ \mathcal{C}$ with a terminal object. This object should be seen as the \emph{generic point}, as is the point $ *$ in the category of topological spaces, or the topos of Sets in the bicategory of Grothendieck toposes, or the 2-elements lattice 2 seen as the trivial locale. Now we equip $\mathcal{C}$ with a factorization system $ (\mathcal{L}, \mathcal{R})$. \\

Now the core idea of this section is that one can classify objects of $\mathcal{C}$ as left or right depending whether their terminal map is left or right; though arbitrary objects may not lie in either of those two classes, it is well known that right classes always define reflective subcategories by this process, which means that any object admits a ``right replacement". 

\begin{definition}
Let $\mathcal{C}$ be a category equiped with a factorization system $(\mathcal{L}, \mathcal{R} )$ and a terminal object. Let define \begin{itemize}
    \item the class of $\mathcal{L}$-objects as those $ L$ such that $ L \stackrel{!L}{\rightarrow} 1$ is in $\mathcal{L}$; together with the $L$ maps they form a subcategory $ \mathcal{L}Obj$.
    \item the class of $\mathcal{R}$-objects as those $ R$ such that $ R \stackrel{!R}{\rightarrow} 1$ is in $\mathcal{R}$; together with the $R$ maps they form a subcategory $ \mathcal{R}Obj$.
\end{itemize}
\end{definition}
\begin{remark}

Now observe the following: \begin{itemize}
    \item 1 is, up to iso, the only object up to iso to be both left and right.
    \item By right cancellation, any point $ 1 \rightarrow R$ of a right object is a right map.
    \item Any arrow toward a right object factorizes through a right object by right cancellability of the right class.
    \item The category $\mathcal{R}Obj $ of right objects and right maps is reflexive in $ \mathcal{C}$ because of the factorization of terminal maps
    \[ \begin{tikzcd}
    A \arrow[]{rr}{!_A} \arrow[]{rd}[swap]{l_A} & & 1 \\ & \mathcal{R}(A) \arrow[]{ru}[swap]{r_A} & 
    \end{tikzcd} \] 
    Indeed the reflector $ A \stackrel{l_A}{\rightarrow} \mathcal{R}(A)$ (which is a left map) is initial amongst those arrows toward a right object: for any $ A \stackrel{f}{\rightarrow} R$, the statute of $(l_A, r_A)$ as the terminal factorization of $!_A$ with a left map on the left induces the dashed arrow in the following:
    \[ \begin{tikzcd}
   & \arrow[]{ld}[swap]{l_A} A \arrow[]{rr}{f} \arrow[]{rd}[swap]{l_f}  & & R \arrow[bend left=20]{ldd}{!_R} \\ 
  \mathcal{R}(A) \arrow[]{drr}[swap]{r_A} \arrow[dashed]{rr}[swap]{\exists \in \mathcal{L}} & & R_f \arrow[]{d}{!_{R_f}} \arrow[]{ru}[swap]{r_f} & \\ & & 1 & 
    \end{tikzcd} \]
    \item In particular the right reflection of a left object is necessarily 1 by applying cancellation property of both classes to its terminal map. 
    \item Any object in $ \mathcal{C}$ admits exactly one $\mathcal{L}$-map into an $\mathcal{R}$-object: its own reflection map, because post composing it with the terminal map of this $\mathcal{R}$-object returns the factorization of the terminal map of $\mathcal{C}$ which is unique. 
    \item Both right and left objects possess a gliding condition along their associated map by stability by composition: in the following 
    
    \[ \begin{tikzcd}
    A \arrow[]{rr}{f } \arrow[]{rd}[swap]{!_A} & & B \arrow[]{ld}{!_B}\\
    & 1 & \end{tikzcd} \]
    Then $!_A$ is either in $ \mathcal{R}$ or $\mathcal{L}$ as soon as both $ f $ and $!_B$ are. However, as we are in the geometric side, this is not the condition from which we want to deduce stability.
    \item Left objects have moreover the co-gliding condition relatively to the $ \mathcal{L}$-maps: in the following 
    \[ \begin{tikzcd}
    L \arrow[]{rr}{l \in \mathcal{L} } \arrow[]{rd}[swap]{!_L}  && A \arrow[]{ld}{!_A}\\
    & 1 & \end{tikzcd} \] if $ L$ is a left object sending a left map toward an object $A$, then $A$ has to be a left object by left cancellation. As a consequence, in the opposite side we have a stable inclusion:
    \[ \mathcal{LO}bj^{op} \hookrightarrow \mathcal{C}^{op} \] 
    However notify we cannot infer any smallness condition, so this has to be token as a local reflexivity. But this just forces us to admits the use of large spaces with a proper (at least a moderate) class of points.
\end{itemize}
\end{remark}
Most of those properties have a relativised version, that is, a corresponding statement where 1 has been replaced by an arbitrary base object $B$ of $\mathcal{C}$. 

\begin{definition}
For $B$ in $\mathcal{C}$, we define $\mathcal{LO}bj[B] $ (resp. $ \mathcal{RO}bj[B]$) as the categories whose objects are respectively left maps $ l : C \rightarrow B  \in \mathcal{L}$ (resp. $ r : C \rightarrow B  \in \mathcal{R}$) and $ \mathcal{L}$-maps (resp $\mathcal{R}$-maps) between them.
\end{definition}

\begin{proposition}
We have the same properties as over 1:\begin{itemize}
    \item $\mathcal{RO}bj[B] $ is a reflective subcategory of $\mathcal{C}/B$ where the free right object of a map $f : C \rightarrow B$ is $ r_f : C_f \rightarrow B$ and its unit is given by the left map $ l_f : C \rightarrow C_f $;
    \item iso morphism $  B' \simeq B$ are the unique object to be both right and left;
    \item the category $\mathcal{LO}bj[B]$ is co-stable in $ C/B$ because of the glidding condition induced by the left cancellation property of the left maps.
\end{itemize}
\end{proposition}


Let $ (\mathcal{L},\mathcal{R})$ be a factorization system on $\mathcal{C}$ as above. Now we are going to see that we also can chose as left objects a subclass $ \mathcal{L}' \subseteq \mathcal{L}$ behaving at least as a left class in an orthogonality structure, while not constituting necessarily the whole left class of a factorization system, plus some additional condition of closure by pullback and cancellability along $\mathcal{L}$-maps. 

\begin{definition}
Let be $\mathcal{C}$ and $( \mathcal{L}, \mathcal{R})$ as above with a distinguished class $ \mathcal{L}' \subseteq \mathcal{L}$. We say that $\mathcal{L}'$ has \emph{right $\mathcal{L}$-cancellation property} when for each triangle as below
\[ \begin{tikzcd}[row sep=small]
A \arrow[]{rr}{l \in \mathcal{L}} \arrow[]{rd}[swap]{l' \in \mathcal{L}'} & & B \arrow[]{ld}{f} \\ & C & 
\end{tikzcd} \]
then $l$ being in $ \mathcal{L}$ and $ l'$ in $\mathcal{L}'$ forces $f$ to be in $\mathcal{L}'$. 
\end{definition}

In the following we fix a class of map $ \mathcal{L}' \subseteq \mathcal{L}$ satisfying the right $\mathcal{L}$-cancellation property which moreover we suppose stable by pullback and composition and containing all iso. In particular $ \mathcal{L}'$-maps are left orthogonal to $ \mathcal{R}$-maps. 

\begin{definition}
We define the category $ \mathcal{L}'Obj^\mathcal{L}$ as having $ \mathcal{L}'$-objects as objects and $ \mathcal{L}$-maps between them. Similarily, for any $ B$ in $\mathcal{C}$, we define  $ \mathcal{L}'Obj^\mathcal{L}[B]$ as having $ \mathcal{L}'$-maps $ l : C \rightarrow B$ for objects and triangles with $ \mathcal{L}$-maps between them as morphisms.  
\end{definition}

\begin{proposition}
In a context above with $\mathcal{L}'$ having right $\mathcal{L}$-cancellation, we have a stable inclusion
\[ (\mathcal{L}'Obj^\mathcal{L})^{op} \hookrightarrow \mathcal{C}^{op}  \]
\end{proposition}

Indeed it is easy to see that we then have the co-gliding condition above any object by absorption of the terminal map of the intermediate object of any $(\mathcal{L}, \mathcal{R})$-factorization: indeed in the situation below
\[ \begin{tikzcd}[row sep=small]
L' \arrow[]{rr}{f} \arrow[]{rd}[swap]{l_f} \arrow[bend right =20]{rdd}[swap]{l' \in \mathcal{L'}} & & A \arrow[bend left =20 ]{ldd} \\ & L_f \arrow[]{ru}[swap]{r_f} \arrow[]{d}{l_f} & \\ & C & 
\end{tikzcd} \]
$l_f$ is in $\mathcal{L}'$ as soon as $ l'$ is.

\begin{definition}
In the spirit of \cite{Anel}, we define as $ \mathcal{L}'$-form of an object $C$ all $\mathcal{R}$-maps $ r : L \rightarrow C$ with a $\mathcal{L}'$-object as domain. In particular any arrow between $ \mathcal{L}'$-forms is in $\mathcal{R}$ and should be seen as an inclusion of fondamental neighborhood, or equivalently, as generalized specialization order.
\end{definition}

This means that any left form of $A$ induces a point of $\mathcal{R}(A)$: to any $ {r_x} : L_x {\rightarrow } A$ we can associate $ p_x = R(r_x) : 1 \rightarrow \mathcal{R}(A)$. Moreover, we could think this assignment as ``continuous" in the sense that any $ \mathcal{R}$-neighborhood $ R' \stackrel{r}{\rightarrow} R$ of $p_x$ induces a $\mathcal{R}$-neighborhood of $r_x$ by pullback

\[\begin{tikzcd}
	{L_x } \\
	& {r^*A} & {A} \\
	{1} \\
	& {R} & {\mathcal{R}(A)}
	\arrow[from=1-1, to=2-2]
	\arrow["{l_A^*r}", from=2-2, to=2-3, shorten <=3pt, shorten >=3pt]
	\arrow["{l_A}", from=2-3, to=4-3]
	\arrow[from=2-2, to=4-2]
	\arrow["{r}"', from=4-2, to=4-3]
	\arrow[from=3-1, to=4-2]
	\arrow["{!_{L_x}}"', from=1-1, to=3-1]
	\arrow["{p_x}" description, from=3-1, to=4-3, curve={height=-12pt}]
	\arrow["{x}", from=1-1, to=2-3, curve={height=-12pt}]
	\arrow["\lrcorner"{very near start, rotate=0}, from=2-2, to=4-3, phantom]
\end{tikzcd}\]

However in general, it is clear by functoriality of $\mathcal{R}(-)$ that the reflection of an object $A$ does not correctly distinguish local forms of $A$ as it collapses the specialization order when this one is non trivial: for a triangle of $\mathcal{R}$-maps
\[ \begin{tikzcd}
L_1 \arrow[bend left=10]{rrd}{r_1} \arrow[]{rd}[swap]{r} & & \\
& L_2 \arrow[]{r}{r_2} & A 
\end{tikzcd} \]
then $\mathcal{R}(r_1) = \mathcal{R}(r_2  r) = \mathcal{R}(r_2)  !_1$ so that both define the same point of $ \mathcal{R}(A)$. Intuitively, $\mathcal{R}$ collapses the local components and produces a disconnected set of points with trivial specialization order. \\


In this context, $ \mathcal{R}$-objects still play the role of ``discrete" objects and $ \mathcal{L}$-objects the role of connected object, while $\mathcal{L}'$-objects as local objects are to be seen as the ``fundamental neighborhood" of the geometry, while the $ \mathcal{R}$-maps codes its etale domain.\\

The $\mathcal{L}'$-maps from an arbitrary object toward an object will then be examples of ``bundles of $\mathcal{L}'$-objects" over a space, that is, bundles that are locally $\mathcal{L}'$-objects in the sense that their fibers are $\mathcal{L}'$-objects and $ \mathcal{L}'$-form of the domain of the bundle, from our key assumption that $ \mathcal{L}$-maps are closed by pullback:
\[\begin{tikzcd}
L_p \arrow[]{d}[swap]{p^*l_C \in \mathcal{L}} \arrow[phantom, very near start]{rd}{\lrcorner} \arrow[]{r}{l_C^*p \in \mathcal{R}} & C \arrow[]{d}{l_c} \\
1 \arrow[]{r}{p} & R
\end{tikzcd}\]

\begin{definition}
Define $ \mathcal{C}_{Spaces}$ as the category of ``$\mathcal{C}$-modelled spaces" whose \begin{itemize}
    \item objects are arbitrary maps $ f : C \rightarrow B $,
    \item and morphisms $f_1 \rightarrow f_2$ are triples $(g,a^\sharp, a_\sharp)$ of the form 

\[ \begin{tikzcd}
C_1 \arrow[bend left=12, end anchor=north west]{rrd}{a_\sharp} \arrow[bend right=20, end anchor=west]{rdd}[swap]{f_1} \arrow[]{rd}{a^\sharp} & &  \\
& g^*C_2 \arrow[]{r}{f_2^*g} \arrow[phantom, very near start]{rd}{\lrcorner} \arrow[]{d}[swap]{g^*f_2}  & C_2  \arrow[]{d}{f_2}  \\
& B_1 \arrow[]{r}{g} & B_2 
\end{tikzcd} \]
\end{itemize} 
\end{definition}
\begin{remark}
Actually, in the diagram above, $ a^\sharp$ and $ a_\sharp $ are obviously mutually determined, but this emphasizes the analogy with modelled toposes as they correspond to the dual maps of the inverse and direct comorphisms. Indeed, one would expect $f_1,f_2$ to code for sheaves of $ \mathcal{C}^{op}$ objects over $ \mathcal{R}$-spaces with $g$ a continuous map between them, along which one can either pull or push those sheaves. In this point of view, $g^*f_2 :  g^*C_2  \rightarrow B_1$ is the inverse image of $ f_2$ over $ \mathcal{B}_1$ and $ g \circ f_1$ really is the direct image $ r_*f_1  $ of $ f_1$ over $B_2$. 
\end{remark}

Now we turn to the objects we want to see as locally modelled objects:

\begin{definition}
For each object $B$, we define $\mathcal{L}oc^\mathcal{L}_{\mathcal{L}'}[B]$ as the category whose:\begin{itemize}
    \item objects are $ C \rightarrow B$ that are stalkwise in $\mathcal{L}'$
    \item maps are triangles that are stalkwise in $\mathcal{L}$\end{itemize}
Then we define the category $\mathcal{L}oc^\mathcal{L}_{\mathcal{L}'} $ as the wide subcategory of $\int \mathcal{L}oc^\mathcal{L}_{\mathcal{L}'}[-]$ whose arrows are restricted to the $(g,a^\sharp)$ with $a^\sharp$ stalkwise in $\mathcal{L}$. Observe that in particular this is a non full subcategory of $ \mathcal{C}-spaces$. 

\end{definition}
\begin{remark}
Observe we have to require the comorphisms to be locally in $\mathcal{L}$ as requiring it globally would not be sufficient from the non stability along pullback; actually, they do not even need to be so globally. This mimics the idea of local morphism, which returns local maps at stalks, in the usual spectral construction.
\end{remark}


In the following diagram

\[ \begin{tikzcd}
 C_1 \arrow[]{rr}{l \in \mathcal{L}} \arrow[]{rd}[swap]{f_1} & & C_2 \arrow[]{ld}{f_2} \\ &C & 
\end{tikzcd} \]
Then, if $ f_1$ is in $\mathcal{L}$-cancellation. If $ f_1$ locally is in $\mathcal{L}$, then by right cancellation of pullback square we know that for any point $p$, $ p^*a^\sharp$ is in $\mathcal{L}$ so that one has over 1

\[ \begin{tikzcd}
 p^*C_1 \arrow[]{rr}{p^*l \in \mathcal{L}} \arrow[]{rd}[swap]{p^*f_1 \in \mathcal{L}} & & p^*C_2 \arrow[]{ld}{p^*f_2} \\ & 1 & 
\end{tikzcd} \]
which forces each $p^*f_2$ to be in $\mathcal{L}$.\\

\begin{proposition}
One has a relatively full and faithful costable inclusion in each $ C$
\[ \mathcal{L}oc^\mathcal{L}_{\mathcal{L}'}[C] \hookrightarrow \mathcal{C}/C  \]
\end{proposition}

\begin{proof}
This just comes from the $\mathcal{L}$-absorption property of $\mathcal{L}'$ applied over 1 in each point of $\mathcal{R}$. The inclusion functor $ Loc^\mathcal{L}_ {\mathcal{L}'}[C]$ always is relatively full and faithful as indeed in the following situation
\[ \begin{tikzcd}
& A_2 \arrow[no head]{d}[near end]{ l_2 \in \mathcal{L}'} \arrow[]{rd}{f } & \\
A_1 \arrow[]{ru}{l} \arrow[]{rd}[swap]{l_1 \in \mathcal{L}'} \arrow[]{rr}[near start]{l' \in \mathcal{L}} & \arrow[]{d} & A_3 \arrow[]{ld}{l_3} \\ & C & 
\end{tikzcd}  \]
$ f$ is forced to be in $\mathcal{L}'$ by the absorption property of $\mathcal{L}'$ in $\mathcal{L}$ (this actually does not depends of $ l_1,l_2,l_3$ being in $\mathcal{L}'$). 

\end{proof}

\printbibliography

@article{Diers,
author= "Diers, Yves",
title= "Une construction universelle des spectres, topologies spectrales et faisceaux structuraux",
journaltitle ="Communication in Algebra",
year = "1984",
doi="10.1080/00927878408823101"
}

@article{Anel,
author= "Anel, Mathieu",
title="Grothendieck topologies from unique factorisation systems",
year = "2009",
url ="https://arxiv.org/abs/0902.1130"
}

@article{Taylor,
author = "Taylor, Paul",
title = "The Trace Factorisation of Stable Functors",
year = "1998",
url="http://www.paultaylor.eu/stable/trafsf.pdf"
}

@article{Cole,
author = "Cole, Julian",
title = "The bicategory of topoi and spectra",
journaltitle ="Reprints in Theory and Applications of Categories",
year="2016", 
pages=" 1-16",
url="http://www.tac.mta.ca/tac/reprints/articles/25/tr25abs.html"
}

@inbook{Coste,
author= "Coste, Michel",
title ="Localisation, spectra and sheaf representation",
booktitle=" Applications of Sheaves ",
pages="212-238 ",
editor="Spinger",
year= "1977"
}

@article{GARNER20121734,
title = "Ionads",
journal = "Journal of Pure and Applied Algebra",
volume = "216",
number = "8",
pages = "1734 - 1747",
year = "2012",
note = "Special Issue devoted to the International Conference in Category Theory `CT2010'",
issn = "0022-4049",
doi = "https://doi.org/10.1016/j.jpaa.2012.02.013",
url = "http://www.sciencedirect.com/science/article/pii/S0022404912000527",
author = "Richard Garner"
}

@misc{adámek2019nice,
      title={How nice are free completions of categories?}, 
      author={Jiří Adámek and Jiří Rosický},
      year={2019},
      eprint={1806.02524},
      archivePrefix={arXiv},
      primaryClass={math.CT}
}

@phdthesis{diers1977categories,
  title={Cat{\'e}gories localisables},
  author={Diers, Yves},
  year={1977},
  school={Paris 6 et Centre universitaire de Valenciennes et du Hainaut Cambr{\'e}sis}
}

@article{Diers-multipres,
  title={Cat{\'e}gories localement multipr{\'e}sentables},
  author={Diers, Yves},
  journal={Archiv der Mathematik},
  volume={34},
  number={1},
  pages={344--356},
  year={1980},
  publisher={Springer}
}

@article{DIA_1983__9__A1_0,
     author = {Lair, C.},
     title = {Sesqui-monades et monades locales},
     journal = {Diagrammes},
     publisher = {Universit\'e Paris 7, Unit\'e d'enseignement et de recherche de math\'ematiques},
     volume = {9},
     year = {1983},
     note = {talk:1},
     zbl = {0521.18012},
     mrnumber = {780082},
     language = {fr},
     url = {http://www.numdam.org/item/DIA_1983__9__A1_0}
}

@article{diers1980multimonads,
  title={Multimonads and multimonadic categories},
  author={Diers, Yves},
  journal={Journal of Pure and Applied Algebra},
  volume={17},
  number={2},
  pages={153--170},
  year={1980},
  publisher={North-Holland}
}

@article{hu1995limits,
  title={Limits in free coproduct completions},
  author={Hu, Hongde and Tholen, Walter},
  journal={Journal of pure and applied algebra},
  volume={105},
  number={3},
  pages={277--291},
  year={1995},
  publisher={Elsevier}
}

@inproceedings{maillard2015fibrational,
  title={A fibrational account of local states},
  author={Maillard, Kenji and Mellies, Paul-Andr{\'e}},
  booktitle={2015 30th Annual ACM/IEEE Symposium on Logic in Computer Science},
  pages={402--413},
  year={2015},
  organization={IEEE}
}

@article{weber2007familial,
  title={Familial 2-functors and parametric right adjoints},
  author={Weber, Mark},
  journal={Theory Appl. Categ},
  volume={18},
  number={22},
  pages={665--732},
  year={2007}
}

@article{weber2004generic,
  title={Generic morphisms, parametric representations and weakly Cartesian monads},
  author={Weber, Mark},
  journal={Theory Appl. Categ},
  volume={13},
  number={14},
  pages={191--234},
  year={2004}
}

@book{borceux1994handbook,
  title={Handbook of categorical algebra: volume 1, Basic category theory},
  author={Borceux, Francis},
  volume={1},
  year={1994},
  publisher={Cambridge University Press}
}

@article{tholen_1984, title={Pro-Categories and Multiadjoint Functors}, volume={36}, DOI={10.4153/CJM-1984-010-2}, number={1}, journal={Canadian Journal of Mathematics}, publisher={Cambridge University Press}, author={Tholen, Walter}, year={1984}, pages={144–155}}

@article{lurie2009derived,
  title={Derived algebraic geometry V: Structured spaces},
  author={Lurie, Jacob},
  journal={arXiv preprint arXiv:0905.0459},
  year={2009}
}

@article{partII,
  title={On Diers theory of Spectra II: Geometries and Dualities},
  author={Osmond, Axel},
  journal={},
  year={2020}
}

@article{survey,
  title={The general construction of spectra},
  author={Osmond, Axel},
  journal={},
  year={2020}
}

@article{diers1980quelques,
  title={Quelques constructions de cat{\'e}gories localement multipr{\'e}sentables},
  author={Diers, Yves},
  journal={Ann. SC. math. Qu{\'e}bec},
  volume={4},
  number={2},
  pages={79401},
  year={1980}
}

@book{adamek1994locally,
  title={Locally presentable and accessible categories},
  author={Ad{a}mek, Ji{\v{r}}{i} and Adamek, J and Rosicky, J and others},
  volume={189},
  year={1994},
  publisher={Cambridge University Press}
}

@article{hebert2001more,
  title={More on orthogonality in locally presentable categories},
  author={H{e}bert, M and Ad{a}mek, J and Rosick{y}, Ji{\v{r}}{i}},
  journal={Cahiers de topologie et g{\'e}om{\'e}trie diff{\'e}rentielle cat{\'e}goriques},
  volume={42},
  number={1},
  pages={51--80},
  year={2001}
}

@article{TholenRosickySolutionSet,
  title={Accessibility and the solution set condition},
  author={Rosick{y}, Ji{\v{r}}{i} and Tholen, Walter},
  journal={Journal of Pure and Applied Algebra},
  volume={98},
  number={2},
  pages={189--208},
  year={1995},
  publisher={Elsevier}
}

@book{makkai1989accessible,
  title={Accessible Categories: The Foundations of Categorical Model Theory: The Foundations of Categorical Model Theory},
  author={Makkai, Michael and Par{\'e}, Robert},
  volume={104},
  year={1989},
  publisher={American Mathematical Soc.}
}

\end{document}